\documentclass[12pt]{amsart}
\usepackage[letterpaper]{geometry}
\usepackage{amsmath,amsfonts,amsthm,amscd,amssymb,mathrsfs,amssymb, multirow, pst-node, comment,enumitem}
\usepackage{stmaryrd}
\usepackage{graphicx}
\usepackage{verbatim}
\usepackage{mathtools}
\usepackage{tikz}
\usepackage{tikz-cd} 
\newtheorem{theorem}{Theorem}

\newtheorem{proposition}[theorem]{Proposition}
\newtheorem{lemma}[theorem]{Lemma}

\newtheorem{corollary}[theorem]{Corollary}

\numberwithin{figure}{section}
\usepackage[colorlinks, linkcolor=blue, anchorcolor=black, citecolor=red]{hyperref}
\theoremstyle{definition}
\newtheorem{definition}[theorem]{Definition}
\newtheorem{remark}[theorem]{Remark}

\newcommand{\la}{\langle}
\newcommand{\ra}{\rangle}

\newcommand{\mcH}{\mathcal{H}}
\newcommand{\mcX}{\mathcal{X}}
\newcommand{\omcX}{\overline{\mathcal{X}}}

\newcommand{\mcG}{\mathcal{G}}
\newcommand{\mcK}{\mathcal{K}}
\newcommand{\mcL}{\mathcal{L}}
\newcommand{\omcL}{\overline{\mathcal{L}}}

\newcommand{\mcJ}{\mathcal{J}}
\newcommand{\mcE}{\mathcal{E}}
\newcommand{\mcM}{\mathcal{M}}

\newcommand{\PP}{\mathbb{P}}
\newcommand{\CC}{\mathbb{C}}
\newcommand{\NN}{\mathbb{N}}
\newcommand{\RR}{\mathbb{R}}
\newcommand{\ZZ}{\mathbb{Z}}
\newcommand{\QQ}{\mathbb{Q}}

\newcommand{\TT}{\mathbb{T}}
\newcommand{\bS}{\mathbb{S}}

\newcommand{\NA}{{\rm NA}}
\newcommand{\MA}{{\rm MA}}

\newcommand{\lc}{{\rm lc}}

\newcommand{\bfE}{{\bf E}}
\newcommand{\bfJ}{{\bf J}}
\newcommand{\bfL}{{\bf L}}
\newcommand{\bfD}{{\bf D}}
\newcommand{\bfM}{{\bf M}}
\newcommand{\bfI}{{\bf I}}
\newcommand{\bfH}{{\bf H}}
\newcommand{\bfF}{{\bf F}}
\newcommand{\Fut}{{\bf Fut}}

\renewcommand{\DH}{{\bf DH}}

\newcommand{\ddc}{{\rm dd^c}}
\renewcommand{\div}{{\rm div}}
\newcommand{\vk}{{ \vec{k}}}

\newcommand{\bvk}{{[\vec{k}]}}
\renewcommand{\max}{{\rm max}}
\newcommand{\FS}{{\rm FS}}
\newcommand{\Bl}{{\rm Bl}}

\newcommand{\eq}{{\rm eq}}

\newcommand{\reff}{{\rm ref}}
\newcommand{\tX}{{\tilde{X}}}
\newcommand{\llg}{{\rm log}}

\newcommand{\w}{\omega}

\newcommand{\oo}{\mathcal{O}}

\newcommand*{\sheafhom}{\mathscr{H}\kern -.5pt om}
\newcommand*{\sheafext}{\mathscr{E}\kern -.5pt xt}

\newcommand{\bC}{\mathbb{C}}
\newcommand{\Aut}{{\rm Aut}}
\newcommand{\vphi}{\varphi}
\newcommand{\vtheta}{\vartheta}
\newcommand{\Lam}{{\bf \Lambda}}
\newcommand{\bP}{\mathbb{P}}
\newcommand{\bR}{\mathbb{R}}
\newcommand{\bT}{\mathbb{T}}
\newcommand{\bQ}{\mathbb{Q}}
\newcommand{\bN}{\mathbb{N}}
\newcommand{\bfm}{{\bf m}}
\newcommand{\bX}{\mathbb{X}}

\newcommand{\bG}{\mathbb{G}}
\newcommand{\triv}{{\rm triv}}
\newcommand{\mcF}{\mathcal{F}}
\newcommand{\cF}{\mathcal{F}}
\newcommand{\vol}{{\rm vol}}
\newcommand{\DHM}{{\bf DH}}
\newcommand{\mcD}{\mathcal{D}}

\newcommand{\ac}{{\rm ac}}

\newcommand{\bV}{\mathbb{V}}
\newcommand{\bZ}{\mathbb{Z}}
\newcommand{\mcQ}{\mathcal{Q}}
\newcommand{\fa}{\mathfrak{a}}
\newcommand{\fr}{\mathfrak{r}}
\newcommand{\bB}{\mathbb{B}}

\numberwithin{equation}{section}
\numberwithin{theorem}{section}
\numberwithin{table}{section}
\numberwithin{table}{section}

\begin{document}
\bibliographystyle{amsalpha}
\title[Yau-Tian-Donaldson conjecture for generalized K\"{a}hler-Ricci solitons]{On the Yau-Tian-Donaldson conjecture for generalized K\"{a}hler-Ricci soliton equations}
\author{Jiyuan Han}
\address{Department of Mathematics, Purdue University, West Lafayette,
IN, 47907}
\email{han556@purdue.edu}
\author{Chi Li}
\address{Department of Mathematics, Purdue University, West Lafayette,
IN, 47907}
\email{li2285@purdue.edu}
\thanks{}
\date{\today}
\begin{abstract}
Let $(X, D)$ be a polarized log variety with an effective holomorphic torus action, and $\Theta$ be a closed positive torus invariant $(1,1)$-current. 
For any smooth positive function $g$ defined on the moment polytope of the torus action, we study the Monge-Amp\`{e}re equations
that correspond to generalized and twisted K\"{a}hler-Ricci $g$-solitons. 
We prove a version of Yau-Tian-Donaldson (YTD) conjecture for these general equations, showing that the existence of solutions is always equivalent to an equivariantly uniform $\Theta$-twisted $g$-Ding-stability. When $\Theta$ is a current associated to a torus invariant linear system, we further show 
that equivariant special test configurations suffice for testing the stability. 
Our results allow arbitrary klt singularities and generalize most of previous results on (uniform) YTD conjecture for (twisted) K\"{a}hler-Ricci/Mabuchi solitons or K\"{a}hler-Einstein metrics.
\end{abstract}
\maketitle
\setcounter{tocdepth}{1}
\tableofcontents
\section{Introduction and main results}
\label{intro}


In this paper, we will use the variational approach to study generalized and twisted K\"{a}hler-Ricci soliton equations on log Fano varieties, which generalize the usual (twisted) K\"{a}hler-Ricci soliton equations on Fano manifolds which we first recall. Let $X$ be a smooth Fano manifold and let $\Theta$ be a closed and positive $(1,1)$-current. By a twisted K\"{a}hler-Ricci soliton (see \cite{TZ99, Zhu00, DaS16}), we mean a Hermitian metric $e^{-\vphi}$ on $-K_X$ such that $\ddc\vphi=\frac{\sqrt{-1}}{2\pi}\partial\bar{\partial}\vphi$ is a K\"{a}hler metric and satisfies the equation:
\begin{equation}\label{eq-KRsm}
Ric(\ddc\vphi)=\ddc\vphi+\mathfrak{L}_\xi \ddc\vphi+\Theta,
\end{equation}
where $\xi$ is a holomorphic vector field and $\mathfrak{L}_\xi$ denotes the Lie derivative. Based on the works in \cite{CTZ05, DR17}, we know that the solvability of this equation is equivalent to appropriate coercivity condition of associated energy functionals. 
It is also possible to show the corresponding version of Yau-Tian-Donaldson conjecture which states that the solvability is equivalent to appropriate (twisted) K/Ding-stability of the data $(X, \xi, \Theta)$. By now there is a long list of such results under different conditions of $(X, \xi, \Theta)$. We refer to \cite{Tia97,Berm15,CDS15,Tia15,CSW18,BBJ18,His19,Li19} for the K\"{a}hler-Einstein case (i.e. $\xi=\Theta=0$), \cite{Der16, DaS16,BBJ18,FS19,RS19,LTW19,TW19} for the twisted/conical K\"{a}hler-Einstein case (i.e. $\xi=0$) and to \cite{BW14, DaS16} for the general K\"{a}hler-Ricci soliton case.  
 Moreover, recently there have been many parallel studies on Mabuchi-soliton metrics (see \cite{Mab01, Mab03, Yao17, Yao19, His19, LZ19} and references therein). Recall that a K\"{a}hler metric $\ddc\vphi\in c_1(X)$ is a Mabuchi soliton (\cite{Mab01, Mab03}) on $(X, \xi)$ (with $\xi$ holomorphic) if it satisfies the equation: 
\begin{equation}\label{eq-Msoliton}
\bar{\partial}\left( \frac{e^{-\vphi}}{(\ddc\vphi)^n}\right)=\iota_\xi (\ddc\vphi).
\end{equation}

Our main purpose in this paper is to generalize these results to a much more general setting. First, we will work with any log pair $(X, D)$ allowing arbitrary (klt) singularities. Second, we will use the set-up of Berman-Witt-Nystr\"{o}m in \cite{BW14} by considering more general complex Monge-Amp\`{e}re equation of the following form:
\begin{equation}\label{eq-KRT1}
g(\bfm_\vphi)\frac{(\ddc\vphi)^n}{n!}=e^{-\vphi-\psi},
\end{equation}
where $\bfm_\vphi$ is the moment map associated to an effective torus action, $g$ is any positive smooth function defined on the associated moment polytope $\bfm_\vphi(X)$. 
We will explain more of notations shortly. As is well known, \eqref{eq-KRT1} reduces to \eqref{eq-KRsm} when $\log g$ is affine, $(X, D)=(X, \emptyset)$ and $\Theta=\ddc \psi$. In particular, when $g$ is a positive constant, we get the twisted K\"{a}hler-Einstein equation. Moreover, when $g$ is affine, then we get the Mabuchi-soliton equation \eqref{eq-Msoliton}.

Our method is based on the variational approach, which applies well to both smooth and singular varieties. Indeed, the variational approach to solve complex Monge-Amp\`{e}re equations on possibly singular varieties were developed in recent years, especially in the works of \cite{BBGZ, BBEGZ, BW14}. Based such variational approach and the study of the space of K\"{a}hler metrics, Tian's properness conjecture has been resolved in \cite{DR17, Dar17, DNG18}. Moreover the variation approach to YTD conjecture via pluripotential theory and non-Archimedean geometry have been successfully carried out in \cite{BBJ18, LTW19, His19, Li19}. 

As we will show in this paper, by working harder to generalize the techniques from these works, we can indeed achieve a version of the Yau-Tian-Donaldson conjecture for more general equation \eqref{eq-KRT1} which gives sufficient and necessary algebrao-geometric conditions for the existence of solutions.
We will also show that the MMP process developed in \cite{LX14, BBJ15, Fuj19a, BLZ19} works equally well for the more general equation \eqref{eq-KRT1}, which allows us to test generalized-twisted K/Ding-stability using only special test configurations. 

To state our results, we introduce some notations. 
Let $X$ be a projective variety of dimension $n$, $D$ an effective divisor such that $K_X+D$ is $\QQ$-Cartier.
Let $\Theta$ be a closed positive $(1,1)$-current which is the curvature of  a possibly singular Hermitian metric $e^{-\psi}$ on a $\QQ$-line bundle $B$. We assume that $L = -(K_X + D)-B$ is ample.
\begin{remark}
By incorporating the divisorial part of $\Theta$ into $D$, we can usually assume that in the Siu-decomposition of $\Theta$, there is no divisorial part.
\end{remark}
Let $\omega_0$ be a smooth K\"ahler metric on $L$, as the curvature of a Hermitian metric $e^{-\vphi_0}$ on $L$, i.e. $\omega_0=\ddc\vphi_0$. 
Then $e^{-\vphi_0-\psi}$ is a Hermitian metric on $-K_X-D$ and we get a globally defined measure on $X$ (see \cite{BBEGZ}):
\begin{equation}
d\mu_0=e^{-\vphi_0-\psi}=|s|^{2/m}  e^{-\vphi_0-\psi} (\sqrt{-1}^{mn^2}s^*\wedge \bar{s}^*)^{1/m},
\end{equation}
where $s$ is any no-where zero local holomorphic section of $m(-K_X-D)$ over an open set for some sufficiently divisible $m$, and $s^*$ is the dual to $s$.

Let $T\cong (S^1)^r$ be a real torus of dimension $r$ with the complexification $T_\bC\cong (\bC^*)^r$. We assume that $T_\bC$ acts effectively and holomorphically on $X$, preserving the divisor $D$. Moreover, we assume the $T_\bC$ action lifts to an action on the $\bQ$-line bundle $B$, which means it lifts to act on the line bundle $mB$ for some $m\in \bN$, and $e^{-\psi}$ is a $T$-invariant Hermitian metric so that $\Theta$ is also $T$-invariant. Note that since $T_\bC$ naturally acts on $K_X+D$, it is equivalent to require that $T_\bC$ lifts to an action on the ample $\bQ$-line bundle $L=-(K_X+D)-B$.

For any K\"{a}hler form $\ddc\vphi \in c_1(L)$, there is an associated moment map
\begin{align}
\bfm_\vphi: X \rightarrow \RR^r.
\end{align}

Let $P\subset \RR^r$ be the image of $\bfm_\vphi$, which is known to be a convex polytope and independent of the choice of $\ddc\vphi$.
Let $g$ be any smooth positive function defined on $P$. For any $x\in X$, for simplicity, we set $g_\vphi(x):= g(\bfm_\vphi(x))$. Then we can re-write the equation \eqref{eq-KRT1} as 
\begin{equation}
\label{MA_equation}
g_\vphi  \frac{(\ddc\vphi)^n}{n!} = {e^{-(\vphi-\vphi_0)}d\mu_0}.
\end{equation}
By \cite{BW14} (see section \ref{functionals} for an alternative way of definition), for any $\vphi\in \mcE^1_T(X, L)$, the left-hand-side of \eqref{MA_equation} is a well-defined Radon measure and the equation can be considered in the pluripotential sense. 
When $\vphi$ is sufficiently smooth, one can easily see that, over $X^{\rm sm}$, the equation \eqref{MA_equation} is equivalent to
\begin{equation}
Ric_\vphi:=Ric(\ddc\vphi) = \ddc\vphi + [D] + \Theta + \ddc \log g_\vphi,
\end{equation}
where $[D]$ is the current of integration along the divisor $D^{\rm sm}$. 
\begin{definition}
We say $\bX:=(X, D+\Theta, T)$ admits a K\"{a}hler-Ricci $g$-soliton (or simply KR $g$-soliton) if there exists a solution $\vphi$ in the finite energy space $\mcE^1_T(X, L)$ (see Definition \ref{eq-cE1T}) to the equation \eqref{MA_equation}.
\end{definition}



In order to get a good existence theory using variational approach and pluripotential theory, we need to require some control of the measure $d\mu_0$. To express this, we choose a $T$-equivariant log resolution of singularities $\rho: \tilde{X}\rightarrow X$ such that $\rho^{-1}(D+X^{\rm sing})$ is supported on a simple normal crossing divisor. We can then write down the ramification formula:
\begin{equation}
K_Y + D' = \rho^*(K_X+D)+\sum_i a_i E_i,
\end{equation}
where $D'=\rho_*^{-1}D$ is the strict transform of $D$ and $E_i$ are exceptional divisors of $\rho$. Choose a smooth volume form $\Omega$ on $\tilde{X}$. Then for any local holomorphic chart $\{U_\alpha, z_j\}$, there exists nowhere zero smooth functions $f(z)$ and local psh function $\tilde{\psi}_\alpha=\rho^*\psi$, such that 
\begin{equation}
\rho^* d\mu_0=\prod_{i} |z_i|^{2 a_i} e^{-\tilde{\psi}_\alpha} f(z) \Omega.
\end{equation}
Following \cite{BBEGZ, BBJ18}, we define:
\begin{definition}
We say that $(X, D+\Theta)$ is klt if the in the above representation $\prod_i |z_i|^{2a_i}e^{-\psi_\alpha}\in L^p(U_\alpha, \Omega)$ for some $p>1$.
\end{definition}

By slightly generalizing the previous work, we show that solutions to \eqref{MA_equation} are critical points of two functionals $\bfD=\bfD_{g,\Theta}$ and $\bfM=\bfM_{g,\Theta}$ (see \eqref{eq-Ding} and \eqref{eq-Mabuchi})
defined on the space $\mcE^1_T(X, L)$ of $T$-invariant finite energy potentials. 
As in the usual K\"{a}hler-Einstein case, we first derive an analytic criterion for the existence of solutions to equation \eqref{MA_equation}.




Let $\Aut(X, D)$ be the automorphism of $(X, D)$ and $\mathfrak{aut}(X, D)$ be its Lie algebra. Let $\mathfrak{t}$ (resp. $\mathfrak{t}_\bC$) be the Lie algebra of $T$ (resp. $T_\bC$)
\begin{definition}\label{def-AutT}
Let $(X, D, \Theta)$ be the same as above. We define:
\begin{equation}\label{eq-autT}
\mathfrak{aut}_T(X, D, \Theta)=\left\{\xi\in \mathfrak{aut}(X, D); \iota_\xi \Theta=0, [\xi, c]=0, \forall c\in \mathfrak{t}_\bC\right\}. 
\end{equation}
Let $\Aut_T(X, D, \Theta)$ be the connected subgroup of $\Aut(X, D)$ generated by $\mathfrak{aut}_T(X, D, \Theta)$. 
\end{definition}
\begin{remark}
If $\Theta=0$, i.e. $[\Theta]=0$, then $T_\bC\subseteq \Aut_T(X, D):=\Aut_T(X, D, 0)$. If $\Theta$ is K\"{a}hler then $\Aut_T(X, D, \Theta)$ is trivial.
\end{remark}

Then we have:
\begin{theorem}
\label{theorem_A}
Let $\bX:=(X,D+\Theta, T)$ be the data as specified above.
Assume that $\bG$ is a reductive subgroup of $\Aut_T(X, D, \Theta)$. 
If $\bfD$ is $\bG$-coercive over $\mcE^1_{T\times K}$ (see Definition \ref{def-Gcoer}), then $\bX$ admits a KR $g$-soliton.

Conversely, if $\bX$ admits a KR $g$-soliton, then $\Aut_T(X, D, \Theta)$ is reductive, and for any reductive subgroup $\bG$ of $\Aut_T(X, D, \Theta)$ that contains a maximal torus of $\Aut_T(X, D, \Theta)$, $\bfD$ is $\bG$-coercive over $\mcE^1_{T\times K}$.
\end{theorem}

Our second main result is on the Yau-Tian-Donaldson type conjecture for the general equation \eqref{MA_equation}.
\begin{theorem}
\label{thm-YTD}
Let $\bX$ be as above and $\bG$ be a reductive subgroup of $\Aut_T(X, D, \Theta)$. 

If $\bX$ is $\bG$-uniformly $g$-Ding-stable over $(T_\bC\times \bG)$-equivariant test configurations, then $\bX$ admits a KR $g$-soliton. 

Conversely if $\bX$  admits a KR $g$-soliton and $\bG$ contains a maximal torus of $\Aut_T(X, D, \Theta)$, then $\bX$ is $\bG$-uniformly $g$-Ding-stable over $(T_\bC\times\bG)$-equivariant test configurations. 

Moreover, when $\Theta$ is a generic current associated to a $T$-invariant linear system (see Lemma \eqref{lem-character}), it is enough to test the $\bG$-uniform $g$-Ding-stability over $(T_\bC\times\bG)$-equivariant special test configurations.

\end{theorem}

\begin{remark}\label{rem-Wang}
{
We point out that the above results work for all log varieties with klt singularities. When $X$ and $\Theta$ are smooth, it might be possible to generalize the metric techniques (including Cheeger-Colding-Tian theory and partial $C^0$-estimates) in \cite{WZ13, JWZ17, Sze16, DaS16, TW19, LS18, RS19} to prove similar results, although F. Wang pointed out to us some difficulty in such approach for non-KR soliton cases.

It is standard to get more regularity for the solution to \eqref{MA_equation} (see Proposition \ref{higher_regularity}). On the other hand, if we just consider weak finite energy solutions, it is possible to allow $g\in C^0(P)$ for Theorem \ref{theorem_A}, and Theorem \ref{thm-YTD} (at least when $X$ is smooth) to hold true.
}
\end{remark}
As mentioned above, the proof of our results depend on generalizing and unifying the techniques in previous works. There are however various new ingredients in our arguments. For example, we get the generalized Mabuchi functional and its generalized Chen-Tian formula for {\it all} smooth $g$ and for {\it all} K\"{a}hler classes (see \ref{Mabuchi_alternative}).
Moreover we would like to emphasize two key observations that make our arguments to work successfully as in the usual K\"{a}hler-Einstein case. The first is that the difficulty in dealing the $g$-soliton equation and the difficulty caused by the twisting $D+\Theta$ are de-coupled, at least when we use the Ding-functional/stability. This could be seen from the formula for the Archimedean/non-Archimedean Ding functional (see \eqref{eq-Ding} and \eqref{eq-DNA}). 
The second is that the well-known construction in the study of equivariant de-Rham cohomology together with the Stone-Weierstrass theorem can be applied in our calculations for various Archimedean/non-Archimedean functionals.

We end this introduction with an outline of the paper. In the next section, we will write down various functionals needed in the variational study of the equation \eqref{MA_equation}, and study their relations. In section \ref{existence_and_properness}, we prove the analytic criterion for the existence of solutions to \eqref{MA_equation}. In section \ref{generalized-Mabuchi}, we prove the technical result that the generalized Mabuchi functional is convex along geodesic rays in the space of $T$-invariant metrics. In section \ref{non-Archimedean functionals}, we introduce various non-Archimedean functionals and show the valuative criterion for the $\bG$-uniform stability. In section \ref{stable_and_proper}, we prove the Yau-Tian-Donaldson conjecture for the equation \eqref{MA_equation}. In section \ref{sec-MMP}, we point out that the stability can be tested over special test configurations when $\Theta$ is associated from a $T$-invariant linear subsystem.
In section \ref{sec-exmp}, we point out some immediate examples for which our results apply.

\subsection{Acknowledgements} 
J. Han would like to thank Jeff Viaclovsky for his teaching and support over many years.
C. Li is partially supported by NSF (Grant No. DMS-1810867) and an Alfred P. Sloan research fellowship, and he would like to thank Yuchen Liu for helpful discussions about equivariant K-stability. We would like to thank Feng Wang and Bin Zhou for helpful comments.
We would like to thank E. Inoue for pointing to us his related work \cite{Ino19b} and the earlier work of A. Lahdili \cite{Lahdili19}, and T. Delcroix for bringing the work in \cite{DeHu18} to our attention.

\section{Functionals and their dualities} 
\label{functionals}
In this section, we will first review a result in \cite{BW14} which defines the Radon measure $g_\vphi (\ddc\vphi)^n$ for all
$T$-invariant $L$-psh Hermitian metrics. We will give an alternative definition of these measures which are more adapted to our discussion. 
After that, we will define the functional $\bfE_g, \bfJ_g, \bfI_g$ which are generalized versions of  the usual  functionals $\bfE,\bfJ,\bfI$.
Then we will define the generalized Ding functional $\bfD$ and generalized Mabuchi functional $\bfM$ by essentially generalizing some definitions from 
\cite{TZ02}. We will show that $\bfD$ and $\bfM$ are convex along weak geodesics. At last, we will show that the duality formalism in \cite{BBGZ,BBEGZ} also apply to our functionals here.

\subsection{$g$-Monge-Amp\`{e}re measure}\label{sec-gMA}
In this subsection, we will work on manifold $\tX$ together with a smooth semi-positive closed (1,1)-form $\omega_0=\ddc\vphi_0=\frac{\sqrt{-1}}{2\pi}\partial\bar{\partial}\vphi_0$.

Define the space of $T$-invariant psh metrics on $L$ as:
\begin{equation}
{\rm Psh}_T(\tX, L)=\{\vphi=\vphi_0+u; u\in L^1_{\rm loc}, \text{and locally $\vphi$ is a $T$-invariant plurisubharmonic function} \}
\end{equation}
and smooth $T$-invariant {\it semipositive} potentials as:
\begin{equation}
\mcH_T(\tX, L) = \{ \vphi=\vphi_0+u \in {\rm Psh}_T(\tilde{X}, L);  u\in C^\infty(\tX, \bR) \}.
\end{equation}
We will use
$\omega = \omega_0 = \ddc\vphi_0=\frac{\sqrt{-1}}{2\pi}\partial\bar{\partial}\vphi_0$, $\ddc\vphi = \ddc(\vphi_0 + u)$.

For each $1\le\alpha\le r$, assume that $\xi_\alpha$ is the holomorphic $(1, 0)$-vector field generating the action of the $\alpha$-th factor of $T_\bC\cong (\bC^*)^r$ and let $\theta_\alpha=\theta_\alpha(\vphi)\in C^\infty(\tilde{X}, \bR)$ be the associated Hamiltonian function for any $\ddc\vphi$. Then we have:
\begin{equation}
\iota_{\xi_\alpha}\ddc\vphi=\frac{\sqrt{-1}}{2\pi} \bar{\partial} \theta_\alpha(\vphi).
\end{equation}
We have the formula $\theta_\alpha(\vphi)=\theta_\alpha(\vphi_0)+\xi_\alpha(\vphi-\vphi_0)$ and the moment map $\bfm_\vphi: \tilde{X}\rightarrow \bR^r$ and moment polytope are then given by:
\begin{equation}
\bfm_\vphi(x)=(\theta_1(\vphi)(x), \dots, \theta_{r}(\vphi)(x)), \quad P=\bfm_\vphi(\tilde{X}).
\end{equation} 
For any smooth function $g$ on $P$, set 
\begin{equation}
g_\vphi=g(\theta_{1}(\vphi), \dots, \theta_{r}(\vphi)).
\end{equation}
When $g>0$, we will often use the function:
\begin{equation}\label{eq-etavphi}
f=\log g \text{ (defined on P)}, \quad\quad  f_\vphi =\log g_\vphi \text{ (defined on $\tilde{X}$)}.
\end{equation}
So we have the formula:
\begin{eqnarray*}
\frac{\sqrt{-1}}{2\pi}\bar{\partial}g_\vphi=\frac{\sqrt{-1}}{2\pi}\sum_{\alpha}\frac{\partial g}{\partial \theta_\alpha}\bar{\partial}\theta_\alpha
=\sum_\alpha g_\alpha \iota_{\xi_\alpha}(\ddc\vphi)=\iota_{V_{g,\vphi}}\ddc\vphi,
\end{eqnarray*}
where we denoted 
\begin{equation}
g_\alpha=D_\alpha g=\frac{\partial g}{\partial \theta_\alpha}, \quad
V_{g,\vphi}:=\sum_\alpha g_\alpha(\bfm_\vphi(x)) \cdot \xi_\alpha.
\end{equation}
Note that $V_{g,\vphi}$ is {\it not} necessarily holomorphic. 
Similarly, for the smooth function $f=\log g$ on $P$, we set $f_\vphi=f(\bfm_\vphi)$ and 
\begin{equation}\label{eq-Vfvphi}
V_{f,\vphi}=\sum_\alpha \frac{\partial f}{\partial \theta_\alpha} \cdot \xi_\alpha=\sum_\alpha f_\alpha \xi_\alpha
\end{equation} 
which satisfy the identity:
\begin{equation}\label{eq-HamVfvphi}
\frac{\sqrt{-1}}{2\pi}\bar{\partial}f_\vphi=\iota_{V_{f,\vphi}}\ddc\vphi.
\end{equation}

It is well-known that both the moment polytope $P={\rm Im}(\bfm_\vphi)$ and the Duistermaat-Heckmann measure $\DHM_T(X, L)=(\bfm_\vphi)_*\frac{(\ddc\vphi)^n}{n!}$ are independent of the choice of $\vphi$ (see \cite[Proposition 4.1]{BW14}). As a consequence, the following integral gives the weighted volume that we will use in the following discussion:
\begin{equation}\label{eq-Vg}
\int_X g_\vphi \frac{(\ddc\vphi)^n}{n!}=\int_P g(y) \DHM_T(X, L):=\bV_g.
\end{equation}
Note that $\bV_1=\frac{c_1(L)^{\cdot n}}{n!}$.
For general $g>0$ on $P$, by multiplying $g$ by a positive constant, we can usually assume the normalization to simplify the notations in our following discussion:
\begin{equation}
\label{normalization}
\bV_g=1, \quad \int_X d\mu_0 = 1.
\end{equation}

For any interval $I \subset \bR$, set $\Sigma_I=\{s+\sqrt{-1}b; s\in I, b\in \bR\}\subset \bC$. 
\begin{lemma}\label{lem-etaPhi}
Let $I\subset \bR$ be any interval and $\Phi=\{\vphi(s)\}: s\in I \rightarrow \vphi(s)\in \mcH_T(\tX, L)$ be a differentiable map. Consider $\Phi$ as a Hermitian metric on $p_2^*L$ where $p_2: \Sigma_I\times \tilde{X}\rightarrow \tilde{X}$ is the natural projection. 
For any smooth function $f$ on $P$, we have the identity:
\begin{equation}
\iota_{V_{f,\vphi}}(\ddc\Phi)=\frac{\sqrt{-1}}{2\pi}\bar{\partial}f_\Phi,
\end{equation}
where $f_\Phi$ is a function on $\Sigma_I\times \tilde{X}$ defined as:
\begin{equation}
f_\Phi(s+\sqrt{-1} b, x)=f_{\vphi(s)}(x).
\end{equation}
Moreover we have the identity:
\begin{equation}
\frac{d}{ds}f_\vphi=V_{f,\vphi}\left(\frac{\partial\vphi}{\partial s}\right).
\end{equation}

\end{lemma}
\begin{proof}
Recall that we have $f_{\vphi(s)}(x)=f(\theta_\alpha(\vphi))$ and $\theta_\alpha(\vphi)=\theta_\alpha(\vphi_0)+\xi_\alpha(\vphi-\vphi_0)$. Set $\dot{\vphi}=\frac{\partial \vphi}{\partial s}$ and $t=s+\sqrt{-1}b$. We get:
\begin{eqnarray*}
\frac{\sqrt{-1}}{2\pi}\bar{\partial}f_\Phi&=&\frac{\sqrt{-1}}{2\pi}\sum_\alpha f_\alpha \left(\bar{\partial}_{\tilde{X}}\theta_\alpha(\vphi)+\xi_\alpha(\dot{\vphi})d \bar{t}\right)\\
&=&\sum_\alpha f_\alpha \iota_{\xi_\alpha}\left(\ddc \vphi+\frac{\sqrt{-1}}{2\pi}\partial{\dot{\vphi}}\wedge d\bar{t}\right)\\
&=&\sum_\alpha f_\alpha \iota_{\xi_\alpha} (\ddc\Phi)=\iota_{V_{f,\vphi}}\ddc\Phi.
\end{eqnarray*}
\end{proof}
We have the following result proved by Berman-Witt-Nystr\"{o}m.
\begin{proposition}[\cite{BW14}]
\label{MA_eta}
Let $g=g(y)$ be a continuous positive function on the moment polytope $P$.
For any $\vphi \in {\rm Psh}_T(\tX, L)$, 
$\MA_g(\vphi) := g(\bfm_\vphi) \frac{(\ddc\vphi)^n}{n!}$
is a well-defined a Radon measure.
\end{proposition}
The proof in \cite{BW14} uses Kiselman's partial Legendre transform and approximation by series of step functions.
We now present an alternative construction of the above measure by using a more direct fibration construction, which is more adapted to our following discussion, in particular to defining the non-Archimedean functions in Section \ref{non-Archimedean functionals}.
\begin{proof}
By Stone-Weierstrauss theorem, the continuous function $g $ on $P\subset \RR^r$ can be approximated uniformly by polynomials
$g_i$ such that $|g (y)-g_i(y)|\to 0$ for any $y = (y_1,\cdots,y_r)\in P$,  as $i\to\infty$.
For $u\in {\rm Psh}_T(\tX,\omega_0)$, set $u_M = \max(u, -M)$. Set $\vphi_M = \vphi_0+u_M$.
For each $\vphi_M$, there exists a decreasing sequence $u_{M,j}\in \mcH_T(\tX,\omega_0)$, 
such that $u_{M,j} \searrow u_M$ in the weak topology as $j\to\infty$.
We will first define $\MA_{g_i}(\vphi_{M,j})$, and then show that $\MA_{g_i}(\vphi_{M,j})$ converges as $j\to\infty, M\to\infty,i\to\infty$, and
define $\MA_g(\vphi)$ as the limit.

Assume that $\xi$ generates an effective $T_\bC$-action where $T_\bC\simeq (\CC^*)^r$. 
Set $N^T_\ZZ={\rm Hom}(\CC^*, T_\bC)$, $M^T_\ZZ={\rm Hom}(T_\bC, \CC^*)$, $N^T_\RR=N^T_\ZZ\otimes_\ZZ \RR$, $M^T_\RR=M^T_\ZZ\otimes_\ZZ \RR$. 
Fix a $\ZZ$-basis $\{e_1,\dots, e_r\}$ of $N^T_\ZZ$ and write $\xi=\sum_{i=1}^r a_i e_i$ such that
$\{a_i; i=1,\dots, r\}$ are $\QQ$-linearly independent.

First we will assume $g $ is a polynomial and show that $\MA_{g}$ is a well-defined Radon measure.
Without the loss of generality, we can furthermore  assume $g $ is a monomial $\prod_{\alpha=1}^r y_\alpha^{k_\alpha}$, $k_\alpha\in \NN,
y\in P$. We will also assume $\vphi$ is smooth at first, and then work on the general $\vphi$ by using the approximation as stated above.
When $\vphi$ is smooth, the moment map $\bfm_\vphi: X\rightarrow P\subset M_\RR$ is well-defined, 
and is determined up to a translation.
Set $\theta_\alpha = \theta_{e_\alpha}(\vphi)=\la \bfm_\vphi, e_\alpha\ra$.
By adding a translation to $\bfm_{\vphi_0}$, we can assume $\theta_{\alpha}(\vphi_0)$ are all positive. 
Consequently, $\theta_{\alpha}(\vphi)$ are all positive, since the polytope $P$ is invariant when $\vphi_0$ is changed by adding a potential function.
Then $0<a_\alpha<\theta_{\alpha}<b_\alpha$ for positive bound $a_\alpha, b_\alpha$.
$g_{\vphi} = \prod_{\alpha=1}^r \theta_{\alpha}^{k_\alpha}(\vphi)$.

For any $\vec{k}=(k_1,\dots, k_r)\in \bN^r$, we set:
\begin{equation}
\bS^{\bvk}=S^{2k_1+1}\times \cdots \times S^{2 k_r+1}.
\end{equation}
$\bS^{\bvk}$ is a natural $(S^1)^r$-principal bundle with the action given by:
\begin{equation}
(e^{i a_1}, \cdots, e^{i a_r})\cdot (z^{(1)}, \cdots, z^{(r)})=(e^{i a_1}z^{(1)}, \dots, e^{i a_r} z^{(r)})
\end{equation}
where $z^{(\alpha)}=(z^{(\alpha)}_0,\dots, z^{(\alpha)}_{k_\alpha})$ and we use the identification $S^{2k_\alpha+1}=\{z^{(\alpha)}\in \bC^{k_1+1}; |z^{(\alpha)}|=1\}$.
Since $(S^1)^r$-acts on $(\tX, L)$, we have the associated fibre bundle $(\tX^\bvk, L^\bvk):=(\tX, L) \times_{(S^1)^r} \bS^\bvk$.

Then we have the diagram:
\begin{equation}
\begin{tikzcd}[column sep = 2.1cm]
\label{projections}
\tX \arrow[leftarrow, "\tilde{\beta}^{\bvk}"]{r} \arrow[d, "(\gamma')^{\bvk}"] & \tX\times \bS^{\bvk} \arrow[r, "\tilde{\pi}^{\bvk}"] \arrow[d, "\gamma^{\bvk}"] & \bS^{\bvk} \arrow[d, "(\gamma'')^{\bvk}"]\\
\tX/S^1 \arrow[leftarrow,"\beta^{\bvk}"]{r}  & \tX^{\bvk}  
\arrow[r, "\pi^{\bvk}"] &  \PP^{\bvk},
\end{tikzcd}
\end{equation}
where each fiber of $\pi^\bvk$ is isomorphic to $\tX$. 


For the Hermitian metric $e^{-\vphi}$ on $L$, there exists a corresponding Hermitian metric $e^{-\vphi^\bvk}$ on $L^\bvk$.
And $\vphi^\bvk-\vphi_0^\bvk = \gamma^{\bvk}_*(\beta^{\bvk})^*(\vphi-\vphi_0)$.
The equivariant cohomology associated to the group $U(1)$ is described in \cite[Section 4]{AB84}. 
For the group $U(1)$, there exists a Weyl group $W = \RR[\zeta, v]$, where $\zeta$ represents the 
angular $1$-form of $U(1)$ assigned with a degree $1$; $v$ is a free variable assigned with degree $2$.
A derivative $D$ is also defined on $W$, which satisfies $D\zeta = v, D v = 0$. Then $(W, D)$ forms the de Rham model of $EU(1)\simeq S^\infty$,
and $H^*(W,D)\simeq \RR$.
The $S^1$-invariant elements in $W$, i.e, $\RR[v]$, coupled with the derivative $D$, forms the de Rham model of $BU(1)\simeq \PP^\infty$, and
$H^*(\RR[v],D)\simeq \RR[v]$.
From the embedding $S^{2k+1}\rightarrow S^\infty$, we can see $(\RR[\zeta,v]/(v^{k+1}),D)$ forms the de Rham model for $S^{2k+1}$, where $\zeta\wedge v^k$ is the
generator of $H^{2k+1}(S^{2k+1},\RR)$.
Similarly, we have the embedding $\PP^{k}\rightarrow \PP^\infty$, and $(\RR[v]/(v^k),D)$ forms the de Rham model of $\PP^k$ (in real coefficient).

Following the idea of the model of $U(1)$, the ring
\begin{equation}
\Omega_T^*(\tX) = \{a\in\Omega^*(\tX):  \mathfrak{L}_\xi a = 0 \text{ for } \xi\in \RR^r\}
\end{equation}
is defined, and
there exists an isomorphism 
\begin{equation}
H\{\Omega_T^*(\tX)[v_1,\cdots,v_r], d_{\rm twist}\} \simeq H^*_T(\tX), 
\end{equation}
where $v_1,\cdots,v_r$ are free variables, and for any $a\in \Omega_T^*(\tX)$, $d_{\rm twist} a = da + \sum_{\alpha=1}^r (\iota_{\xi_\alpha}a) v_\alpha$.
This implies an isomorphism 
\begin{equation}
H\{\Omega_T^*(\tX)[v_1/(v_1)^{k_1+1},\cdots,v_r/(v_r)^{k_r+1}] \}\simeq H^*(\tX^\bvk).
\end{equation}

We will construct the metric $\ddc\vphi_j^{\bvk}$ in the following, which will fit with the equivariant cohomology picture above.
Since we are going to work on complex field, it is more convenient to consider the following diagram:
\begin{equation}
\begin{tikzcd}[column sep = 2.1cm]
\label{projections}
\CC^* \arrow[r, ""] \arrow[d, ""] & L\times (\CC^{k_1}\setminus\{0\})\times \cdots\times (\CC^{k_r}\setminus\{0\}) \arrow[r, ""] \arrow[d, ""] & L^{\bvk} \arrow[d, ""]\\
\CC^* \arrow[r,""]  & \tX\times (\CC^{k_1}\setminus\{0\})\times \cdots\times (\CC^{k_r}\setminus\{0\})
\arrow[r, ""] &  \tX^{\bvk}.
\end{tikzcd}
\end{equation}
We will use the coordinate $(x, s, (z_0^1,\cdots,z_{k_1}^1),\cdots, (z_0^r,\cdots,z_{k_r}^r))$ for points in $L\times (\CC^{k_1}\setminus\{0\})\times\cdots\times (\CC^{k_r}\setminus\{0\})$. The action of $t = (t_1,\cdots,t_r)\in N^T_\bR$ on $L\times (\CC^{k_1}\setminus\{0\})\times\cdots\times(\CC^{k_r}\setminus\{0\})$ is applied by $t\rightarrow (\sigma_t^*x, \sigma_t^*s, (t_1 z_0^1,\cdots,t_1 z_{k_1}^1),\cdots, (t_r z_0^r,\cdots,t_r z_{k_r}^r))$.
Then $(x,s, [z_0^1,\cdots,z_{k_1}^1],\cdots,[z_0^r,\cdots,z_{k_r}^r])$ and $(\sigma_t^*x,\sigma_t^* s, [z_0^1,\cdots,z_{k_1}^1],\cdots,[z_0^r,\cdots,z_{k_r}^r])$
are the same point in $L^{\bvk}$.

We will consider a neighborhood of a fiber $\tX\rightarrow \tX^\bvk\rightarrow \PP^{\bvk}$. 
Then we have
\begin{equation}
d(\vphi^\bvk) = d_x(\sigma_t^*\vphi)+d_t(\sigma_t^*\vphi) = d_x(\sigma_t^*\vphi) + \sum_{\alpha=1}^r \mathfrak{L}_{\xi_\alpha}(\vphi) A^\alpha,
\end{equation}
where $d_x$ means the differentiation along $x$ direction, $\mathfrak{L}_{\xi_\alpha}(\vphi) = \theta_\alpha$, $A^\alpha\in T^*\tX^\bvk$ is a connection of the line bundle, and $d_t A^\alpha = \ddc\vphi_\alpha^\FS$,
where $\vphi^{\FS}_\alpha$ is the pullback of the Fubini-Study metric from $\PP^{k_\alpha}$. 
Furthermore,
\begin{align} 
\begin{split}
\ddc (\vphi^\bvk) &= \ddc\vphi - \sum_{\alpha=1}^r d_x (\theta_\alpha)\wedge A^\alpha + d_x(\theta_\alpha)\wedge A^\alpha + \sum_{\alpha=1}^r \theta_\alpha \ddc\vphi_\alpha^\FS \\
&= \ddc\vphi + \sum_{\alpha=1}^r \theta_\alpha \ddc\vphi_\alpha^\FS.
\end{split}
\end{align} 
Note that by our normalization of $\bfm_\vphi$, each $\theta_\alpha$ is positive. So $\ddc\vphi^{\bvk}$ is indeed a metric
on $\tX^{\bvk}$.
As shown in the formula, $\ddc\vphi^\bvk$ can be decomposed into two parts:
$\ddc\vphi$ is the projection along fiber $\tX$, and $\sum_{\alpha=1}^r \theta_\alpha\ddc\vphi_\alpha^\FS$ is the projection perpendicular to the fiber direction.
From the construction above, we can also see that
\begin{align}
\begin{split}
\vphi^\bvk - \vphi_0^\bvk &= (\vphi - \vphi_0)^\bvk \\
\ddc (\vphi-\vphi_0)^\bvk &= \ddc (\vphi-\vphi_0) + \sum_{\alpha=1}^r (\theta_\alpha(\vphi) - \theta_\alpha(\vphi_0)) \ddc\vphi^\FS_\alpha .
\end{split}
\end{align}
For any $T$-invariant continuous function $w(x)$ on $\tX$, we set
\begin{align}
\begin{split}
\int_\tX w \MA_g(\vphi) &= \int_\tX w g_\vphi \frac{(\ddc\vphi)^n}{n!}  \\
&= \int_{\PP^\bvk}(\int_\tX w^\bvk \prod_{\alpha=1}^r \theta_\alpha^{k_\alpha}\frac{(\ddc\vphi)^n}{n!})\wedge \frac{(\ddc\vphi^\FS_1)^{k_1}}{k_1!}\cdots \wedge \frac{(\ddc\vphi^\FS_r)^{k_r}}{k_r !} \\
&=  \int_{\tX^\bvk} w^\bvk \frac{(\ddc\vphi^\bvk)^{n+k}}{(n+k)!} .
\end{split}
\end{align}

Now consider the general $\vphi\in {\rm Psh}_T(\tX,\omega_0)$. 
The regularization of $\vphi_M$ induces a decreasing sequence of smooth metrics $\vphi^\bvk_{M,j}$. Then the sequence converges to a limit $\vphi^\bvk_M$, 
which is a singular metric on $\tX^\bvk$.
Since the Monge-Amp\`{e}re measure converges along a decreasing sequence, we have $(\ddc\vphi^\bvk_M)^{n+k} = \lim_{j\to\infty} (\ddc\vphi^\bvk_{M,j})^{n+k}$.
Similarly, as $M\to\infty$, $\vphi_M$ is also a decreasing sequence. Then the limit
$\vphi^\bvk$ is a singular metric, and define $(\ddc\vphi^\bvk)^{n+k} = \lim_{M\to\infty}(\ddc\vphi^\bvk_M)^{n+k}$.
For any $T$-invariant test function $w(x)\in C^\infty(\tilde{X})$,
\begin{equation}
\int_\tX w \MA_g(\vphi) = \lim_{M\to\infty}\lim_{j\to\infty} \int_{X^\bvk} w^\bvk \frac{(\ddc\vphi_{M,j}^\bvk)^{n+k}}{(n+k)!} .
\end{equation}
For any test function $w$, we set $w^T=\int_T w(\sigma \cdot x)d\sigma$ where $d\sigma$ is the unit volume Haar measure on $T$. Then we define:
\begin{equation}
\int_\tX w \MA_g(\vphi)=\int_\tX w^T \MA_g(\vphi).
\end{equation}

Then $\MA_{g}$ is well-defined for any monomial and hence for any polynomial $g$.

Now consider the general continuous function $g $. For any $\epsilon>0$, there exists an $i_0>0$, such that for $i > i_0$,
$|g -g_i|<\epsilon$. Let $i,i'$ be any two indices larger than $i_0$.
For any test function $w(x)$, there exist $M>0, j>0$, such that
$|\int_\tX w(x) (\MA_{g_i}(\vphi_{M,j}) - \MA_{g_i}(\vphi))|< \epsilon$,
$|\int_\tX w(x) (\MA_{g_{i'}}(\vphi_{M,j}) - \MA_{g_{i'}}(\vphi))|< \epsilon$.
Then
\begin{align*}
&|\int_\tX w(x) (\MA_{g_i}(\vphi) - \MA_{g_{i'}}(\vphi))| \leq |\int_X w(x) (\MA_{g_i}(\vphi)-\MA_{g_i}(\vphi_{M,j}))| + \\
&|\int_\tX w(x) (\MA_{g_{i'}}(\vphi)-\MA_{g_{i'}}(\vphi_{M,j}))| +
|\int_\tX w(x) (\MA_{g_{i}}(\vphi_{M,j})-\MA_{g_{i'}}(\vphi_{M,j}))| \\
&\leq 2\epsilon + |\int_\tX w (g_{i, \vphi_{M,j}}-g_{i', \vphi_{M,j}})\frac{(\ddc\vphi_{M,j})^n}{n!}|\\
&\leq 2\epsilon + 2\epsilon \int_\tX |w| \frac{(\ddc\vphi_{M,j})^n}{n!}.
\end{align*}
By the estimate above, $\MA_{g_i}(\vphi_{M,j})$ converges, as $j\to\infty, M\to\infty, i\to\infty$.
Henceforth we can see
\begin{equation}
\MA_g(\vphi) = \lim_{i\to\infty}\lim_{M\to\infty}\lim_{j\to\infty} \MA_{g_i}(\vphi_{M,j})
\end{equation}
is a well-defined Radon measure by Riesz representation theorem.
\end{proof}

\subsection{Finite energy space}
By pushing down the measure defined in the previous subsection from $\tX$ to $X$, 
we have the definition of $\MA_g$ as a Radon measure on $X$.
The normalization \eqref{normalization} corresponds to assuming that $\MA_g, d\mu_0$ are probablity measures on $X$ (i.e. $\bV_g=1$ in \eqref{eq-Vg}).
We can define the following finite energy space (see \eqref{eq-Vg})
\begin{align}
\mcE_T:=\mcE_T(X, L) = \left\{ \vphi \in {\rm Psh}_T(X, L): \int_X g_\vphi \frac{(\ddc\vphi)^n}{n!} = \bV_g \right\}. 
\end{align}
\begin{remark}
$\mcE_T$ is the subspace of ${\rm Psh}_T(X,\omega_0)$ that $\MA_g(\vphi)$ does not charge pluripolar subset $\{u=-\infty\}$.
Since there exists $C>0$  that depends only on $g$ such that $|g_\vphi|<C$ for all $\vphi\in \mcH_T(X, L)$, $\MA_g(\vphi)$ 
is absolutely continuous with respect to $\frac{(\ddc\vphi)^n}{n!}$. By approximation, this holds for any $\vphi\in {\rm Psh}_T(X, L)$. As a result,
$\mcE_T \cong \{ \vphi\in {\rm Psh}_T(X, L): \int_X \frac{(\ddc\vphi)^n}{n!} = \bV_1\}$.
\end{remark}
We will also define
\begin{equation}\label{eq-cE1T}
\mcE^1_T:=\mcE^1_T(X, L) =\left\{ \vphi\in \mcE_T: \int_X  (\vphi-\vphi_0) \MA_g(\vphi) > -\infty\right\}. 
\end{equation}

For $\vphi\in \mcE^1_T(X, L)$, $\vphi = \vphi_0+u$, set $\vphi_t=\vphi_0+ t u$ and define:
\begin{align}
\bfE_g(\vphi) &= \frac{1}{\bV_g}\int_0^1 \int_X  (\vphi-\vphi_0) g_{\vphi_t}\frac{(\ddc\vphi_t)^n}{n!}dt \\
\bfI_g(\vphi) &= \frac{1}{\bV_g}\int_X  (\vphi-\vphi_0) \left(g_{\vphi_0}  \frac{(\ddc\vphi_0)^n}{n!} - g_\vphi \frac{(\ddc\vphi)^n}{n!}\right) \\
\bfJ_g(\vphi) &= \frac{1}{\bV_g} \int_X  (\vphi-\vphi_0) g_{\vphi_0}  \frac{\omega_0^n}{n!} - \bfE_g(\vphi) =: \Lam_g( \vphi) - \bfE_g( \vphi).
\end{align}
{Similar as in  \cite{TZ99}}, 
\begin{align}
\label{I-J}
\frac{1}{C} (\bfI-\bfJ) \leq \bfI_{g}-\bfJ_{g} \leq C (\bfI-\bfJ).
\end{align}

For simplicity of notations, we assume $\bV_g=1$ (by rescaling $g$).
For $\bfJ_g$, let $ u_t = t  u$.
\begin{align}
\label{J_1}
\begin{split}
\bfI_g(\vphi)-\bfJ_g(\vphi) &= \int_0^1 \frac{d}{dt}(\bfI_g-\bfJ_g)(\vphi_t) dt \\
&= -\int_0^1 \int_X  u_t (V_{f,\vphi}( u) + \Delta_t  u) g_{\vphi_t} \frac{\omega_{ u_t}^n}{n!} dt\\
&= \int_0^1 \int_X t g_{\vphi_t} \frac{\sqrt{-1}}{2\pi}\partial  u \wedge \bar\partial  u \wedge \frac{\omega_{ u_t}^{n-1}}{(n-1)!} dt\\
&= \int_X \sum_{k=0}^{n-1} (\int_0^1 \frac{1}{k! (n-k-1)!}(1-t)^k t^{n-k} g_{\vphi_t} dt) \frac{\sqrt{-1}}{2\pi}\partial  u \wedge \bar\partial  u \wedge 
\omega^{k} \wedge \ddc\vphi^{n-k-1} ,
\end{split}
\end{align}
where the last equality is by using $\omega_{ u_t} = (1-t)\omega + t\ddc\vphi$.
Similarly, we have
\begin{align}
\label{J_2}
\begin{split}
\bfJ_g(\vphi) &= \int_0^1 \int_X  u(g \frac{\w^n}{n!}- g_{\vphi_t}\frac{\w_{ u_t}^n}{n!})dt \\
&= -\int_0^1 \int_X  u \int_0^t(V( u) + \Delta_s  u) g_{\vphi_s} \frac{\omega_{ u_s}^n}{n!} ds dt\\
&= \int_0^1 \int_0^t \int_X g_{\vphi_s}\frac{\sqrt{-1}}{2\pi}\partial u\wedge \bar\partial u\wedge \frac{\w_{ u_s}^{n-1}}{(n-1)!} ds dt\\
&= \frac{1}{\bV_g}\int_X \int_0^1\int_0^t g_{\vphi_s} \frac{1}{k!(n-k-1)!} (1-s)^k s^{n-k-1} ds dt \frac{\sqrt{-1}}{2\pi} \partial u\wedge\bar\partial u\wedge \w^k \wedge \w_u^{n-k-1}
.
\end{split}
\end{align}

Then we have
\begin{align}
\label{J}
\frac{1}{C} \bfJ \leq \bfJ_g \leq C \bfJ
\end{align}

and 
\begin{align}
\label{I_J}
\frac{1}{C} \bfJ_g \leq \bfI_g -\bfJ_g \leq C \bfJ_g.
\end{align}
Moreover, for any $t\in [0,1]$, if we let $\vphi_t=(1-t)\vphi_0+t\vphi$, then
\begin{align*}
\frac{d}{dt}\bfJ_g(\vphi_t) = \frac{1}{\bV_g}\int_X  (\vphi-\vphi_0) \left(g_{\vphi_0} \frac{(\ddc\vphi_0)^n}{n!} - g_{\vphi_t}\frac{(\ddc\vphi_t)^n}{n!}\right) \ge \frac{1}{t}(1+\frac{1}{C})\bfJ_g(\vphi_t) .
\end{align*}
This is equivalent to 
$\frac{d}{dt}\log(\bfJ_g(\vphi_t)) \geq (1+\frac{1}{C})\frac{1}{t}$ and hence implies
\begin{align}
\label{J_t}
t^{1+\frac{1}{C}} \bfJ_g(\vphi) \geq \bfJ_g (\vphi_t).
\end{align}

By \eqref{I-J} \eqref{J}, we also have
\begin{align}
\label{I}
\frac{1}{C} \bfI \leq \bfI_g \leq C \bfI .
\end{align}

\begin{remark}
\label{smooth_net}
The calculation in the middle steps of \eqref{J_1} and \eqref{J_2} is done for $ \vphi\in \mcH_T$.
This is sufficient. Indeed, since $\mcE^1_T$ is the completion of $\mcH_T$ under the strong topology for each $ \vphi\in\mcE^1_T$, we can choose a sequence $ \vphi_j\in\mcH_T$
that converges to $\vphi$ in strong topology. Then we only need to show \eqref{J_1} and \eqref{J_2} hold for each $ \vphi_j$.
\end{remark}

\begin{lemma}\label{lem-Hartogs}
There exists $C>0$ such that for any $\vphi\in \mcE^1_T(X, L)$, 
\begin{equation}
 \sup(\vphi-\vphi_0)-C\le \Lam_g(\vphi)\le \sup(\vphi-\vphi_0) .
\end{equation}
\end{lemma}
\begin{proof}
The second inequality is clear. 
The first inequality follows easily from the Hartogs' lemma. Indeed, if there is no such $C>0$, then there exists a sequence $\{\vphi_j\}_{j=1}^\infty$ such that $u_j=\vphi_j-\vphi_0$ satisfies
\begin{equation}
\frac{1}{\bV_g}\int_X (u_j-\sup u_j)g_{\vphi_0}\frac{(\ddc\vphi_0)^n}{n!}\le -j.
\end{equation}
However, by Hartogs' lemma, $u_j-\sup u_j$ converges in $L^1$ to some $u_\infty\in L^1((\ddc\vphi_0)^n)$. Letting $j\rightarrow+\infty$, we get a contradiction.
\end{proof}

\begin{proposition}
$\bfE_g$ is monotone increasing, concave and upper semi-continuous for $ \vphi\in \mcE_T^1(X,L)$.
Moreover, for any decreasing sequence $ \vphi_j$ that converges weakly to $ \vphi\in \mcE_T^1(X, L)$,
$\bfE_g(\vphi_j)$ converges to $\bfE_g(\vphi)$. 
\end{proposition}
\begin{proof}
The monotonicity is a direct result of the definition of $\bfE_g$. 
For $ \vphi\in\mcE^1_T$, 
it can be approximated by a decreasing sequence of $ \vphi_j = \vphi_0+u_j\in \mcH_T$. By a similar calculation as in \eqref{J_2}, with normalization $\bV_g=1$,
\begin{align*}
&\bfE_g(\vphi_j) =\\
& -\int_X \sum_{k=0}^{n-1} \Big( \frac{1}{k!(n-k-1)!}\int_0^1\int_0^t (1-s)^ks^{n-k-1} g_{\vphi_{j,s}} dsdt \Big) \frac{\sqrt{-1}}{2\pi}\partial u_j\wedge \bar\partial u_j
\wedge \omega^k \wedge (\ddc\vphi_j)^{n-k-1} ,
\end{align*}
where $\vphi_{j,s} = (1-s)\vphi_0+s\vphi_j$.
Since $g_{\vphi_{j,s}}$ is uniformly bounded, the convergence result of Monge-Amp\`{e}re measure along a decreasing sequence $\vphi_j$
can be adapted to show that
$\bfE_g(\vphi_j)$ converges to $\bfE_g(\vphi)$.
As a consequence of the convergence of $\bfE_g$, to show $\bfE_g$ is concave, it suffices to show $\bfE_g$ is concave for $ \vphi\in \mcH_T$.
The concavity is shown by the second variation $\bfE_g''(t u)|_{t=0} = -\int_X g_\vphi \frac{\sqrt{-1}}{2\pi} \partial u\wedge\bar\partial u\wedge\frac{(\ddc\vphi)^{n-1}}{(n-1)!} \leq 0$.
The upper-semicontinuity is a consequence of the monotonicity and approximation property:
for a sequence $ \vphi_i$ that converges to $\vphi$ weakly in $\mcE^1_T$,
we can choose decreasing sequences $ \vphi_i^j\searrow \vphi_i,  \vphi^j\searrow \vphi$, and $ \vphi_i^j\xrightarrow{i\to\infty}  \vphi^j$ smoothly. 
Then $\lim_{i\to\infty}\bfE_g(\vphi_i^i) \leq \bfE_g(\vphi^j)$ for any $j$. Let $j\to\infty$, the limit is $\leq \bfE_g(\vphi)$ and the upper semi-continuity is proved.
\end{proof}

We will use the notation $\bfE_{g,\vphi}$ to emphasize the reference metric is $\ddc\vphi$.
As in the case when $g=1$, 
it can be shown that $\bfE_g$ also satisfies the cocycle property:
\begin{equation}
\bfE_{g,\vphi_0} ( \vphi_1) - \bfE_{g,\vphi_0}( \vphi_2) = \bfE_{g, \vphi_2}( \vphi_1),
\end{equation}
where $ \vphi_0, \vphi_1, \vphi_2\in \mcE_T^1(X,L)$.


\begin{lemma}
$\bfI_g$ satisfies a quasi-trangle inequality:
\begin{equation}
c(n,g) \bfI_g( \vphi_1, \vphi_2) \leq \bfI_g(\vphi_1,\vphi_0) + \bfI_g(\vphi_0, \vphi_2) ,
\end{equation}
where $\vphi_0, \vphi_1, \vphi_2,\in \mcE^1_T(X, L)$, and the constant $c(n,g)$ depends on the dimension $n$ and the function $g$.
\end{lemma}
\begin{proof}
By the inequality \eqref{I},
\begin{equation*}
\bfI_g( \vphi_1, \vphi_2) \leq C \bfI( \vphi_1, \vphi_2).
\end{equation*}
where $C$ depneds on $n,g$.
By \cite[Theorem 1.8]{BBEGZ}, there exists a constant $c(n)$ such that
\begin{equation*}
c(n)\bfI( \vphi_1, \vphi_2) \leq \bfI( \vphi_1,\vphi_0)+\bfI(\vphi_0, \vphi_2).
\end{equation*}
Use inequality \eqref{I} again,
then we have
\begin{equation*}
\bfI_g( \vphi_1, \vphi_2) \leq \frac{C^2}{c(n)}(\bfI_g( \vphi_1,\vphi_0)+\bfI_g(\vphi_0, \vphi_2)) .
\end{equation*}
\end{proof}

\begin{proposition}[see \cite{BBEGZ}]
$\mcE^1_{T,{\sup}}(X,L)=\{ \vphi\in \mcE^1_T(X,L):\int_X  (\vphi-\vphi_0) g \frac{\omega_0^n}{n!} = 0\}$ is complete under the topology induced by $\bfI_g$.
\end{proposition}
\begin{proof}
Let $ \vphi_j\in \mcE^1_{T,\sup}$ be a sequence such that $\bfI_g( \vphi_j)$ is a Cauchy sequence. 
By Hartogs' compactness lemma,
there exists a $ \vphi\in {\rm Psh}_T$, such that up to a subsequence, $ \vphi_j$ converges to $ \vphi$ in weak topology.
By \eqref{I_J} and \cite{BBEGZ}, $\bfI_g( \vphi_j, \vphi)$ converges to $0$, which implies $\bfJ_g( \vphi_j)-\bfJ_g( \vphi) = \int_X ( \vphi_j- \vphi)g \frac{\omega_0^n}{n!} +\bfE_g( \vphi_j)-\bfE_g( \vphi)$ converges to $0$.
Since weak convergence implies $L^1$-convergence, $\int_X ( \vphi_j- \vphi)g \frac{\omega_0^n}{n!}$ converges to $0$. This implies,
$\bfE_g( \vphi_j)$ converges to $\bfE_g( \vphi)$ which is $>-\infty$. Then $ \vphi\in \mcE^1_{T,\sup}(X,L)$.
\end{proof}

\subsection{Generalized Ding functional v.s generalized Mabuchi functional}
\begin{definition}
The generalized Ding functional is defined as
\begin{align}\label{eq-Ding}
\bfD(\vphi):=\bfD_{g,\Theta}(\vphi) = -\bfE_g(\vphi) + \bfL_\Theta(\vphi) ,
\end{align}
where $\bfL_\Theta(\vphi) = - \log(\int_X e^{-(\vphi-\vphi_0)}d\mu_0) = - \log(\int_X e^{-u+h_0}\frac{\omega_0^n}{n!})$.
\end{definition}

\begin{definition}
The generalized Mabuchi functional is defined as
\begin{align}
\label{eq-Mabuchi}
\bfM(\vphi) &:= \bfM_{g,\Theta}(\vphi)= \int_X \log(\frac{g_\vphi (\ddc\vphi)^n}{e^{h_0} \omega_0^n}) g_\vphi  \frac{(\ddc\vphi)^n}{n!} +\bfJ_g(\vphi) - \bfI_g(\vphi) .
\end{align}
\end{definition}
\begin{remark}
The generalized Mabuchi functional \eqref{eq-Mabuchi} can be considered as the ``integral'' of the Futaki invariant.
(More details about generalized Mabuchi functional can be found in section \ref{generalized-Mabuchi}.)
This is a generalization of the Mabuchi functional defined in \cite{TZ02}, but different from the one used in \cite{BW14} and \cite{His19}
(See Remark \ref{rem-Mabuchi}).
By the Moment map picture explained in Appendix \ref{moment}, $\bfM$ can be considered as a Kempf-Ness functional, which is expected to be convex
along a geodesic. (The convexity is proved in section \ref{generalized-Mabuchi}.) 
\end{remark}
Denote the probability measures
$d\nu = g_\vphi \frac{(\ddc\vphi)^n}{n!}, d\mu_0= e^{h_0} \frac{\omega_0^n}{n!}$.
Then 
\begin{align}
\label{Mabuchi_2}
\begin{split}
\bfM(\vphi) &:= \bfM_{g,\Theta}(\vphi):=\int_X \log(\frac{d\nu}{d\mu_0}) d\nu + \bfJ_g(\vphi) - \bfI_g(\vphi) \\
&= \bfH_{g,\Theta}(\vphi) + \bfJ_g(\vphi) - \bfI_g(\vphi) .
\end{split}
\end{align}

\begin{lemma}
\label{M_D}
For any $\vphi\in \mcE^1_T(X, L)$, 
$\bfM(\vphi) \geq \bfD(\vphi)$.
\end{lemma}
\begin{proof}
Note that 
\begin{align*}
{\bfI_g}(\vphi)-{\bfJ_g}(\vphi) 
&= -\int_X  (\vphi-\vphi_0) g_\vphi \frac{(\ddc\vphi)^n}{n!} + \bfE_g(\vphi).
\end{align*}
So we have the identity:
\begin{align*}
\bfM(\vphi) = -\int_X \log(\frac{e^{-(\vphi-\vphi_0)}d \mu_0}{d \nu}) d \nu - \bfE_g(\vphi) .
\end{align*}
By Jensen's inequality, we have
\begin{align*}
\bfM(\vphi) \geq -\log\left(\int_X e^{-(\vphi-\vphi_0)} d\mu_0\right) - \bfE_g(\vphi) = \bfL_\Theta(\vphi) - \bfE_g(\vphi) = \bfD(\vphi) .
\end{align*}
\end{proof}

\subsection{Duality}

Let $d\chi$ be a measure in $\mcM^1_T$. Using the Legendre transform, we define
\begin{align} 
\label{E^*}
\bfE_g^*(d\chi) &= \sup_{v\in \mcE^1_T} \left\{\bfE_g(v+\vphi_0) - \int_X v d\chi\right\}, \\
\bfL_\Theta^*(d\chi) &= \sup_{v\in \mcE^1_T}\left\{-\log(\int_X e^{-v} d\mu_0) - \int_X v d\chi\right\} .
\end{align} 

\begin{definition}
We define the space of energy bounded Radon measures $\mcM^1_T(X)$ as:
\begin{align}
\label{M^1_f}
\mcM^1_T(X) &= \{d\chi \text{ is a Radon measure }: \int_X d\chi = 1, d\chi \text{ is $T$-invariant}, \bfE_g^*(d\chi)<\infty\} .
\end{align}
\end{definition}
We have the following result: 
\begin{proposition}[\cite{BW14}]
\label{Berman_Nystrom}
Let $d\chi$ be a Radon measure on $X$, and $\int_X d\chi = 1$.There exists a unique (up to constant) $u\in \mcE^1_T(X,\omega)$ such that
$\MA_g(\vphi) = d\chi$ if and only if $d\chi \in \mcM^1_T(X)$.
\end{proposition}

The following duality lemma is essentially contained in \cite{BBEGZ}. For reader's convenience, we also
state a brief proof.
\begin{lemma}
Let $d\nu = g_\vphi \frac{(\ddc\vphi)^n}{n!}$. Then
\begin{align}
\bfE_g^*(d\nu) &= \bfI_g(\vphi) - \bfJ_g(\vphi) \\
\bfL_\Theta^*(d\nu) &= \bfH_{g,\Theta}(\vphi).
\end{align}
\end{lemma}
\begin{proof}
By letting $v = u = \vphi-\vphi_0$ in the definition \eqref{E^*}, we have $\bfE_g^*(d\nu) \geq \bfI_g(\vphi) - \bfJ_g(\vphi)$.
Since the first variation of $\bfE_g(v+\vphi_0) - \int_X v d\nu = 0$ at $v = u$,
and 
$\bfE_g(v+\vphi_0)$ is concave while $\int_X vd\nu$ is linear,
we have $\bfE_g^*(d\nu) \leq \bfI_g(\vphi) - \bfJ_g(\vphi)$.

By letting $v = -\log(\frac{d\nu}{d\mu_0})$, we have $\bfH_{g,\Theta}(u) \leq \bfL_\Theta^*(d\nu)$.
By Jensen's inequality, 
\begin{align*}
-\log(\int_X e^{-v-\log(\frac{d\nu}{d\mu_0})}d\nu) - \int_X v d\nu \leq \int_X \log(\frac{d\nu}{d\mu_0})d\nu ,
\end{align*}
we have the inequality of the other direction.
\end{proof}
 
This lemma implies the identity:
\begin{align}
\bfM = -\bfE_g^*(d\nu) + \bfL_\Theta^*(d\nu).
\end{align}

With the above discussion, we get the following two lemmas using the similar proof as in \cite{BBEGZ}:
\begin{lemma}
\label{H_compact}
For any $C>0$, $\{\vphi\in \mcE^1_T: \bfH_{g,\Theta}(\vphi)<C\}$ is precompact.
And $\bfH_{g,\Theta}$ is lower semicontinuous (l.s.c.) in $\mcE^1_T$ under strong topology. 
\end{lemma}

\begin{lemma}
\label{minimizer_equal}
$\inf_{\vphi\in\mcE^1_T} \bfD(\vphi) = \inf_{\vphi_\in\mcE^1_T}\bfM(\vphi)$.
\end{lemma}

\section{Existence and properness}
\label{existence_and_properness}
In this section, we will study the existence of the generalized Monge-Amp\`{e}re equation
\begin{align}
\label{MA}
\MA_g(\vphi) = e^{-(\vphi-\vphi_0)} d\mu_0 ,
\end{align}
where $\MA_g(\phi) = g_\vphi \frac{(\ddc\vphi)^n}{n!}$, $d\mu_0 =  e^{h_0}\frac{\omega_0^n}{n!}$ are probability measures on $X$. 

We use the notations $\Aut_T(X, D, \Theta), \bG$ defined in section \ref{intro}.  
Denote $\TT = C(\bG)\simeq (\bC^*)^\fr = ((S^1)^\fr)^\CC$ the center of $\bG$, where $(S^1)^\fr\subset K$;
and $\mathfrak{t}$ as the Lie algebra of $\TT$. And let $N_\RR$ be a $\fr$-dimensional real vector space. For any $\xi\in N_\RR$,
$\xi-\sqrt{-1}J\xi\in \mathfrak{t}$, where $J$ is the complex structure.


\begin{lemma}
\label{minimizer}
Any minimizer of the generalized Ding functional $\bfD$ is a solution to \eqref{MA}.
Conversely, any solution to \eqref{MA} is a minimizer of $\bfD$.
\end{lemma}
\begin{proof}
The proof is essentially the same as \cite[Proposition 2.16]{BW14}.
The key of the proof is to use the projection operator:
$P(f) = \sup^*\{\varphi\in {\rm Psh}(X, L): \varphi \leq f\}$.
Assume $\vphi$ is a minimizer of $\bfD$.
For any continuous function $v$, let $d(t) = -\bfE_g(P(\vphi+t v))+\bfL_\Theta(\vphi+tv)$.
Since $d(t)\leq \bfD(P(\vphi+tv))$, $d(0)$ is a minimizer.
By \cite{BB10}, $d(t)$ is differentiable at $t=0$, $d'(0) = \int_X v (-\MA_g(\vphi)+e^{-\vphi}d\mu) = 0$. Since $v$ can be any continuous function,
this implies $-\MA_g(\vphi) + e^{-\vphi}d\mu = 0$.
Then the minimizer $\vphi$ is a solution to \eqref{MA}.
The inverse direction is a straightforward result of the convexity of $\bfD$.
\end{proof}

The first part of the following lemma is similar to \cite[Theorem 5.2]{BBEGZ}, \cite[Theorem 6]{CDS15} or \cite[Proposition 7]{DaS16}. The second part is similar to \cite[Theorem 3.3]{His16b} and \cite[Theorem 2.15]{Li20}.
\begin{lemma}\label{lem-unique}
Assume that there exist $T$-invariant solutions to \eqref{MA}.
\begin{enumerate}
\item[i.)]
If $\vphi_1, \vphi_2$ are two $T$-invariant solutions to \eqref{MA}, there exists a 1-parameter subgroup $\lambda: (\bC, +)\rightarrow \Aut_T(X, D, \Theta)$ such that $\lambda(1)^*\vphi_2=\vphi_1$.
As a consequence, $\Aut_T(X, D, \Theta)$ is a reductive complex Lie group,
i.e, it is the complexification of a compact Lie group $K_\Aut$ and $K_\Aut$ is a subgroup of ${\rm Iso}(X, \ddc\vphi)$.

\item [ii.)]
If $\bG=K_\bC$ is a reductive subgroup that contains a maximal torus of $\Aut_T(X, D, \Theta)$, and $\vphi_0, \vphi_1\in \mcE^1_{T\times K}(X, L)$ are both solutions to \eqref{MA}, 
then there exists an $\sigma$ in the center of $\bG$,
such that $\ddc\vphi_1 = \sigma^*\ddc{\vphi_0}$.
\end{enumerate}
\end{lemma}
\begin{proof}
In the following proof, we will denote $G = \Aut_T(X, D, \Theta)$, and by $C(G)$ its center.
Let $\vphi_0, \vphi_1$ be two $T$-invariant solutions to \eqref{MA}. 
Let $\vphi_t$ be geodesic segments connecting $\vphi_0$ and $\vphi_1$.

By \cite{Bern15}, $\bfL_\Theta$ (or $\bfD$ since $\bfE_g$ is affine) is convex along a geodesic. 
{\bf Claim}: if $\bfL_\Theta$ is affine, then the geodesic is induced
by a holomorphic vector field $\xi$, where the imaginary part of $\xi$ is a Killing vector field.
We will show that our setting fits into the framework of \cite[Appendix C]{BBEGZ}.
By assumption, $D+\Theta$ is a klt current, and $\Theta = \ddc\psi \geq 0$. Then $(X,D)$ is a klt pair.
We have the ramification formula
\begin{equation}
K_\tX + D' = \rho^*(K_X+D) - E^- + E^+ ,
\end{equation}
where $D'$ is the strict transform of $D$, $E^-,E^+$ are effective exceptional divisors. Denote $\Delta = D'+E^-$.
Then coeffecients of prime divisors in $\Delta$ is in $[0,1)$.
Recall that, $d\mu_0 = \frac{s\wedge\bar{s}}{|s|^2_{e^{\vphi_0+\psi}}}$, where $s$ is a local holomorphic section of $K_X+D$.
Without the loss of generality, here we let the multiplicity $r=1$ .
Then $\rho^*d\mu_0 = \frac{\sigma\wedge\bar{\sigma}}{e^{\vphi_0+\psi+\vphi_\Delta}}$, 
where $\sigma$ is a local holomorphic section of $-Q = \rho^*(K_X+D)-\Delta$,  $\ddc\vphi_\Delta = [\Delta]$.
Then
\begin{align}
\begin{split}
\bfL_\Theta(\vphi) &= -\log(\int_X e^{-(\vphi-\vphi_0)} d\mu_0) \\
&= -\log (\int_X e^{-(\vphi+\psi)} s\wedge \bar{s})\\
&= -\log (\int_\tX e^{-(\vphi+\psi+\vphi_\Delta)}\sigma\wedge \bar{\sigma}) .
\end{split}
\end{align}
We have $h^0(\tX,K_\tX+Q) = h^0(\tX, E^+) = 1$. Since $-\rho^*(K_X+D)$ is semi-ample, by
Kawamata-Viehweg vanishing theorem,
$h^1(\tX,K_\tX+Q) = h^1(\tX,K_\tX-\rho^*(K_X+D)+\Delta) = 0$.
Then we can apply \cite[Appendix C]{BBEGZ} to conclude the claim. 
Denote this holomorphic morphism induced by $\xi$ as $\sigma_\xi$.
Since $\iota_\xi ([D]+\Theta) = 0$, $\sigma_\xi$ fixes the divisor $D$ and positive current $\Theta$ and $\sigma_\xi\in \Aut_0(X,D,\Theta)$.
As the geodesic is induced by a holomorphic vector field, $\vphi(t)$ is smooth in $X^{\rm sm}$.
Since $\xi^i = \vphi^{i\bar{j}}\dot{\vphi}_{\bar{j}}$ in $X^{\rm sm}$, and $\vphi$ is $T$-invariant, $\mathfrak{L}_{\xi'}\xi = 0$
for any holomorphic vector field $\xi'$ in the Lie algebra of $T$.
Then $\sigma_\xi$ commutes with $T$ and $\sigma_\xi\in G$. The rest of argument is the same as \cite[Proof of Theorem 5.1]{BBEGZ}.

For ii), assume there are two $T\times K$ invariant solutions $\vphi_0,\vphi_1$ to \eqref{MA}. Then again the geodesic segment that connects $\vphi_0,\vphi_1$ is induced
by a holomorphic vector field $\xi$, which generates a holomorphic action $\sigma_\xi\in G$ such that $\sigma_\xi^*\ddc{\vphi_0} = \ddc{\vphi_1}$.
Furthermore, $\ddc{\vphi_0}, \ddc{\vphi_1}$ are both $T\times K$-invariant.
In addition, $K$ is a maximal subgroup of $\bG$, then $\sigma_\xi^{-1} K \sigma_\xi = K$. Arguing as in \cite[Proof of Theorem 2.15]{Li19},
there exist $t\in C(\bG)$, $k\in K$, such that $\sigma_\xi = t k$. We can define $\sigma = t$. This concludes (ii).
\end{proof}


\begin{definition}\label{def-Gcoer}
Let $\bG=K_\bC\subset \Aut_T(X, D, \Theta)$ be a reductive Lie group and $\bT$ be its center.
We say a functional $F$ on $\mcE_{T\times K}^1$ is $\bG$-coercive (over $\mcE_{T\times K}^1$) if there exist positive constants $\delta>0, C>0$, such that
for any $\vphi\in \mcE^1_{T\times K}$, 
\begin{align}
F(\vphi) \geq \delta \inf_{\sigma\in \bT}\{\bfJ_g(\sigma^* \vphi)\} -C.
\end{align}
\end{definition}
We will consider $F$ either as the generalized Ding functional or generalized Mabuchi functional.
Since $\bfJ_g$ is comparable with $\bfJ$, this definition of $\bG$-coercivity is compatible with the one defined in
\cite{CTZ05} in case of K\"{a}hler-Ricci soliton.

\begin{lemma}
We have the following properties for the functional $\bfJ_g$.
\label{J_convex}
\begin{enumerate}
\item[i)]
$\bfJ_g$ is convex and proper over $N_\RR$.
\item[ii)]
For any $\vphi \in \mcE^1_{T\times K}$, $\inf_{\sigma\in \bT}\{\bfJ_g(\sigma^*\vphi)\}$ can be obtained at some
$\sigma\in \TT$.
\end{enumerate}
\end{lemma}
\begin{proof}
The proof of the convexity follows from the lines of \cite[Proposition 1.6]{His16b}.
We will work on the resolution $\pi:\tX\rightarrow X$. We will abbreviate $\pi^*\vphi$ as $\vphi$.
For $\xi\in N_\RR$, let $\sigma_{t\xi}\in \TT$ be the group action indued by $t\xi$.
The action $\sigma_{t\xi}$ induces a family $\tX\times\Delta$ and a metric $\Phi$ 
on the total space , where $\Phi|_{(x,t)} = \sigma_{t\xi}^*\vphi$.
By Lemma \ref{d-closed}, we have $\ddc \bfE_g = \int_\tX g_{\sigma_{t\xi}^*\vphi}\frac{(\ddc \Phi)^{n+1}}{n!}$ .
However, by the construction of $\Phi$, $(\ddc \Phi)^{n+1} = \sigma_{t\xi}^*(\ddc\vphi)^{n+1} = 0$.
This implies $\bfE_g$ is affine along $N_\RR$.
Then $\ddc \bfJ_g = \int_\tX \ddc(\Phi)\wedge g  \frac{\omega_0^n}{n!} \geq 0 $.
By Lemma \ref{trivial}, we have the slope at the infinity (along the trajectory of $\sigma_{t\xi}$) $\bfJ_g^{\NA} \neq 0$ if $\xi\neq 0$. 
This and the convexity implies the properness of $\bfJ_g$. 
And (ii) is a direct concequence of (i).
\end{proof}

The following generates the the analytic criterions in different situations as considered in \cite{Tia97, CTZ05,  PSSW, DR17, Dar17, DNG18, His16b, LZ19, Li19}.
\begin{theorem}
\label{equivalent}
Let $\bG\subseteq \Aut_T(X, D, \Theta)$ (see Definition \eqref{def-AutT}) be a connected reductive subgroup. 
Consider the following statements:
\begin{enumerate}
\item[i)]
the generalized Mabuchi functional $\bfM$ is $\bG$-coercive. 
\item[ii)]
the generalized Ding functional $\bfD$ is $\bG$-coercive. 
\item[iii)]
there exists a solution $\vphi\in \mcE_{T\times K}^1(X,\omega)$ to equation \eqref{MA}.
\end{enumerate}
Then i) $\Leftrightarrow$ ii) $\Rightarrow$ iii). 

Moreover, if $\bG\subseteq \Aut_T(X, D, \Theta)$ and contains a maximal torus of $\Aut_T(X, D, \Theta)$, then iii) $\Rightarrow$ i).
\end{theorem}
\begin{proof}
i) $\Rightarrow$ ii): 
Without the loss of generality, we may normalize $\int_X (\vphi-\vphi_0) e^{f_{\vphi_0}}\frac{(\ddc\vphi_0)^n}{n!} = 0$. Then $\bfJ_g(\vphi) = -\bfE_g(\vphi)$.
The constants $C$ in the following estimates may change from line by line.
By the $\bG$-coercivity of $\bfM$, we have
$\bfL_\Theta^*(d\chi) \geq  (1+\delta)\bfE_g^*(d\chi) -C$, for some $\delta > 0$.
Let $\epsilon = \frac{1}{1+\delta}$.
Then we have $\bfE_g^*(d\chi) \leq \epsilon \bfL_\Theta^*(d\chi) + C$, for $d\chi \in \mcM^1_T(X)$.
Then
\begin{eqnarray*}
\bfE_g((1-\epsilon)\vphi_0+\epsilon \vphi)&=& \inf_{d\chi\in\mcM^1_T} \left\{\int_X \epsilon(\vphi-\vphi_0) d\chi + \bfE_g^*(d\chi)\right\} \\
&\leq& \epsilon \inf_{d\chi\in \mcM^1_T}\left\{\int_X(\vphi-\vphi_0) d\chi + \bfL_\Theta^*(d\chi)\right\}+C\\
&\leq&\epsilon\cdot \bfL_\Theta(\vphi)+C .
\end{eqnarray*}
Then by \eqref{J_t},
\begin{align}
\bfL_\Theta(\vphi) \geq \epsilon^{-1} \bfE_g((1-\epsilon)\vphi_0+\epsilon\vphi) -C \geq \epsilon^{\frac{1}{C}} \bfE_g(\vphi)-C ,
\end{align}
which implies
\begin{align}
\bfD(\vphi) = \bfL_\Theta(\vphi)-\bfE_g(\vphi) \geq (1-\epsilon^{\frac{1}{C}}) \bfJ_g(\vphi) - C .
\end{align}
\\
ii) $\Rightarrow$ i):  This follows from Lemma \ref{M_D}. \\
ii) $\Rightarrow$ iii): If $\bfD$ is $\bG$-coercive, then it is bounded from below. Then any minimizing sequence which converges to a minimizer of $\bfD$ by the lower semicontinuity of $\bfD$ with respect to the weak topology. Then the implication follows by Lemma \ref{minimizer}. \\
Now assume $\bG$ contains a maximal torus of $\Aut_T(X, D, \Theta)$.\\
iii) $\Rightarrow$ i): 
Assume $\bfM$ is not $\bG$-coercive. Let $\vphi$ be a solution to \eqref{MA}. By Lemma \ref{minimizer_equal},
$\bfM(\vphi) = \bfD(\vphi)$.
By the assumption and Lemma \ref{J_convex}(ii), 
there exists a sequence of $\vphi_j\in \mcE^1_{T\times K}$, where $\bfJ_g(\vphi_j) = \inf_{\sigma\in \bT}\bfJ_g(\sigma^*\vphi_j)$, such that
$\bfM(\vphi_j)\leq \delta_j \bfJ_g(\vphi_j) - C_j$, where $\delta_j \to 0$, and $C_j\to \infty$.
Since the entropy $\bfH_{g,\Theta}(\vphi_j)$ is positive, we have $\bfJ_g(\vphi_j)\to \infty$ or else
$\bfH_{g,\Theta}(\vphi_j)\leq \delta_j \bfJ_g(\vphi_j)+(\bfI_g-\bfJ_g)(\vphi_j)-C_j\to -\infty$.
Let $\Phi_j(t)$ be the geodesic ray, that emanates from $\vphi$ and passes through $\vphi_j$. Denote the distance between $\vphi$ and $\vphi_j$
by $T_j$. Then $\Phi_j(T_j) = \phi_j$.
Since $\bfD\leq \bfM$, by the convexity of $\bfD$, for $t\in [0,T_j]$, we have
$\bfD(\Phi_j(t)) \leq \bfD(\vphi)+\frac{t}{T_j}(\bfD(\vphi_j)-\bfD(\vphi)) \leq (1-\frac{t}{T_j}) \bfD(\vphi) + (\delta_j-\frac{C_j}{T_j})t$.
As $j\to\infty$, up to choosing a subsequence, $\Phi_j$ converges weakly to a geodesic ray $\Phi$. 
From our construction, we can see that $\Phi$ is not an orbit of $\TT$ (see Lemma \ref{destablizer}).
By the lower semi-continuity of $\bfD$,
$\bfD(\Phi(t)) \leq \lim_{j\to\infty} \bfD(\Phi_j(t)) \leq \bfD(\vphi)$.
Then for any $t$, $\Phi(t)$ is a minimizer of $\bfD$. By Lemma \ref{minimizer}, $\Phi(t)$ is a solution of \eqref{MA}.
By Lemma \ref{lem-unique}.(ii), $\Phi$ is in an orbit of $\bT$, which is a contradiction.
\end{proof}

The proof of the corollary below essentially follows from \cite{BBEGZ}.
\begin{corollary}
\label{C^0}
Let 
\begin{equation}
\pi: \tilde{X}\rightarrow X
\end{equation}
be a resolution of $X$. If $\ddc\vphi$ is a weak solution to K\"{a}hler-Ricci soliton,
then $\pi^*\vphi \in L^\infty(\tilde{X})$.
\end{corollary}
\begin{proof}
Let $\bG = \Aut_T(X, D, \Theta)$.
By iii) $\Rightarrow$ i) in Theorem \ref{equivalent}, there exists a $\sigma\in \TT$, such that
\begin{align*}
\bfM(\vphi) \geq \frac{\delta}{2} \bfJ(\sigma^* \vphi) - C. 
\end{align*}
Since $\bfM$ is $\bG$-invariant, and the regularity of $\vphi, \sigma^* \vphi$ are the same, we will rename $\sigma^* \vphi$
as $\phi$, and have
\begin{align*}
\bfM(\vphi) \geq \frac{\delta}{2} \bfJ(\vphi) - C .
\end{align*}
Denote $d\nu = g_\vphi \frac{(\ddc\vphi)^n}{n!}$.
By duality, this implies
\begin{align*}
\bfL_\Theta^*(d\nu) - \bfE_g^*(d\nu) \geq \epsilon \bfE_g^*(d\nu) - C .
\end{align*}
Let $p = 1+\epsilon>1$. 

We have the Legendre transforms
\begin{align}
\bfL_\Theta^*(d\nu) &= \sup_{v\in \mcE^1_T}\left\{\int v d\nu - \log(\int_Xe^{v}d\mu)\right\},\\
\bfE_g(u) &= \inf_{d \nu\in \mcM^1_T}\left\{\int_X u d\nu + \bfE_g^*(d\nu)\right\} .
\end{align}
Together with the properness inequality, we have
\begin{align}
\log(\int_X e^{-p u}d\mu) &= \sup_{d\nu\in \mcM^1_T}\left\{\int_X -p u d\nu - \bfL_\Theta^*(d\nu)\right\}\\
&\leq \sup_{d\nu\in \mcM^1_T}\left\{-p\int_X u d\nu - p \bfE_g^*(d\nu)\right\}-C p\\
&\leq -p \inf_{d\nu\in \mcM^1_T}\left\{\int_X u d\nu + \bfE_g^*(d\nu) \right\}-C p
\leq -p \bfE_g(u) -C p ,
\end{align}

\begin{align}
\int_X e^{-p u}d\mu \leq C\cdot e^{-p \bfE_g(u)} .
\end{align}
This implies $e^{-u} \in L^p(d\mu)$. As $| f_\vphi |<C$, $\MA_g(\vphi)$ is comparable with $\MA(\vphi)$.
Then the machinary of \cite{EGZ09} can be applied to obtain the $L^\infty$-estimate. 
\end{proof}

\begin{proposition}
\label{higher_regularity}
Assume that $g>0$ is smooth on $P$ and let $\vphi$ be the solution  to \eqref{MA}. 
Let $S = E \cup D' \cup \{x\in \tX:\pi^*\psi \text{ is singular at $x$}\}$.
Then $\pi^*\vphi$ is smooth on $X\setminus S$.
\end{proposition}
\begin{proof}
We will state the proof briefly which basically follows from \cite{BBEGZ}.
For simplicity, we will denote $\vphi$ as the solution to \eqref{MA} on $\tX$. 
Let $\tilde{\omega} = \ddc\tilde{\vphi}$ be a $K\times T$-invariant K\"ahler metric on $\tX$.
The right-hand side of \eqref{MA} can be rewritten into the form
$e^{-\vphi-\vphi_0}d\mu_0 = e^{\psi^+-\psi^-}\Omega$, where $\psi^+,\psi^-$ are quasi-psh functions, i.e, 
$\ddc\psi^+,\ddc\psi^- \geq -C \tilde{\omega}$ for some constant $C>0$, 
$\Omega$ is a smooth non-degenerate volume form.
By the $L^\infty$-bound of $\vphi$ and the klt condition, $e^{-\psi^-}\in L^p(\omega_0^n)$ for some $p>1$.
Furthermore, for any open subset $U\subset \tX\setminus S$,
there exists a quasi-psh function $\bar{\psi}$ such that $U\subset \tX\setminus \{\bar{\psi}=-\infty\}$.
Without the loss of generality, we can set $U = \tX\setminus \{\bar{\psi}=-\infty\}$.

Let $\omega_\epsilon = \omega + \epsilon\tilde{\omega}$.
By Demailly's approximation result, there exist smooth approximations $\psi^\pm_\epsilon$ of $\psi^\pm$ which decreases to $\psi^\pm$,
such that 
$\ddc\psi^\pm_\epsilon\geq -C\tilde{\omega}$, and converges to $\ddc\psi^\pm$  weakly in current sense.
And $\|e^{-\psi^-_\epsilon}\|_{L^p(\omega_0^n)}$ is uniformly bounded. It's not hard to see that $\psi^\pm$ can be chosen to be $T$-invariant.
Then $(C+1)\tilde{\vphi}+\psi^\pm_\epsilon$ is a $T$-invariant K\"ahler metric on $\tX$, and $\xi_\alpha ((C+1)\tilde{\vphi}+\psi^\pm_\epsilon)$
is a coordinate function of the moment polytope corresponding to the K\"ahler class $[\ddc((C+1)\tilde{\vphi}+\psi^\pm_\epsilon)]$.
This implies $\xi_\alpha(\psi^\pm_\epsilon)$ is uniformly bounded.
For the same reason, since $\bar{\psi}$ can be approximated by $\bar{\psi}_\epsilon$, $\ddc\bar{\psi}_\epsilon\geq -C\tilde{\omega}$,
and $\bar{\psi}_\epsilon$ converges to $\bar{\psi}$ smoothly on $U$, we have
$\xi_\alpha(\bar{\psi})$ is uniformly bounded on $U$.

We will consider the following continuity method
\begin{equation}
e^{t f_{\vphi_{\epsilon,t}}} (\ddc\vphi_{\epsilon,t})^n = d_{\epsilon,t} e^{\psi^+_\epsilon-\psi^-_\epsilon}\Omega,
\end{equation}
where $\ddc\vphi_{\epsilon,t} = \omega_\epsilon + \ddc u_{\epsilon,t}$, $0\leq t\leq 1$, $0\leq \epsilon \ll 1$, and $d_{\epsilon,t}$ is a constant to
balance the cohomologous equality.
At $t=0$, by \cite[Lemma 3.6]{BBEGZ}, there exists a solution $u_{\epsilon,0}$ which is smooth on $\tX\setminus S$.
The openness over $t$ and $\epsilon$ is clear. The uniform $C^0$-estimate follows from the similar proof of corollary \ref{C^0}.
Next we need to show the uniform $C^2$-estimate.

Let $\theta_{\alpha,\omega_\epsilon}$ (abbreviated as $\theta_\alpha$ below) be the Hamiltonian function that corresponds to the holomorphic vector field  $\xi_\alpha$ with respect to $\omega_\epsilon$. 
We have $\theta_{\alpha,\vphi_{\epsilon,t}} = \theta_{\alpha}+\xi_\alpha(u_{\epsilon,t})$,
$\partial_{\bar{j}} f_{\vphi_{\epsilon,t}} = f_\alpha \partial_{\bar{j}}\theta_{\alpha,\vphi_{\epsilon,t}} = f_\alpha (\partial_{\bar{j}}\theta_\alpha + \xi_\alpha(\partial_{\bar{j}}u_{\epsilon,t}))$. Then
\begin{equation}
\label{equation3}
\Delta  f_{\vphi_{\epsilon,t}} \geq f_\alpha \xi_\alpha(n+\Delta u_{\epsilon,t}) - C (n+\Delta u_{\epsilon,t}) - C,
\end{equation}
where $\Delta$ is with respect to $\omega_\epsilon$. 
By Chern-Lu's inequality, 
\begin{align}
\begin{split}
\Delta_{\vphi_{\epsilon,t}} \log(n+ \Delta u_{\epsilon,t})  &\geq \frac{\Delta f-\Delta\psi^-_\epsilon}{n+\Delta u_{\epsilon,t}} -C - \sum_i \frac{C}{1+\partial_{i\bar{i}}u_{\epsilon,t}}\\
&\geq \frac{f_\alpha \xi_\alpha(n+\Delta u_{\epsilon,t})}{n+\Delta u_{\epsilon,t}}-\frac{\Delta\psi^-_\epsilon}{n+\Delta u_{\epsilon,t}} -C - \sum_i \frac{C}{1+\partial_{i\bar{i}}u_{\epsilon,t}} ,
\end{split}
\end{align}
where the constant $C>0$ above may change from line by line, but is uniformly bounded.
We will apply the maximum principle to $\log(n+\Delta u_{\epsilon,t})-A(u_{\epsilon,t}+\bar{\psi})+\psi^-_{\epsilon}$ for some $A$ sufficiently large.
Since $-\bar{\psi}$ approaches to $+\infty$ near the boundary of $U$, 
$\log(n+\Delta u_{\epsilon,t})-A (u_{\epsilon,t}+\bar{\psi})-\psi^-_\epsilon$ will achieve its maximal value at a point $x_0\in U$.

At  $x_0$,  $\xi_\alpha(\log(n+\Delta u_{\epsilon,t})-A (u_{\epsilon,t}+\bar{\psi})+\psi^-_\epsilon) = 0$;
$\xi_\alpha u_{\epsilon,t} = \theta_{\alpha,\vphi_{\epsilon,t}}-\theta_\alpha$ is uniformly bounded, 
and $\xi_\alpha \psi^\pm_\epsilon, \xi_\alpha \bar{\psi}$ are uniformly bounded. 
Then  at $x_0$,
\begin{align*}
\frac{f_\alpha \xi_\alpha(n+\Delta u_{\epsilon,t})}{n+\Delta u_{\epsilon,t}} &= f_\alpha \xi_\alpha(\log(n+\Delta u_{\epsilon,t})-A (u_{\epsilon,t}+\bar{\psi})+\psi^-_\epsilon)+ \\
& A f_\alpha \xi_\alpha(u_{\epsilon,t}+\bar{\psi}) - f_\alpha \xi_\alpha(\psi^-_\epsilon) \\
&= A f_\alpha (\theta_\alpha(\vphi)-\theta_\alpha + \bar{\psi}) - f_\alpha \xi_\alpha(\psi^-_\epsilon), 
\end{align*}
which is bounded.
Then
\begin{equation}
\Delta_{\vphi} (\log(n+\Delta u_{\epsilon,t})-A (u_{\epsilon,t}+\bar{\psi}) + \psi^-_\epsilon) \geq -C + \sum_{i}\frac{C}{1+\partial_{i\bar{i}}u_{\epsilon,t}} .
\end{equation}

Then
\begin{equation}
[\frac{n+\Delta u_{\epsilon,t}}{\prod_{i=1}^n (1+\partial_{i\bar{i}}u_{\epsilon,t})}]^{\frac{1}{n-1}} \leq \sum_{i}\frac{1}{1+\partial_{i\bar{i}}u_{\epsilon,t}} \leq C .
\end{equation}
Since $ f_{\vphi_{\epsilon,t}}$ is uniformly bounded, $e^{\psi^+}$ is uniformly bounded above, 
we have $n+\Delta u_{\epsilon,t}\leq (\ddc\vphi_{\epsilon,t})^n \leq C e^{-\psi^-}$.
This gives the uniform $C^2$-estimate on $U$.

The uniform $C^2$-estimate implies equation \eqref{MA} is a uniform elliptic equation. Then we can apply the standard Krylov-Evans estimate (which is a local estimate) to obtain a uniform $C^{2,\alpha}$-bound. The standard maximum principle would then imply the higher order estimates.
\end{proof}

\section{Generalized Mabuchi functional and its convexity}
\label{generalized-Mabuchi}

In this section, we will discuss generalized Mabuchi functional and generalized Futaki invariant. In particular, we will show that $\bfM$ is convex along a weak geodesic.

\subsection{Generalized Mabuchi functional $\bfM$}
\label{convexity}

Generalized Mabuchi functional $\bfM$ can be reformulated into:
\begin{align}
\label{Mabuchi_2}
\bfM(\vphi) &= \int_X \log\left(\frac{g_\vphi (\ddc\vphi)^n}{e^{ f_0} \omega_0^n}\right)g_\vphi \frac{(\ddc\vphi)^n}{n!}
-(\bfI_g(u)-{\bfJ_g}(u)) - \int_X (h_0- f_0) g_\vphi \frac{(\ddc\vphi)^n}{n!} ,
\end{align}
where $e^{ f_0} = g_{\vphi_0}$,
\begin{align*}
{\bfI_g}(\vphi)-{\bfJ_g}(\vphi) 
&= -\int_X  (\vphi-\vphi_0) g_\vphi \frac{(\ddc\vphi)^n}{n!} + \bfE_g(\vphi).
\end{align*}
Then
\begin{align}
\label{Mabuchi_1}
\begin{split}
& \bfM(\vphi) = \int_X \log(\frac{(\ddc\vphi)^n}{e^{-u+h_0} \omega_0^n})g_\vphi \frac{(\ddc\vphi)^n}{n!} 
 + \int_X   f_\vphi  g_\vphi \frac{(\ddc\vphi)^n}{n!}  - \bfE_g(\vphi) \\
&= \int_X \log(\frac{(\ddc\vphi)^n/n!}{e^{-u} d\mu_0})g_\vphi \frac{(\ddc\vphi)^n}{n!} 
 + \int_X   f_\vphi  g_\vphi \frac{(\ddc\vphi)^n}{n!}  - \bfE_g(\vphi)  .
\end{split}
\end{align}

Following \cite{TZ99} which generalizes \cite{Fut83}, we define a generalized Futaki invariant as
\begin{definition}
For $\vphi\in \mcH(X, L)$, set $e^{h_\vphi}=\frac{n! \cdot e^{-\vphi-\psi}}{(\ddc\vphi)^n}$ and $ f_\vphi =\log g_{\vphi}$.
For any $\xi\in \mathfrak{aut}_T(X, D, \Theta)$ (see \eqref{eq-autT}), we define:
\begin{align}
\label{Futaki}
\Fut_g(\xi):=\Fut_g([\omega], \xi):=\Fut_g(\ddc\vphi, \xi) = \int_{X} \xi (h_\vphi -  f_\vphi ) e^{f_\vphi} \frac{(\ddc\vphi)^n}{n!}.
\end{align}
\end{definition}
To see that the Futaki invariant is well-defined, we need to verify that the integral does not depend on the choice of K\"{a}hler metrics. Before explaining this fact in Proposition \ref{prop-Fut}, which is in some sense well-known, we first introduce some notation.

For any $\xi\in \mathfrak{aut}_T(X, D, \Theta)$ (see \eqref{eq-autT}), by definition we have the vanishing:
\begin{equation}\label{eq-iotavan}
\iota_\xi \Theta=\frac{\sqrt{-1}}{2\pi}\bar{\partial}(\xi(\psi))=0.
\end{equation} 
This together with $\iota_\xi \{D\}=0$ implies that $\xi$ lifts to induce an infinitesimal action $\tilde{\xi}$ on
(Cartier multiples of the $\bQ$-)line bundles $L=-K_X-D-B$ and hence on $B$. Moreover the vanishing \eqref{eq-iotavan} implies that the function
\begin{equation}
\chi_\psi(\tilde{\xi}):=\frac{\mathcal{L}_{\tilde{\xi}} e^{-\psi}}{e^{-\psi}}=-\tilde{\xi}(\psi)
\end{equation}
is globally a constant. It is well-known that different liftings of $\xi$ differ by a rescaling vector field along the fibre of the line bundle. In particular, by choosing $\tilde{\xi}_*=\tilde{\xi}+\tilde{\xi}(\psi)w\frac{\partial}{\partial w}$ where $w$ is a linear variable along the $\bC$-fibre, we get
\begin{equation}\label{eq-canlift}
\chi_\psi(\tilde{\xi}_*)=0.
\end{equation}
We call the lifting $\tilde{\xi}_*$ that satisfies \eqref{eq-canlift} the canonical lifting of $\xi$.
\begin{proposition}\label{prop-Fut}
With the canonical lifting of $\xi$ satisfying \eqref{eq-canlift}, we have the following formula for the generalized Futaki invariant:
\begin{equation}\label{eq-Futform}
\Fut_g(\ddc\vphi, \xi)=-\int_X \theta_{\xi,\vphi}e^{f_\vphi}\frac{(\ddc\vphi)^n}{n!},
\end{equation}
where $\theta_{\xi, \vphi}$ is the moment function associated to the canonical lifting of $\xi$ with respect to $\ddc\vphi$.
As a consequence, the Futaki invariant is well-defined and independent of the choice of $\ddc\vphi$.
\end{proposition}
\begin{proof}
We can calculate:
\begin{eqnarray*}
&&\int_X \xi(h_\vphi- f_\vphi )e^{f_\vphi}\frac{(\ddc\vphi)^n}{n!}=\int_X \xi\left(\log \frac{e^{-\vphi-\psi}}{(\ddc\vphi)^n/n!}- f_\vphi \right)e^{f_\vphi} \frac{(\ddc\vphi)^n}{n!}\\
&=&\int_X \frac{\mathfrak{L}_{\tilde{\xi}_*} e^{-\vphi-\psi}}{e^{-\vphi-\psi}} e^{f_\vphi}\frac{(\ddc\vphi)^n}{n!}=-\chi_{\psi}(\tilde{\xi}_*)-\int_X \theta_{\xi,\vphi} e^{f_\vphi}\frac{(\ddc\vphi)^n}{n!}.
\end{eqnarray*}
The above integration by parts on $X$ can be verified by lifting the integral to a resolution of $X$.
In case that $\xi$ generates a $\bC^*$-group $\la \xi\ra$ that commutes with $T_\bC$, the last integral can be calculated using the Duistermaat-Heckmann (DH) measure associated to the action of $\la \xi\ra\times T_\bC$ which does not depend on the choice of $\ddc\vphi$. The general case can be verified using the equivariantly closed differential forms as in \cite{Fut2}.

\end{proof}

\subsection{Alternative definition of generalized Mabuchi functional}
\label{Mabuchi_alternative}
In this subsection, we will discuss some equivalent formulas of $\bfM$. For simplicity, we require $X$ to be smooth.
We consider $\vphi$ as the metric of an ample $\RR$-line bundle of the form $L = -(K_X+D+\Theta)+H$, where $H$ is a $\RR$-line bundle.
An alternative definition of the generalized Mabuchi functional is by considering it as a Kempf-Ness functional: (see \cite{Mab86} and Appendix \ref{moment}.)
\begin{align}
\label{Mabuchi-double-integral}
\begin{split}
\bfM(\vphi) 
&= \int_0^1 \int_X \dot{\vphi}\Big(-R_\vphi + C_R +  \Delta f_\vphi  + {\rm tr}_\vphi(\ddc\psi+[D]) + V_{f,\vphi}( f_\vphi )\\
&\quad\quad +\frac{C_R}{n} \sum_\alpha f_\alpha  \mathfrak{L}_{\xi_\alpha}(\vphi) + \sum_\alpha  f_\alpha  \mathfrak{L}_{\xi_\alpha}(\log(\ddc\vphi)^n) \Big)g_\vphi \frac{(\ddc\vphi)^n}{n!} dt 
\end{split}
\end{align}
where 
$C_R$ is a constant such that $\int_X R_\vphi - {\rm tr}_\vphi(\ddc\psi+[D]) \frac{(\ddc\vphi)^n}{n!} =\int_X C_R \frac{(\ddc\vphi)^n}{n!}$,
which does not depend on the specific choice of $\vphi$. 
Note that we can write the integrand in terms of $\theta_\alpha:=\theta_\alpha(\vphi)=\mathfrak{L}_{\xi_\alpha}(\vphi)$ by the following identities:
\begin{eqnarray*}
&&\Delta  f_\vphi = f_{\alpha\beta}  (\theta_\alpha)_i (\theta_\beta)^i+ f_\alpha  \Delta \theta_\alpha, \quad
V_{f,\vphi}( f_\vphi )=  f_\alpha  f_\beta (\theta_\alpha)_i (\theta_\beta)^i\\
&&\mathfrak{L}_{\xi_\alpha}(\log (\ddc\vphi)^n)=\frac{\mathfrak{L}_{\xi_\alpha}(\ddc\vphi)^n}{(\ddc\vphi)^n}=\Delta \theta_\alpha.
\end{eqnarray*}
\begin{lemma}
When $L = -(K_X+D+\Theta)$, 
definition \eqref{Mabuchi-double-integral} is equivalent to \eqref{Mabuchi_2} and \eqref{Mabuchi_1}.
\end{lemma}
\begin{proof}
When $L = -(K_X+D+\Theta)$, $C_R = n$ and
\begin{eqnarray*}
\sum_\alpha  f_\alpha  \theta_\alpha+ f_\alpha  \Delta \theta_\alpha& =&\sum_\alpha  f_\alpha  \left(\mathfrak{L}_{\xi_\alpha}(\vphi)+ \mathfrak{L}_{\xi_\alpha} \log ((\ddc\vphi)^n)\right) \\
&=&\sum_\alpha  f_\alpha  \xi_\alpha \left(\log(\frac{(\ddc\vphi)^n/n!}{e^{-\vphi-\psi}})\right) = -V_{f,\vphi}(h_\vphi).
\end{eqnarray*}
Here we have used the canonical lifting $\mathfrak{L}_{\tilde{V}_\alpha}(\psi)=0$. 
Then \eqref{Mabuchi-double-integral} can be reduced to
\begin{align}
\begin{split}
\bfM(\vphi) &= \int_0^1 \int_X \dot{\vphi}(-R_\vphi + n +\Delta f_\vphi  + {\rm tr}_\vphi(\ddc\psi+[D]) - V_{f,\vphi}(h_\vphi- f_\vphi ))g_\vphi  \frac{(\ddc\vphi)^n}{n!} dt \\
&= -\int_0^1\int_X \dot{\vphi} \big( \Delta(h_\vphi- f_\vphi )+ V_{f,\vphi}(h_\vphi- f_\vphi ) \big) g_\vphi  \frac{(\ddc\vphi)^n}{n!}dt -C_0. 
\end{split}
\end{align}
The constant $C_0 = \int_X (h_0- f_0)e^{ f_0}\frac{\omega_0^n}{n!}$ in the second equality.
By using integration by parts, 
we can see the definition above is an integral of the generalized Futaki invariant (up to minus a constant).
Indeed, this alternative definition agrees with \eqref{Mabuchi_2} and \eqref{Mabuchi_1}. Since the proof is short, we will include it below.
\begin{eqnarray*}
&& -\int_0^1\int_X \dot{\vphi} \big( \Delta(h_\vphi- f_\vphi )+ V_{f,\vphi}(h_\vphi- f_\vphi ) \big) g_\vphi  \frac{(\ddc\vphi)^n}{n!}dt \\
&=& \int_0^1\int_X g_\vphi \frac{\sqrt{-1}}{2\pi} \partial(h_\vphi- f_\vphi )\wedge \bar\partial \dot{\vphi} \wedge \frac{(\ddc\vphi)^{n-1}}{(n-1)!} dt\\
&=& \int_0^1\int_X (h_\vphi- f_\vphi )(g_\vphi \frac{(\ddc\vphi)^n}{n!})'dt\\
&=& -\int_X (h_\vphi- f_\vphi ){g_\vphi}\frac{(\ddc\vphi)^n}{n!} + \int_X (h_0- f_0){g_{\vphi_0}}\frac{(\ddc\vphi_0)^n}{n!} + \int_0^1\int_X (h_\vphi- f_\vphi )' g_\vphi  \frac{(\ddc\vphi)^n}{n!} dt \\
&=& \int_X \big( -h_\vphi +  f_\vphi  \big) g_\vphi \frac{(\ddc\vphi)^n}{n!} - 
\int_0^1\int_X (-\Delta\dot{\vphi}-\dot{\vphi}-\dot{f}_\vphi) g_\vphi \frac{(\ddc\vphi)^n}{n!} dt +C_0\\
&=& \int_X \log(\frac{(\ddc\vphi)^n/n!}{e^{-u}d\mu_0}) g_\vphi  \frac{(\ddc\vphi)^n}{n!} + \int_X  f_\vphi  g_\vphi \frac{(\ddc\vphi)^n}{n!}
-\bfE_g(\vphi) + C_0 ,
\end{eqnarray*}
where  we used 
$\dot{f}_\vphi = V_{f,\vphi}(\dot{\vphi})$,
$(\Delta\dot{\vphi}+\dot{f}_\vphi)g_\vphi \frac{(\ddc\vphi)^n}{n!} = (g_\vphi \frac{(\ddc\vphi)^n}{n!})'$.
\end{proof}
The following generalized Chen-Tian's formula will be used in section \ref{convex_singular} on a {\it smooth} ambient space to approximate Mabuchi functional on singular spaces. Such type of formula has appeared in the study of K\"{a}hler-Ricci solitons (see \cite{JL19}) and the twisted K-energy formula in \cite{BDL17}. 
After we finish the first version of our paper, we notice that such type of Chen-Tian formula has also been established in a very general setting (at least for the un-twist case) in \cite[Theorem 5]{Lahdili19} for the so-called $(v, w)$-Mabuchi energy. However because the notations in \cite{Lahdili19} differ from ours significantly and that we are also dealing with the twisted case, we will keep our proof for the reader's convenience.


\begin{lemma}
For $\bfM$ defined as in \eqref{Mabuchi-double-integral}, we have generalized Chen-Tian's formula:
\begin{align}
\label{Chen-Tian}
\bfM(\vphi) &= \int_X \log(\frac{g_\vphi(\ddc\vphi)^n}{\omega_0^n})g_{\vphi}\frac{(\ddc\vphi)^n}{n!}dt
-\bfE_{g}^{Ric^\eq_0}(\vphi)+ C_R \bfE_g(\vphi) \\
& + \frac{C_R}{n}\int_0^1\int_X\dot{\vphi}\sum_\alpha  f_\alpha  \theta_\alpha g_\vphi\frac{(\ddc\vphi)^n}{n!}dt  
 +\bfE_g^{\ddc\psi}(\vphi)+\bfE_g^D(\vphi) \nonumber 
\end{align}
where
\begin{eqnarray*}
\bfE_{g}^{Ric^\eq_0}(\vphi)&=&
\int_0^1\int_X\dot{\vphi}g_\vphi (n Ric_0-\sum_\alpha  f_\alpha  (\mathfrak{L}_{\xi_\alpha}(\log\omega_0^n))(\ddc\vphi))\wedge \frac{(\ddc\vphi)^{n-1}}{n!}dt \\
&=&\int_0^1 \int_X \dot{\vphi} \left(n g Ric_0-\sum_\alpha g_\alpha (\mathfrak{L}_{\xi_\alpha}(\log \omega_0^n))(\ddc\vphi)\right)\wedge\frac{(\ddc\vphi)^{n-1}}{n!}dt\\
\bfE_g^D(\vphi) &=& \int_0^1\int_D \dot{\vphi}g_\vphi \frac{(\ddc\vphi)^{n-1}}{(n-1)!}dt, \;
\bfE_g^{\ddc\psi}(\vphi) = \int_0^1\int_X \dot{\vphi}g_\vphi \ddc\psi\wedge \frac{(\ddc\vphi)^{n-1}}{(n-1)!}dt.
\end{eqnarray*}
\end{lemma}

\begin{proof}
The proof is by applying integration by parts calculation to \eqref{Mabuchi-double-integral}. 
\begin{align*}
& -\int_0^1\int_X \dot{\vphi} R_\vphi g_\vphi \frac{(\ddc\vphi)^n}{n!} dt  \\
&= \int_0^1\int_X \dot{\vphi}\Delta \log(\frac{(\ddc\vphi)^n}{\omega_0^n}) g_\vphi \frac{(\ddc\vphi)^n}{n!}dt -\int_0^1\int_X \dot{\vphi}g_\vphi Ric_0\wedge \frac{(\ddc\vphi)^{n-1}}{(n-1)!} dt \\
&= -\int_0^1\int_X g_\vphi \vphi^{i\bar{j}}\partial_i\log(\frac{(\ddc\vphi)^n}{\omega_0^n})\partial_{\bar{j}}\dot{\vphi}\frac{(\ddc\vphi)^{n}}{n!}dt 
-\int_0^1\int_X \dot{\vphi}  g_\vphi \vphi^{i\bar{j}}\partial_i \log(\frac{(\ddc\vphi)^n}{\omega_0^n})\partial_{\bar{j}}  f_\vphi \frac{(\ddc\vphi)^{n}}{n!}dt \\
&-\int_0^1\int_X \dot{\vphi}g_\vphi Ric_0\wedge \frac{(\ddc\vphi)^{n-1}}{(n-1)!}dt \\
&= \int_0^1\int_X\log(\frac{(\ddc\vphi)^n}{\omega_0^n})(\vphi^{i\bar{j}}\partial_i  f_\vphi \partial_{\bar{j}}\dot{\vphi}+\Delta\dot{\vphi})g_\vphi\frac{(\ddc\vphi)^n}{n!}dt
- \int_0^1\int_X \dot{\vphi}g_\vphi V_{f,\vphi}(\log(\frac{(\ddc\vphi)^n}{\omega_0^n})) \frac{(\ddc\vphi)^n}{n!}dt \\
&-\int_0^1\int_X \dot{\vphi}g_\vphi Ric_0\wedge \frac{(\ddc\vphi)^{n-1}}{(n-1)!}dt \\
&= \int_0^1\int_X \log(\frac{(\ddc\vphi)^n}{\omega_0^n}) (g_\vphi\frac{(\ddc\vphi)^n}{n!})'dt
- \int_0^1\int_X \dot{\vphi}g_\vphi V_{f,\vphi}(\log(\frac{(\ddc\vphi)^n}{\omega_0^n})) \frac{(\ddc\vphi)^n}{n!}dt \\
&-\int_0^1\int_X \dot{\vphi}g_\vphi Ric_0\wedge \frac{(\ddc\vphi)^{n-1}}{(n-1)!}dt \\
&= \int_X \log(\frac{(\ddc\vphi)^n}{\omega^n})g_\vphi\frac{(\ddc\vphi)^n}{n!} -\int_0^1\int_X\Delta \dot{\vphi} g_\vphi\frac{(\ddc\vphi)^n}{n!}dt \\
& -\int_0^1\int_X \dot{\vphi}  V_{f,\vphi}(\log(\frac{(\ddc\vphi)^n}{\omega^n}))g_\vphi \frac{(\ddc\vphi)^n}{n!}dt - \int_0^1\int_X\dot{\vphi}g_\vphi Ric_0 \wedge \frac{(\ddc\vphi)^{n-1}}{(n-1)!}dt.
\end{align*}
Then
\begin{align*}
&\bfM(\vphi) = \int_X \log(\frac{(\ddc\vphi)^n}{\omega^n_0})g_\vphi \frac{(\ddc\vphi)^n}{n!}  -\int_0^1\int_X\dot{\vphi} g_\vphi Ric_0\wedge\frac{(\ddc\vphi)^{n-1}}{(n-1)!}dt \\
& - \int_0^1\int_X \dot{\vphi} V_{f,\vphi}(\log(\frac{(\ddc\vphi)^n}{\omega^n_0}))g_\vphi \frac{(\ddc\vphi)^n}{n!} 
 + \int_0^1\int_X \dot{\vphi} \sum_\alpha  f_\alpha  \mathfrak{L}_{\xi_\alpha}(\log((\ddc\vphi)^n)) g_\vphi \frac{(\ddc\vphi)^n}{n!}dt \\
&+  \frac{C_R}{n}(n\bfE_{g}(\vphi)+\int_0^1\int_X \dot{\vphi}  f_\vphi  g_\vphi\frac{(\ddc\vphi)^n}{n!}dt)  
+ \bfE_g^{\ddc\psi}(\vphi)+\bfE_g^D(\vphi) \\
&+ \int_0^1\int_X \big( -\Delta \dot{\vphi} g_\vphi\frac{(\ddc\vphi)^n}{n!} +\dot{\vphi}\Delta f_\vphi  g_\vphi \frac{(\ddc\vphi)^n}{n!}+\dot{\vphi} V_{f,\vphi}( f_\vphi ) g_\vphi \frac{(\ddc\vphi)^n}{n!} \big) dt  \\
&= \int_X \log(\frac{(\ddc\vphi)^n}{\omega_0^n}) g_\vphi \frac{(\ddc\vphi)^n}{n!} -\int_0^1\int_X \dot{\vphi} g_\vphi (nRic_0-\sum_\alpha  f_\alpha  \mathfrak{L}_{\xi_\alpha}(\log(\omega_0^n)) \wedge \frac{(\ddc\vphi)^{n-1}}{n!}dt\\
&+ \frac{C_R}{n}( n \bfE_{g}(\vphi) +\int_0^1\int_X \dot{\vphi}\sum_\alpha  f_\alpha \theta_\alpha(\vphi) g_\vphi \frac{(\ddc\vphi)^n}{n!}dt)  +\bfE_g^{\ddc\psi}(\vphi)+\bfE_g^D(\vphi),
\end{align*}
where the integral in the fourth line vanishes by using integration by parts. 
\end{proof}
\begin{remark}\label{rem-ERic0}
One can verify that the above functionals in Lemma \ref{Chen-Tian} do not depend on the choice of path connecting $\vphi_0$ to $\vphi$. 
We briefly explain this fact for $\bfE^{Ric_0^\eq}_g(\vphi)$ which is an equivariant analogue of the $\bfE^{Ric_0}$-functional in the original Chen-Tian's formula. Consider the following one form defined on the space of K\"{a}hler metrics:
\begin{eqnarray*}
\Gamma(u)&=&\int_X u g_\vphi(n Ric_0-\sum_\alpha  f_\alpha  \mathfrak{L}_{\xi_\alpha}(\log(\omega_0^n)))\wedge \frac{(\ddc\vphi)^{n-1}}{n!}\\
&=&\int_X u \left(g_\vphi \vphi^{i\bar{j}}(Ric_0)_{i\bar{j}}-\sum_\alpha g_{\alpha,\vphi} \Delta_0 \theta^0_\alpha\right) \frac{(\ddc\vphi)^n}{n!}.
\end{eqnarray*}
By using the identity $\bar{\partial}\Delta_0 \theta^0_\alpha=\iota_{\xi_\alpha}Ric_0$, a straightforward calculation via several integration by parts shows that:
\begin{eqnarray*}
u_1(\Gamma(u_2))&=& {\rm Re}\left(\int_X ((u_1)_{\bar{i}}(u_2)_j (Ric_0)_{i\bar{j}} g-(u_1)_{\bar{i}}(u_2)_i (Ric_0)_{k\bar{k}})g \frac{(\ddc\vphi)^n}{n!}\right)=u_2(\Gamma(u_1)),
\end{eqnarray*}
which means that $\Gamma$ is a closed 1-form and so the integral does not depend on the choice of paths. 
\end{remark}

\subsection{Convexity of $\bfM$ when $X$ is smooth}
In this subsection, we will use Berman-Berndsson's work \cite{BB17} to show that $\bfM$ is convex along a weak geodesic in case $X$ is smooth.
This is a prelude to the general case, which will be discussed in Section \ref{convex_singular}.
Define the metric on K\"ahler potential space by
\begin{align}
\label{L2_norm}
\la v_1, v_2\ra = \int_X v_1 v_2 g_\vphi  \frac{(\ddc\vphi)^n}{n!} .
\end{align}
A simple calculation shows that the connection is given by
\begin{align}
\nabla v=\dot{v} + \frac{1}{2}\la\nabla \dot{u}, \nabla v\ra
\end{align}

and the geodesic induced by the metric above is 
\begin{align}
\label{geodesic_equation}
 \ddot{u} - \frac{1}{2}|\nabla  \dot{u}|^2 = 0 .
\end{align}

By the Kempf-Ness picture (see Appendix \ref{moment}),$\bfM$ should have a certain convex property. 
A rigorous proof of the convexity along a $C^{1,1}$-geodesic is stated in the following.

Consider the total space $X\times A$, where $A = [0,1]\times S^1$ is an annulus, with complex local coordinate $z_{n+1}$, and $t = Re(z_{n+1})$.
Let $d d^c \Phi = \pi_1^* \ddc\vphi$ be the metric on $X\times A$. 

\begin{lemma}
\label{d-closed}
For any differentiable form $v$ on $X\times A$,
\begin{equation}
d(\int_X v(x) g_\vphi (\ddc\Phi)^n) = \int_X g_\vphi   dv \wedge (dd^c \Phi)^n,
\end{equation}
where $\int_X$ is the operator of integration along the fibre for the trivial fibration $X\times A\rightarrow X$.
\end{lemma}
\begin{proof}
Since the exterior derivative commutes with the integral along the fibers, we get
\begin{eqnarray*}
\bar{\partial} \int_X v g_\vphi (\ddc\Phi)^n&=&\int_X (\bar{\partial} v) \wedge g_\vphi (\ddc\Phi)^n+ \int_X v \bar{\partial} g_\Phi   \wedge (\ddc\Phi)^n.
\end{eqnarray*}
By Lemma \ref{lem-etaPhi}, 
\begin{eqnarray*}
\bar{\partial}(g_\Phi)\wedge (\ddc\Phi)^n&=&\frac{2\pi}{\sqrt{-1}}\iota_{V_{g,\vphi}}(\ddc\Phi)\wedge (\ddc\Phi)^n\\
&=&\frac{2\pi}{\sqrt{-1}}\frac{1}{n+1}\iota_{V_{g,\vphi}}(\ddc\Phi)^{n+1} .
\end{eqnarray*}
Since $V_{g,\vphi}$ is along the vertical direction of the projection $X\times A\rightarrow A$, it is easy to see that the second integral vanishes. Similar consideration applies to the differential operator $\partial$.
\end{proof}

In the following, we will use the expression \eqref{Mabuchi_1} to show that $\bfM$ is convex along a weak geodesic.
We will divide $\bfM$ into two parts,
$ (I) = \int_X \log(\frac{(\ddc\vphi)^n}{e^{-u+h_0} \omega_0^n})g_\vphi \frac{(\ddc\vphi)^n}{n!}$, 
$ (II) =  \int_X   f_\vphi  g_\vphi \frac{(\ddc\vphi)^n}{n!} - \bfE_g( \vphi)$. 



We will follow Berman-Berndtsson's calculation for (I).
By Lemma \eqref{d-closed},
\begin{align}
\begin{split}
&\ddc \int_X \log(\frac{(\ddc\vphi)^n}{e^{-u+h_0}\omega_0^n}) g_\vphi  (\ddc \Phi)^n  = \int_X g_\vphi  \ddc \log(\frac{(\ddc\vphi)^n}{e^{-u+h_0}\omega_0^n}) 
\wedge (\ddc \Phi)^n \\
&= \int_X g_\vphi  \ddc \log((\ddc\vphi)^n) \wedge (\ddc \Phi)^n  
- \int_X g_\vphi  \ddc \log(e^{-u+h_0}\omega_0^n)\wedge (\ddc \Phi)^{n} \\
&= \int_X g_\vphi  \ddc \log((\ddc\vphi)^n) \wedge (\ddc \Phi)^n  
+ \int_X g_\vphi   (\ddc \Phi)^{n+1} +\int_X g_\vphi \ddc\psi\wedge (\ddc\Phi)^n + \int_D g_\vphi (\ddc\Phi)^n\\
&\ge \int_X g_\vphi  \ddc \log((\ddc\vphi)^n) \wedge (\ddc \Phi)^n  .
\\
\end{split}
\end{align}
In the third equality, 
the last three terms are induced by $-\ddc\log(e^{-u}d\mu_0) = \ddc\vphi + [D] + \Theta$. Each of the three terms is nonnegative along a geodesic.
The term in the last line is positive by using Bergman kernel approximation approach by Berman-Berndtsson.
Then (I) is convex.
For (II), by a direct calculation, we can see that $\int_X  f_\vphi  g_\vphi \frac{(\ddc\vphi)^n}{n!}$ is constant.
This can also be seen by considering it as an integral on polytope $P$ associated with a Duistermaat-Heckman measure.
Then $\ddc {\rm (II)} = \ddc \bfE_g(u) = \int_X g_\vphi \frac{(\ddc\Phi)^{n+1}}{(n+1)!} = 0$.
Then we conclude with
\begin{equation}
\ddc \bfM(\vphi) \ge \int_X g_\vphi  \ddc \log((\ddc\vphi)^n) \wedge \frac{(\ddc \Phi)^n}{n!}
\end{equation}
and $\bfM$ is convex along a weak geodesic.

\begin{remark}
In a recent paper \cite{Lahdili20}, over a smooth manifold $X$, the Mabuchi functional is also shown to be convex along weak geodesics under a very general setting. This generalized Mabuchi functional is called $(v,w)$-Mabuchi energy, and we refer readers to \cite{Lahdili20} for more details. 
\end{remark}

\subsection{Convexity of generalized Mabuchi functional in singular case}
\label{convex_singular}
In this subsection, we will prove that the generalized Mabuchi functional $\bfM$ is convex along a geodesic when $X$ is a variety.
The idea of the proof is similar to \cite[4.1.2]{LTW19}.
Let $P = \pi^*L-E_b$ be an ample $\QQ$-line bundle on the resolution $\tX$. (We will abbreivate $\pi^*L$ as $L$ when there is no ambiguity.)
Our strategy is to approximate $\vphi$ by $\vphi_\epsilon\in c_1(L+\epsilon P)$,
then prove that the generalized Mabuchi functional $\bfM_{\vphi_{0,\epsilon}}(\vphi_\epsilon)$ is convex along a geodesic.
(where the subscript is to emphasis $\vphi_{0,\epsilon}$ is the chosen reference metric.)
At last we will show that $\bfM_{\vphi_{0,\epsilon}}(\vphi_\epsilon)$ converges to $\bfM(\vphi)$ at the end points which would imply the convexity of the latter.

By generalized Stone-Weierstrass theorem, $g$ can be approximated by polynomials in $C^2$-norm. 
Let 
\begin{align}
\begin{split}
\hat{\mcH}_T(X,\mcL) &= \{\vphi = \vphi_0+u\in {\rm Psh}_T(X,\mcL): u\in L^\infty(\tX)\cap C^\infty(\tX\setminus E),
(\ddc\pi^*\vphi)^n \text{ is smooth over } \tX,\\
& \text{ and there exist $\alpha, C>0$ such that }
|\ddc u| \leq C |s_E|^{-\alpha} \}
\end{split}
\end{align}
We consider a pair of metrics $\vphi(0), \vphi(1)\in \hat{\mcH}_T(X,\mcL)$,  which is connected by a weak geodesic $\vphi(t)$.
We assume $\sup_X \vphi(i) = 0$.
Let  $\omega_P = \ddc\vphi_P$ be a $K\times T$-invariant K\"ahler metric on $\tX$,
$\ddc\vphi_{0,\epsilon} = \omega_\epsilon = \omega + \epsilon \omega_P$, where $0\leq \epsilon \leq 1$;
$\vphi_\epsilon(i)\in \mcH(\tX,L_\epsilon)$, which will be determined later;
$\vphi_\epsilon = \vphi_{0,\epsilon}+u_\epsilon$;
$\vphi_\epsilon(t)$ be the geodesic that connects $\vphi_\epsilon(0),\vphi_\epsilon(1)$.
Since the $T_\CC$-action can be lifted up to the resolution $\tX$, we can lift the action of $T_\CC$ to $E$, as well as to $L+\epsilon P$.
Then the moment map $\bfm_{\vphi_{\epsilon}}$ is well-defined.
By the generalized Chen-Tian's formula \eqref{Chen-Tian} for the polarized pair $(\tilde{X}, D'=\pi^{-1}_*D, L+\epsilon P)$, we have
\begin{align}
\begin{split}
\label{Hartog's}
& \bfM_{\vphi_{0,\epsilon}}(\vphi_\epsilon) \\
&= \int_\tX \log(\frac{(\ddc\vphi_\epsilon)^n}{\omega_\epsilon^n})e^{f_{\vphi_\epsilon}}\frac{(\ddc\vphi_\epsilon)^n}{n!} \\
&+\Big( \frac{C_{R,\epsilon}}{n}\bfE_g^{\ddc{\vphi_\epsilon}^\eq}(\vphi_\epsilon)  - \bfE_{g}^{Ric_\epsilon^\eq}(\vphi_\epsilon) + \bfE_g^{D'}(\vphi_\epsilon) + \bfE_g^{\ddc\psi}(\vphi_\epsilon)  \Big)
+ C_{1,\epsilon}\\
&= \bfH_{g, \omega_\epsilon}(\vphi_\epsilon) + \bfF_{g,\epsilon}(\vphi_\epsilon) + C_{1,\epsilon} ,
\end{split}
\end{align}
where the constant
$C_{R,\epsilon}$ satisfies $\int_\tX(R_{\vphi,\epsilon}-{\rm tr}_{\vphi_{\epsilon}}(\ddc\psi+[D']))\frac{(\ddc\vphi_\epsilon)^n}{n!} = \int_\tX C_{R,\epsilon}\frac{(\ddc\vphi_\epsilon)^n}{n!}$,
$C_{1,\epsilon} = \int_\tX  f_{\vphi_{0,\epsilon}}g_{\vphi_{0,\epsilon}}\frac{\omega_\epsilon^n}{n!}$ (without the loss of generality, we can drop the constant $C_{1,\epsilon}$ in the following analysis), 
and
\begin{eqnarray}
\label{E_omega_eq}
\begin{split}
\bfE_g^{\ddc{\vphi_\epsilon}^\eq}(\vphi_\epsilon) 
= \int_0^1\int_\tX \dot{\vphi}_{s,\epsilon} {g_{\vphi_{s,\epsilon}}}\big( \frac{(\ddc\vphi_{s,\epsilon})^n}{(n-1)!} +\sum_\alpha  f_\alpha  \mathfrak{L}_{\xi_\alpha}(\vphi_{s,\epsilon})\frac{(\ddc\vphi_{s,\epsilon})^n}{n!}\big) ds \\
= \int_0^1\int_\tX u_\epsilon {g_{\vphi_{s,\epsilon}}}\big( \frac{(\ddc\vphi_{s,\epsilon})^n}{(n-1)!} + \sum_\alpha  f_\alpha  \mathfrak{L}_{\xi_\alpha}(\vphi_{s,\epsilon})\frac{(\ddc\vphi_{s,\epsilon})^n}{n!}\big) ds ,
\end{split}
\end{eqnarray}
where $\vphi_{s,\epsilon} = \vphi_{0,\epsilon} + s u_\epsilon$, $V_{s,\epsilon} = \sum_{1\leq\alpha\leq r}f_\alpha(\bfm_{\vphi_{s,\epsilon}}) V_{\alpha}$, and $V_{\alpha}$ still denotes the lifting of $\xi_\alpha$ over $\tX$.
\begin{eqnarray}
\label{E_Ric_eq}
\begin{split}
\bfE_{g}^{Ric^\eq_\epsilon}(\vphi_\epsilon) 
= \int_0^1\int_\tX \dot{\vphi}_{s,\epsilon} {g_{\vphi_{s,\epsilon}}}\big( Ric_\epsilon\wedge \frac{(\ddc\vphi_{s,\epsilon})^{n-1}}{(n-1)!} - \sum_\alpha  f_\alpha  \mathfrak{L}_{\xi_\alpha}(\log(\omega_\epsilon^n))\frac{(\ddc\vphi_{s,\epsilon})^n}{n!}\big) ds\\
= \int_0^1\int_\tX u_\epsilon \big( g Ric_\epsilon\wedge \frac{(\ddc\vphi_{s,\epsilon})^{n-1}}{(n-1)!} - \sum_\alpha g_\alpha\mathfrak{L}_{\xi_\alpha}(\log(\omega_\epsilon^n))\frac{(\ddc\vphi_{s,\epsilon})^n}{n!}\big) ds ,
\end{split}
\end{eqnarray}
where $Ric_\epsilon = -\ddc\log(\omega_\epsilon^n)$.
\begin{eqnarray}
\label{E_D}
\begin{split}
\bfE_g^{D'}(\vphi_\epsilon) =\int_0^1\int_{D'} \dot{\vphi}_{s,\epsilon} {g_{\vphi_{s,\epsilon}}}\frac{(\ddc\vphi_{s,\epsilon})^{n-1}}{(n-1)!}ds\\
=\int_0^1\int_{D'} u_\epsilon {g_{\vphi_{s,\epsilon}}}\frac{(\ddc\vphi_{s,\epsilon})^{n-1}}{(n-1)!}ds ,
\end{split}
\end{eqnarray}
\begin{eqnarray}
\label{E_psi}
\begin{split}
\bfE_g^{\ddc\psi}(\vphi_\epsilon) = \int_0^1\int_\tX \dot{\vphi}_{s,\epsilon} {g_{\vphi_{s,\epsilon}}}\ddc\psi \wedge \frac{(\ddc\vphi_{s,\epsilon})^{n-1}}{(n-1)!}ds \\
= \int_0^1\int_\tX u_\epsilon {g_{\vphi_{s,\epsilon}}}\ddc\psi \wedge \frac{(\ddc\vphi_{s,\epsilon})^{n-1}}{(n-1)!}ds  .
\end{split}
\end{eqnarray}

\begin{lemma}
$\bfM_{\vphi_{0,\epsilon}}$ is convex along $\vphi_\epsilon(t)$.
\end{lemma}
\begin{proof}
We will first prove the case when $g = \sum_\vk a_\vk \prod_{\alpha} y_\alpha^{k_\alpha}$ is a polynomial, where $y\in \RR^r$, $\vk$ belongs to a finite set.
Then $\sum_\alpha g_\alpha \xi_\alpha=\sum_{\alpha}D_\alpha g(\bfm_{\vphi_{s,\epsilon}})V_{\alpha}= \sum_\vk a_\vk \sum_\alpha k_\alpha \prod_\beta  \theta_{\beta,\epsilon}^{k_\beta-\delta_{\alpha,\beta}}V_{\alpha}$.
Denote $\vtheta_{\alpha,\epsilon} = \mathfrak{L}_{V_{\alpha}}(\log(\omega^n_\epsilon))$.
\begin{align*}
&\int_\tX \dot{\vphi}_\epsilon \sum_\alpha g_\alpha (\mathfrak{L}_{\xi_\alpha} (\log\omega^n_\epsilon))\frac{(\ddc\vphi_\epsilon)^n}{n!} = \sum_\vk a_\vk \int_X \dot{\vphi}_\epsilon \sum_\alpha k_\alpha \prod_\beta \theta_{\beta,\epsilon}^{k_\beta-\delta_{\alpha,\beta}}
\vtheta_{\alpha,\epsilon} \frac{(\ddc\vphi_\epsilon)^n}{n!} =\\
& \sum_\vk a_\vk \int_{\tX^\bvk} \dot{\vphi}_\epsilon \sum_\alpha \vtheta_{\alpha,\epsilon} (\ddc\vphi^\FS_{\alpha,\epsilon})\wedge \frac{(\ddc\vphi_\epsilon)^n}{n!}\wedge \frac{(\theta_{1,\epsilon}\ddc\vphi_{1,\epsilon}^\FS)^{k_1}}{k_1!} \\
&\hskip 5cm \cdots\wedge \frac{(\theta_{\alpha,\epsilon}\ddc\vphi^\FS_{\alpha,\epsilon})^{k_\alpha-1}}{(k_\alpha-1)!}\wedge \cdots \wedge \frac{(\theta_{r,\epsilon}\ddc\vphi^\FS_{r,\epsilon})^{k_r}}{k_r!}.
\end{align*}
Then
\begin{align}
\label{bfE_Riceq}
\begin{split}
-\bfE_{g}^{Ric^\eq_\epsilon}(\vphi_\epsilon) &= \sum_\vk a_\vk \int_0^1\int_{\tX^\bvk}u_\epsilon \ddc \log(\omega_\epsilon^n)^\bvk \wedge \frac{(\ddc\vphi^\bvk_{s,\epsilon})^{n+k-1}}{(n+k-1)!} ds \\
&= \sum_\vk a_\vk \frac{1}{(n+k)!}\sum_{j=0}^{n+k-1}\int_{\tX^\bvk} u_\epsilon \ddc \log(\omega_\epsilon^n)^\bvk \wedge (\ddc\vphi_\epsilon^\bvk)^{n+k-1-j}
\wedge (\ddc\vphi_{0,\epsilon}^\bvk)^j .
\end{split}
\end{align}
Then
\begin{equation}
-\ddc \bfE_{g}^{Ric^\eq_\epsilon} = \sum_\vk a_\vk \int_{\tX^\bvk} \ddc \log(\omega_\epsilon^n)^\bvk \wedge \frac{(\ddc\Phi^\bvk_\epsilon)^{n+k}}{(n+k)!} .
\end{equation}
By a similar calculation,
\begin{align}
\label{bfE_keq}
\begin{split}
\bfE_g^{\ddc{\vphi_\epsilon}^\eq}(\vphi_\epsilon) &= \sum_\vk a_\vk \frac{(n+k)}{(n+k+1)!} \sum_{j=0}^{n+k} \int_{\tX^\bvk} u_\epsilon (\ddc\vphi^\bvk_\epsilon)^{n+k-j}\wedge (\ddc\vphi^\bvk_{0,\epsilon})^j \\
&= \sum_\vk a_\vk (n+k) \bfE^\bvk_{\vphi^\bvk_{0,\epsilon}}(\vphi^\bvk_\epsilon).
\end{split}
\end{align}
\begin{equation}
\label{ddcbfE_keq}
\ddc \bfE_g^{\ddc{\vphi_\epsilon}^\eq} = \sum_\vk a_\vk (n+k)\int_{\tX^\bvk} \frac{(\ddc\Phi^\bvk_\epsilon)^{n+k+1}}{(n+k+1)!}
\end{equation}
and (compare with the calculation in Remark \ref{rem-ERic0})
\begin{align}
\label{nn}
\begin{split}
&\ddc (\int_\tX \log(\frac{(\ddc\vphi_\epsilon)^n}{\omega_\epsilon^n}){g_{\vphi_\epsilon}} \frac{(\ddc\vphi_\epsilon)^n}{n!} -\bfE_{g}^{Ric^\eq_\epsilon}) \\
&= \ddc\big( \sum_\vk a_\vk \int_{X^\bvk} \log(\frac{(\ddc\vphi_\epsilon)^n}{\omega_\epsilon^n}) \frac{(\ddc\vphi_\epsilon^\bvk)^{n+k}}{(n+k)!} \big) -\ddc\bfE_{g}^{Ric^\eq_\epsilon} \\
&= \sum_\vk a_\vk \int_{X^\bvk} \ddc\log((\ddc\vphi_\epsilon)^{n})^\bvk \wedge \frac{(\ddc\Phi_\epsilon^\bvk)^{n+k}}{(n+k)!} 
- \int_{\tX^\bvk} \ddc\log(\omega_\epsilon^n)^\bvk \wedge \frac{(\ddc\Phi_\epsilon^\bvk)^{n+k}}{(n+k)!} - \ddc\bfE_{g}^{Ric^\eq_\epsilon}\\
&=  \sum_\vk a_\vk\int_{\tX^\bvk} \ddc\log((\ddc\vphi_\epsilon)^{n})^\bvk \wedge \frac{(\ddc\Phi_\epsilon^\bvk)^{n+k}}{(n+k)!} \\
&=  \int_\tX {g_{\vphi_\epsilon}} \ddc\log((\ddc\vphi_\epsilon)^{n}) \wedge \frac{(\ddc\Phi_\epsilon)^{n}}{(n)!}  .
\end{split}
\end{align}
The equality of the last two lines can be checked locally. Let $U$ be an open affine chart in $\tX$,
\begin{align}
\begin{split}
\ddc \int_U \log((\ddc\vphi_\epsilon)^n) {g_{\vphi_\epsilon}}\frac{(\ddc\vphi_\epsilon)^n}{n!}
&= \ddc\big( \sum_\vk a_\vk \int_{U^\bvk} \log((\ddc\vphi_\epsilon)^n)^\bvk \frac{(\ddc\vphi_\epsilon^\bvk)^{n+k}}{(n+k)!} \big)\\
&= \sum_\vk a_\vk \int_{U^\bvk} \ddc\log((\ddc\vphi_\epsilon)^n)^\bvk \wedge \frac{(\ddc\Phi_\epsilon^\bvk)^{n+k}}{(n+k)!}  . 
\end{split}
\end{align}
By Berman-Berndtsson's Bergman kernel approximation method, \eqref{nn} is non-negative. 
When $\Phi_\epsilon$ is a geodesic, $\Phi_\epsilon^\bvk$ is also a geodesic. (compare with \eqref{L2_norm}-\eqref{geodesic_equation}.)
From \eqref{ddcbfE_keq}, along a geodesic, we have
$\ddc \bfE_{g,\omega^\eq_{\vphi_\epsilon}}(\vphi_\epsilon) =0$.
Similar as in the proof of the smooth case, it can be shown that $\bfE_g^D(\vphi_\epsilon),\bfE_g^{\ddc\psi}(\vphi_\epsilon)$ are convex along $\vphi_\epsilon(t)$.
Then when ${g_{\vphi_\epsilon}}$ is a polynomial, $\bfM_{\vphi_{0,\epsilon}}(\vphi_\epsilon)$ is convex along the geodesic $\vphi_\epsilon(t)$.

{ In the general case, we can approximate $g $ by polynomials $g_j$. 
For simplicity, denote the corresponding generalized Mabuchi functional by $\bfM_j(t)$, and denote the limit by $\bfM(t)$.
For any $\epsilon>0$, when $j$ is sufficiently large, $|\bfM_j(t)-\bfM(t)|<\epsilon$.
Then $\bfM(t) \leq \bfM_j(t)+\epsilon \leq (1-t)t\bfM_j(0) + t\bfM_j(1)+\epsilon \leq (1-t)\bfM(0)+t\bfM(1) + 3\epsilon$.
Take the limit, we can see $\bfM_{\vphi_{0,\epsilon}}$ is convex.}
\end{proof}

\begin{remark}
Our expansion of generalized Mabuchi functional is partially motivated by the expansion in \cite{Tia05}. See also \cite{Nak11}.
\end{remark}

Let $\vphi_\epsilon(i)$ $(i=0,1)$ be the solution to the Monge-Amp\`{e}re equation
\begin{equation}
\label{MA_approximation}
{g_{\vphi_\epsilon(i)}} (\ddc\vphi_\epsilon(i))^n = d_\epsilon {g_{\vphi(i)}}(\ddc\vphi(i))^n, \; \sup_\tX u_\epsilon = 0 ,
\end{equation}
where the right-hand side of the equation is a degenerate smooth volume form and $d_\epsilon = \int_\tX {g_{\vphi_{0,\epsilon}}}\omega_\epsilon^n / \int_\tX {g_{\vphi_0}}\omega^n$. 
\begin{lemma}
The solution $u_\epsilon(i)\in C^2(\tX\setminus E)$, and $|\nabla^2 u_\epsilon(i)| \leq C |s_E|_{\vphi_{0,1}}^{-\alpha}$, for some $C, \alpha>0$.
\end{lemma}
\begin{proof}
We will brief the proof. Denote $\Omega = g_\vphi (\ddc\vphi)^n$, which is a smooth volume form.
Consider the continuty method $(\ddc\vphi_{\epsilon,t})^n = d_{\epsilon,t}e^{-t f_{\vphi_{\epsilon,t}}}\Omega$.
At $t=0$, this is solved by Yau's resolution \cite{Yau78}. The $C^0$-estimate follows \cite{EGZ09}.
The $C^2$-estimate follows from a similar argument as in the proof of Proposition \ref{higher_regularity}.
\end{proof}

\begin{lemma}
\label{geodesic_ends}
For $i=0,1$, up to a subsequence, $u_\epsilon(i)$ converges to $u(i)$ in $L^1(\omega_{1}^n)$. Furthermore, $\lim_{\epsilon\to 0} \bfE_{g,\vphi_{0,\epsilon}}(\vphi_\epsilon(i)) = \bfE_g(\vphi(i))$, $\lim_{\epsilon\to 0} \bfE_g^{\ddc\vphi(i)_\epsilon^\eq}(\vphi_\epsilon(i)) = \bfE_g^{\ddc\vphi(i)^\eq}(\vphi(i))$.
\end{lemma}
\begin{proof}
Abbreviate $u(i)$ as $u$.
By Hartogs' compactness theorem, up to a subsequence, $\ddc\vphi_\epsilon$ converges to $\ddc\vphi$ as current, which furthermore implies $u_\epsilon$ converges
to $u$ in $L^1(\omega_{1}^n)$.

$\bfE_{g,\vphi_{0,\epsilon}}(\vphi_\epsilon) - \bfE_{g, \vphi_0}(\vphi) = \big(\bfE_{g,\vphi_{0,\epsilon}}(\vphi_\epsilon) - \bfE_{g,\vphi_{0,\epsilon}}(\vphi-\vphi_0 + \vphi_{0,\epsilon})\big)
+\big( \bfE_{g,\vphi_{0,\epsilon}}(\vphi-\vphi_0+\vphi_{0,\epsilon})-\bfE_{g,\vphi_0}(\vphi)\big)$.
Note that $\vphi-\vphi_0+\vphi_{0,\epsilon} = \vphi_{0,\epsilon}+u = \vphi + \epsilon \vphi_P$.
\begin{align*}
\bfE_{g,\vphi_{0,\epsilon}}(\vphi_\epsilon) - \bfE_{g,\vphi_{0,\epsilon}}(\vphi+\epsilon\vphi_P) &= \bfE_{g,\vphi+\epsilon\vphi_P} (\vphi_\epsilon) \\
&= \bfI_{g,\vphi+\epsilon\vphi_P}(\vphi_\epsilon) -\bfJ_{g,\vphi+\epsilon\vphi_P}(\vphi_\epsilon) + \int_\tX (u_\epsilon-u){g_{\vphi_\epsilon}}(\ddc\vphi_\epsilon)^n
\end{align*}
$\bfI_{g,\vphi+\epsilon\vphi_P}(\vphi_\epsilon) = \int_\tX (u_\epsilon-u)\big( {g_{\vphi+\epsilon\vphi_P}}(\ddc(\vphi+\epsilon\vphi_P))^n - {g_{\vphi_\epsilon}}(\ddc\vphi_\epsilon)^n\big)$.
There exists some $C>0$, such that ${g_{\vphi+\epsilon\vphi_P}}(\ddc(\vphi+\epsilon\vphi_P))^n \leq C \omega_1^n$.
Then
\begin{equation}
|\int_\tX (u_\epsilon-u) {g_{\vphi+\epsilon\vphi_P}}(\ddc(\vphi+\epsilon\vphi_P))^n| \leq C \int_\tX |u_\epsilon-u|\omega_1^n \xrightarrow{\epsilon\to 0} 0 .
\end{equation}
By \eqref{MA_approximation}, $\bfH_{g,\omega_\epsilon}(\vphi_\epsilon)$ is uniformly bounded. Then by H\"{o}lder-Young's inequality,
\begin{align}
\begin{split}
|\int_\tX (u_\epsilon-u) {g_{\vphi_\epsilon}}(\ddc\vphi_\epsilon)^n| = |\int_\tX (u_\epsilon-u)  \frac{(\ddc\vphi_\epsilon)^n}{\omega_\epsilon^n} {g_{\vphi_\epsilon}}  \omega_\epsilon^n |\\
\leq C |u_\epsilon-u|_{L^{\chi^*}}|\frac{(\ddc\vphi_\epsilon)^n}{\omega_\epsilon^n}|_{L^\chi} \leq C |\bfH_{g,\omega_\epsilon}(\vphi_\epsilon)| |u_\epsilon-u|_{L^{\chi^*}},
\end{split}
\end{align}
where $\chi(s) = (s+1)\log(s+1)-s$, $\chi^*(s) = e^s-s-1$. Since $\chi^*(s) \leq s e^s$, the inequality above implies
$|u_\epsilon-u|_{L^{\chi^*}}\leq \int_\tX |u_\epsilon-u|e^{|u_\epsilon-u|}\omega_1^n \leq C |u_\epsilon-u|_{L^1(\omega_1^n)}\xrightarrow{\epsilon\to 0} 0$, since $|u_\epsilon-u|$ is uniformly bounded.
In addition, $\bfJ_{g} \leq C \bfI_{g}$. Then we have as $\epsilon\to 0$, $0\leq \bfJ_{g,\vphi+\epsilon\vphi_P}(\vphi_\epsilon) \leq C \bfI_{g,\vphi+\epsilon\vphi_P}(\vphi_\epsilon) \rightarrow 0$.
Similarly, 
\begin{equation}
|\int_\tX (u_\epsilon-u){g_{\vphi_\epsilon}}(\ddc\vphi_\epsilon)^n| = |d_\epsilon||\int_\tX (u_\epsilon-u){g_{\vphi_\epsilon}}(\ddc\vphi)^n| 
\leq C \int_\tX |u_\epsilon-u|\omega_1^n \xrightarrow{\epsilon} 0 .
\end{equation}
Then we have
$|\bfE_{g,\vphi_{0,\epsilon}}(\vphi_\epsilon) - \bfE_{g,\vphi_{0,\epsilon}}(\vphi-\vphi_0+\vphi_{0,\epsilon})|\xrightarrow{\epsilon\to 0} 0$ .
Denote $\vphi_s = \vphi_0+su$. 
\begin{align}
\begin{split}
&|\bfE_{g,\vphi_{0,\epsilon}}(\vphi-\vphi_0+\vphi_{0,\epsilon})-\bfE_{g,\vphi_0}(\vphi)| = |\int_0^1\int_\tX u \big( {g_{\vphi_s+\epsilon\vphi_P}}(\ddc\vphi_s+\epsilon\vphi_P)^n - {g_{\vphi_s}}(\ddc\vphi_s)^n \big) ds|\\
&= |\int_0^1\int_\tX u\big(  ({g_{\vphi_s+\epsilon\vphi_P}}-{g_{\vphi_s}}) (\ddc(\vphi_s+\epsilon\vphi_P))^n + {g_{\vphi_s}}((\ddc(\vphi_s+\epsilon\vphi_P))^n - (\ddc\vphi_s)^n) \big) ds| \\
&\leq C \int_0^1\int_\tX |{g_{\vphi_s+\epsilon\vphi_P}}-{g_{\vphi_s}}|\omega_1^n ds +  |\int_0^1\int_\tX u {g_{\vphi_s}}((\ddc(\vphi_s+\epsilon\vphi_P))^n - (\ddc\vphi_s)^n) ds| \\ 
&\xrightarrow{\epsilon\to 0} 0 ,
\end{split}
\end{align}
where in the third line, we use the fact $u\in L^\infty(\tX,\omega_1)$, in the last line, we use the dominated convergence theorem.
Then $\lim_{\epsilon} \bfE_{g,\vphi_{0,\epsilon}}(\vphi_\epsilon) = \bfE_{g}(\vphi)$.

Let $g_j$ be the polynomial approximation of $g$. By \eqref{bfE_keq} and a similar argument as above, for any $\epsilon'>0$, there exists a $\epsilon_0>0$, such that
when $\epsilon<\epsilon_0$, $|\bfE_{g_j}^{\ddc\vphi_\epsilon^\eq}(\vphi_\epsilon)-\bfE_{g_j}^{\ddc\vphi^\eq}(\vphi)|<\epsilon'$.

Meanwhile, since the moment polytope is bounded, $\mathfrak{L}_{V_{\alpha}}(\vphi_{s,\epsilon})=\theta_{\alpha,\vphi_{s,\epsilon}}$ is uniformly bounded. By the $C^2$-convergence of $g_j$, we have $j_0>0$, for $j>j_0$,
$|\sum_\alpha ( f_{j,\alpha}\mathfrak{L}_{\xi_\alpha}(\vphi_{s,\epsilon})- f_\alpha \mathfrak{L}_{\xi_\alpha}(\vphi_{s,\epsilon}))|
\leq \sum_{\alpha} |D_\alpha f_j-D_\alpha f||\mathfrak{L}_{V_{\alpha}}(\vphi_{s,\epsilon})|<\epsilon'$. 
 In addition, $u_\epsilon$ is uniformly bounded. Then there exists a $C>0$,
such that $|\bfE_{g_j}^{\ddc\vphi_\epsilon^\eq}(\vphi_\epsilon)-\bfE_{g}^{\ddc\vphi_\epsilon^\eq}(\vphi_\epsilon)|< C\epsilon'$.
Then 
\begin{align*}
&|\bfE_{g}^{\ddc\vphi_\epsilon^\eq}(\vphi_\epsilon)-\bfE_g^{\ddc\vphi^\eq}(\vphi)| \leq \\
& 
|\bfE_{g}^{\ddc\vphi_\epsilon^\eq}(\vphi_\epsilon) - \bfE_{g_j}^{\ddc\vphi_\epsilon^\eq}(\vphi_\epsilon|
+|\bfE_{g_j}^{\ddc\vphi_\epsilon^\eq}(\vphi_\epsilon)-\bfE_{g_j}^{\ddc\vphi^\eq}(\vphi)| + |\bfE_{g_j}^{\ddc\vphi^\eq}(\vphi)-\bfE_g^{\ddc\vphi^\eq}(\vphi)|\\
& \xrightarrow{j\to\infty,\epsilon\to 0} 0
\end{align*}
\end{proof}

\begin{lemma}
\label{geodesic_t}
For any $t$, $u_\epsilon(t)$ is uniformly bounded (with respect to $t$ and $\epsilon$); $u_\epsilon(t)$ converges to $u(t)$ in $L^1(\tX,\omega_1^n)$ uniformly (with respect to $t$).
Furthermore, $\lim_{\epsilon\to 0} \bfE_{g,\vphi_{0,\epsilon}}(\vphi_\epsilon(t)) = \bfE_g(\vphi(t))$, 
$\lim_{\epsilon\to 0} \bfE_g^{\ddc\vphi(t)_\epsilon^\eq}(\vphi_\epsilon(t)) = \bfE_g^{\ddc\vphi(t)^\eq}(\vphi(t))$.
\end{lemma}
\begin{proof}
By \cite[Proposition 1.4]{DNG18},
\begin{equation}
\label{equicontinuous}
|u_\epsilon(t)-u_\epsilon(s)|\leq |u_\epsilon(1)-u_\epsilon(0)|(t-s)
\end{equation}
for $0\leq s\leq t\leq 1$, we have $u_\epsilon(t)$ is uniformly bounded with respect to $\epsilon,s$, and $u_\epsilon(t)$ is equicontinuous with respect to $t$. 
Then up to choose a subsequence,  as $\epsilon\to 0$, $u_\epsilon(t)$ converges to a limit $\hat{u}(t)$ in $L^1(\tX,\omega_1^n)$, uniformly with respect to $t$.
Since $\bfM_{\vphi_{0,\epsilon}}(\vphi_\epsilon(t))$ is convex along $u_\epsilon(t)$, $\bfM_{\vphi_{0,\epsilon}}(\vphi_\epsilon(i))$ are uniformly bounded for $i=0,1$, we have
$\bfM_{\vphi_{0,\epsilon}}(\vphi_\epsilon(t))$ is uniformly bounded with respect to $\epsilon$ and $t$.
From \eqref{E_omega_eq},\eqref{E_Ric_eq},\eqref{E_D},\eqref{E_psi}, using the fact that $u_\epsilon(t)$ is uniformly bounded and $f \in C^2(P)$,
it's not hard to check $\bfF_{g,\epsilon}(u_\epsilon(t))$ is uniformly bounded with respect to $\epsilon, t$. Then the entropy term $\bfH_{g,\omega_\epsilon}(\vphi_\epsilon(t))$ is also uniformly bounded.
As $\bfH_{g,\omega_\epsilon}$ and $u_\epsilon(t)$ are both uniformly bounded, we can adapt the same argument as in the proof of Lemma \ref{geodesic_ends} to show that
$\bfE_{g,\vphi_{0,\epsilon}}(\vphi_\epsilon(t))$ converges to $\bfE_{g}(\hat{\vphi}(t))$ both as a convex function and concave function over $t$.
Then $\bfE_{g}(\hat{\vphi}(t))$ is affine over $t$. This implies that $\hat{u}(t)=u(t)$.

By a similar argument as in the proof of Lemma \ref{geodesic_ends}, we also have
$\lim_{\epsilon\to 0} \bfE_g^{\ddc\vphi(t)_\epsilon^\eq}(\vphi_\epsilon(t)) = \bfE_g^{\ddc\vphi(t)^\eq}(\vphi(t))$.
\end{proof}

\begin{lemma}
In the formula \eqref{Hartog's}, for $0\leq t\leq 1$, $\lim_{\epsilon\to 0}\bfF_{g,\epsilon}(\vphi_\epsilon(t)) = \bfF_{g}(\vphi(t))$.
\end{lemma}
\begin{proof}
We will abbreviate $\vphi(t)$ by $\vphi$.
The convergence of $\bfE_{g}^{\ddc\vphi_\epsilon^\eq}(\vphi_\epsilon)$ is already shown in Lemma \ref{geodesic_t}.
Let $g_j$ be the polynomial approximation of $g$. 
Using the expression \eqref{bfE_Riceq}, by a similar argument as in the proof of Lemma \ref{geodesic_ends} and Lemma \ref{geodesic_t},
$\lim_{\epsilon\to 0} \bfE_{g_j}^{Ric_\epsilon^\eq}(\vphi_\epsilon) = \bfE_{g_j}^{Ric_\epsilon^\eq}(\vphi)$.
For any $\epsilon'>0$, by the $C^2$-convergence of $g_j$, there exists a $j_0>0$, such that for any $j>j_0$,
$|g_{j,\vphi_{s,\epsilon}}-g_{\vphi_{s,\epsilon}}|<\epsilon'$, $|V_{s,\epsilon,j}(\log(\omega_\epsilon^n))-V_{s,\epsilon}(\log(\omega_\epsilon^n))|<\epsilon'$,
where the second inequality is by the same argument as in the proof of Lemma \ref{geodesic_ends}.
In addition, $u_\epsilon$ is uniformly bounded. Then
\begin{align*}
|\bfE_{g_j}^{Ric_\epsilon^\eq}(\vphi_\epsilon)-\bfE_{g}^{Ric_\epsilon^\eq}(\vphi_\epsilon)|
&\leq C \epsilon' \int_0^1\int_\tX |Ric_\epsilon\wedge\frac{(\ddc\vphi_{s,\epsilon})^{n-1}}{(n-1)!}| + \frac{(\ddc\vphi_{s,\epsilon})^n}{n!} ds\\
&\leq C \epsilon'
\end{align*}
Then by the $3\epsilon$-argument as in the proof of Lemma \ref{geodesic_ends},
we have
$\lim_{\epsilon\to 0} \bfE_{g}^{Ric_\epsilon^\eq}(\vphi_\epsilon) = \bfE_{g}^{Ric^\eq}(\vphi)$.
Similarly, we can show the convergence of $\bfE_g^D(\vphi_\epsilon), \bfE_g^{\ddc\psi}(\vphi_\epsilon)$.
\end{proof}

By the convergence of $\bfF_{g,\epsilon}(\vphi_\epsilon(t))$ and the lower semi-continuity of $\bfH_{g,\omega_\epsilon}(\vphi_\epsilon)$, 
$\bfM (\vphi(t))\leq \lim_{\epsilon\to 0} \bfM_{\vphi_{0,\epsilon}}(\vphi_\epsilon(t))$. 
Then
\begin{align*}
\bfM(\vphi(t)) &\leq \lim_{\epsilon\to 0} \bfM_{\vphi_{0,\epsilon}}(u_\epsilon(t)) \\
&\leq \lim_{\epsilon\to 0} \big( (1-t) \bfM_{\vphi_{0,\epsilon}}(u_\epsilon(0)) + t \bfM_{\vphi_{0,\epsilon}}(u_\epsilon(1)) \big) \\
&\leq (1-t)\bfM(\vphi(0)) + t\bfM(\vphi(1)).
\end{align*}

Following the proof of \cite[Lemma 3.1]{BDL17}, we have
\begin{lemma}
\label{app_H}
For any $\vphi\in \mcE^1_T$, there exists a sequence of $\vphi_j\in \hat{\mcH}_T(X,L)$, such that
$\vphi_j$ converges to $\vphi$ under the strong topology of $\mcE^1_T$, and
$\bfH_{g,\Theta}(\vphi_j)$ converges to $\bfH_{g,\Theta}(\vphi)$.
\end{lemma}

Combine the approximation Lemma \ref{app_H} and the analysis above,
we can conclude that $\bfM(u(t))$ is convex along the geodesic $u(t)$.

\section{Non-Archimedean functionals and $\bG$-uniform stability}
\label{non-Archimedean functionals}

\subsection{Psh rays and Non-Archimedean metrics}
In the previous sections, we have studied the equation \eqref{MA} from the Archimedean point of view.
In the following sections, we will reconsider this question from the Non-Archimedean side. We will be brief in our discussion and only emphasize the key modifications in our generalized case. 


For a projective variety $X$, let $X^\NA$ be the Berkovich analytification of $X$ with respect to the trivial norm on $\bC$. Then $X^\NA$ is a compact, Hausdorff space.  The set of divisorial valuations $X^{\div}_\bQ$ is a dense subset of $X^{\NA}$.

Let $\bB = \{\tau\in \CC: |\tau|\leq 1\}$ be the unit disk in $\CC$.
We call a continuous map $\Phi=\{\vphi(t)\}: (\cdot, t)\in \RR_{>0} \rightarrow \mcE^1_T(X, L)$ a psh ray, if 
$\Phi(x, -\log|\tau|)$ is a psh metric on $X\times \bB^*$.
A psh ray $\Phi$ is said to have linear growth if 
\begin{equation}
\lim_{t\to +\infty} \frac{\sup_X (\vphi(t)-\vphi_0)}{t} =: \lambda_{\max}(\Phi) < +\infty.
\end{equation}
$\Phi$ is called sup-normalized if $\lambda_\max(\Phi)=0$.
Any sup-normalized a psh ray $\Phi$ extends to a psh metric on $X\times \bB$.
The upper bound of the growth rate guarantees the existence of a constant $a>0$, such that
$\Phi + a \log|\tau| < +\infty$.
$\Phi$ induces metric $\Phi^{\NA}$ on $L^\NA$. Let $\phi_\triv$ be the trivial metric on $L^\NA$. Then the relative potential $\Phi^\NA-\phi_\triv$ is represented by a function on $X^\div_\bQ$:
for any $v\in X^\div_\bQ$, 
$$(\Phi^\NA-\phi_\triv)(v) = -\sigma(v)(\Phi + a \log|\tau|)+ a$$
where $\sigma: X^\div_\bQ\rightarrow (X\times \CC)^\div_\bQ$ is the Gauss extension, 
and when $v = v_E$, where $E$ is a Weil divisor, $\sigma(v)(\Phi + a \log|\tau|)$ is the {
generic Lelong number
of $\Phi+a \log|\tau|$ at $E$.}

For a projective variety $X$ coupled with a line bundle $L$, a test configuration $(\mcX,\mcL)$
is a projective variety $\mcX$ coupled with a line bundle $\mcL$, with
\begin{enumerate}
\item[(1).] $\pi: \mcX\rightarrow \bB$ is a flat projective morphism, which is $\CC^*$-equivariant.
\item[(2).] $\mcL$ is a $\bC^*$-equivariant $\bQ$-line bundle.
\item[(2).] $(\mcX_t, \mcL_t)$ is isomorphic to $(X, L)$ when $t\neq 0$.
\end{enumerate}

For a test configuration $(\mcX,\mcL)$, by resolution of singularity, we can assume $(\mcX, \mcL)$ is dominating, i.e. there exists a $\bC^*$-equivariant birational morphism $\rho: \mcX \rightarrow X\times\bB$. Then $\mcL = \rho^* L + D$ for some $\bQ$-Cartier divisor $D$.
The test configuration $(\mcX,\mcL)$ induces a canonical Non-Archimedean metric $\phi_{(\mcX,\mcL)}$ on $L^\NA$, which
satisfies
\begin{equation}
(\phi_{(\mcX,\mcL)} - \phi_{\rm triv})(v) = \sigma(v) (D)
\end{equation}
for any $v\in X^\div_\bQ$.
We set:
\begin{equation}
\mcH^\NA(L):=\{\phi_{(\mcX,\mcL)}: (\mcX,\mcL) \text{ is a test configuration of } (X,L)\}.
\end{equation}

We call $\Phi$ a smooth psh ray associated to $(\mcX, \mcL)$ if $e^{-\Phi}$ extends to a smooth metric on $\mcL$.
Specifically, the pull-back of the Fubini-Study metric $\Phi_{\FS}$ on $\mcX$ is a smooth psh ray.

{
\begin{lemma}
For any $T_\CC\times \bG$-equivariant test configuration $\phi = \phi_{\mcX,\mcL}\in \mcH^\NA$,
there exists a smooth psh ray $\Phi$ associate to it, such that $\Phi$ is $T\times K$-invariant.
\end{lemma}
\begin{proof}
Since $\mcL$ is semi-ample over $\mcX$, there exists a smooth psh ray $\Phi$ on $\mcX$. Averaging 
$\Phi$ by $T\times K$-action. Then $\Phi$ becomes $T\times K$-invariant.
\end{proof}
}

\subsection{NA-functionals}
In this subsection, we will define the Non-Archimedean version of $\bfE_g, \bfJ_g, \bfD_g$ functionals.

Recall the definition of $\bfE^\NA,\bfJ^\NA, \bfL_\Theta^\NA$ for any $\phi=\phi_{(\mcX,\mcL)}\in \mcH^\NA$ (see \cite{BBJ18}):
\begin{align}
\bfE^\NA(\phi) &=\bfE^\NA(\mcL)= \frac{1}{\bV_1}\frac{\bar{\mcL}^{\cdot n+1}}{(n+1)!}=\frac{1}{L^{\cdot n}}\frac{\bar{\mcL}^{\cdot n+1}}{n+1}\label{eq-E1NA}\\
\Lam^\NA(\phi)&=\Lam^\NA(\mcL)=\frac{1}{\bV_1}\bar{\mcL}\cdot \frac{L^{\cdot n}_{\bP^1}}{n!}=\frac{1}{L^{\cdot n}}\bar{\mcL}\cdot L_{\bP^1}^{\cdot n}\\
\bfJ^\NA(\phi) &= \bfJ^\NA(\mcL)=\Lam^\NA-\bfE^\NA \\
\bfL_\Theta^\NA(\phi) &=\bfL_\Theta^\NA(\mcL)=\inf_{v\in X^\div_\bQ}(A_{(X,D+\Theta)}(v) + (\phi-\phi_\triv)(v)) ,
\end{align}
where $A_{(X, D+\Theta)}$ is the log discrepancy function: for any $v\in X^{\rm div}_\bQ$:
\begin{equation}
A_{(X, D+\Theta)}(v)=A_{(X,D)}(v)-v(\Theta).
\end{equation}

We will follow the conventions used in the proof of Proposition \ref{MA_eta}.
Our strategy to define non-Archimedean functionals is motivated by the method in the Archimedean case. 
We will first deal with the case when $g$ is a polynomial by using a fibration construction (see the proof of Proposition \ref{MA_eta}).
Then the functionals associated to a general $g$ can be defined as a limit of the functionals associated to polynomial $g$.

\subsection{Polynomial $g$}

First assume $g  = \prod_{\alpha=1}^r y_\alpha^{k_\alpha}$ to be a monomial.
We will use the notations in our proof of Proposition \ref{MA_eta}  and denote $\bfE_g = \bfE^\bvk $.
When $\vphi\in \mcH_T(X, L)$ we have the identity:
\begin{eqnarray}\label{eq-Egvk}
\bfE_g(\vphi)&=&\bfE^\bvk(\vphi):=\int_0^1 dt \frac{1}{\bV_g}\int_X \dot{\vphi} \prod_{i=1}^r\theta_\alpha(\vphi)^{k_i} \frac{(\ddc\vphi)^n}{n!}= \frac{1}{\bV_g}\int_0^1dt \int_{X^\bvk} \dot{\vphi}^\bvk \frac{(\ddc\vphi^\bvk)^{n+k}}{(n+k)!}\nonumber \\
&=&\frac{1}{\bV_g}\frac{1}{(n+k+1)!}\sum_{p=0}^{n+k} \int_{X^\bvk} (\vphi^\bvk-\vphi_0^\bvk) (\ddc \vphi_0^\bvk)^p\wedge ((\ddc\vphi)^\bvk)^{n+k-p},
\end{eqnarray}
where $k = |\vec{k}| = k_1+\cdots + k_r$.
The formula in the second equality is also well-defined for $\vphi \in \mcE^1_T(X, L)$. We will use this formula as the definition 
of $\bfE^\bvk$.

Using the fibration construction in section \ref{sec-gMA}, we set:
\begin{equation}\label{eq-mcXvk}
(\mcX^{[\vec{k}]}, \mcL^{[\vec{k}]}, \mcX_0^{[\vec{k}]})=(\mcX, \mcL, \mcX_0)\times \bS^{2k+1})/(S^1)^r
\end{equation}
where $(S^1)^r$ acts $\mcX$ vertically along the fibres of $\mcX\rightarrow \bC$.
Then similar to the non-Archimedean $\bfE$ functional in \eqref{eq-E1NA}, we set:
\begin{eqnarray}
(\bfE^\bvk)^\NA =\frac{1}{\bV_g} \frac{(\bar{\mcL}^\bvk)^{\cdot n+k+1}}{(n+k+1)!}.
\end{eqnarray}

As a consequence, $\bfE_g^\NA$ is also defined if $g $ is a polynomial.



\begin{lemma}
\label{poly_E}
For any normal test configuration $(\mcX,\mcL)$. Fix a smooth Hermitian metric $e^{-\vphi}$ on $\mcL$. Assume that $\eta$ is the holomorphic vector field generating the $\bC^*$-action and let $\theta_\eta(\vphi)$ be the Hamiltonian function of $\eta$. Then for any polynomial $g$, we have the identity:
\begin{equation}
\bfE_g^\NA(\mcL)=\frac{1}{\bV_g}\int_{\mcX_0} \theta_\eta(\vphi) g_\vphi \frac{(\ddc\vphi)^n}{n!}.
\end{equation}
\end{lemma}
\begin{proof}
We only need to show the equality for 
 $(\bfE^\bvk)^\NA$ when $g=\prod_\alpha \theta_\alpha^{k_\alpha}$ is a monomial. Consider the spaces defined in \eqref{eq-mcXvk}. Because $T\cong (S^1)^r$ acts vertically and commutes with the $\bC^*$-action,
we then have a $\bC^*$-equivariant morphism $\mcX^\bvk\rightarrow \bC$ such that the central fibre is given by $\mcX_0^\bvk$. We can assume that $\mcL^{\bvk}$ is relatively ample over $\bC$ (by adding a positive constant to $\theta$ as in section \ref{sec-gMA}). So we are reduced to the usual case of the $\bfE$-functional. 
By the well-known result for the usual $\bfE$-functional (see \cite[Proposition 3.12]{BHJ17}), we get the identity:

\begin{eqnarray*}
(\bfE^\bvk)^\NA (\mcL) &=& \frac{1}{\bV_g}\frac{(\bar{\mcL}^\bvk)^{\cdot n+k+1}}{n+k+1}\\
&=&\frac{1}{\bV_g}\int_{\mcX_0^{[\vec{k}]}}\theta_\eta(\vphi)^{[\vec{k}]} \frac{((\ddc\vphi)^\bvk)^{n+k}}{(n+k)!}\\
&=&\frac{1}{\bV_g}\int_{\mcX_0} \theta_\eta(\vphi)\prod_{\alpha}\theta_{\alpha}(\vphi)^{k_\alpha} \frac{(\ddc\vphi)^n}{n!}.
\end{eqnarray*}
\end{proof}

\subsection{Continuous $g$}
\begin{definition}
For any $\phi=\phi_{(\mcX, \mcL)}\in \mcH^\NA$, we choose a smooth psh metric $e^{-\vphi}$ on $\mcL|_{\mcX_0}$ and let $\eta$ denote the holomorphic vector field generating the $\bC^*$-action. We define:
\begin{equation}\label{eq-ENA}
\bfE^\NA_g(\phi):=\bfE^\NA_g(\mcL):=\frac{1}{\bV_g}\int_{\mcX_0} \theta_\eta(\vphi) g_{\vphi}\frac{(\ddc \vphi)^n}{n!}.
\end{equation}
\end{definition}
\begin{lemma}
The quantity on the right-hand-side of \eqref{eq-ENA} does not depend on the choice of $\vphi$.
\end{lemma}
\begin{proof}
Let $\DH_{T_\bC\times\bC^*}$ be the Duistermaat-Heckman measure associated to the $T_\bC\times\bC^*$-action on $(\mcX_0, \mcL|_{\mcX_0})$. Then the right-hand-side is given by:
\begin{equation}
\frac{1}{\bV_g}\int_{P\times\bR} \lambda\cdot g(y)  \DH_{T_\bC\times \bC^*} (\mcX_0, \mcL_0),
\end{equation}
where $\lambda$ and $y$ are the variable on $\bR$ and $P$ respectively. It is well-known that the DH measure and hence the above integral do not depend on the choice $\tilde{\vphi}$ (see \cite{BHJ17}).

\end{proof}

\begin{lemma}
Let $(\mcX, \mcL)$ be an 
ample normal test configuration.
For a continuous function $g$, let $g_i$ be a sequence of polynomials that converges to it over the moment polytope $P$.
Then
\begin{equation}
\lim_{i\to\infty} \bfE^\NA_{g_i}(\mcL)  = \bfE^\NA_g(\mcL) .
\end{equation}
\end{lemma}
\begin{proof}
By the definition of $\MA_g(\vphi)$, $\MA_{g_i}(\vphi)$ converges to $\MA_g(\vphi)$ as a measure.
By rescaling, without loss of generality we can assume that $\bV_{g_i}=\bV_g=1$.
Then
\begin{align}
\begin{split}
\lim_{i\to\infty}\bfE_{g_i}^\NA(\mcL) &= \lim_{i\to\infty} \int_{\mcX_0} \theta_\eta(\vphi) e^{g_{i,\vphi}}\frac{(\ddc\vphi)^n}{n!}\\
&= \lim_{i\to\infty} \sum_j e_j \int_{E_j} \theta_\eta(\vphi) \MA_{g_i}(\vphi) \\
&= \sum_j e_j \int_{E_j}\theta_\eta(\vphi) \MA_g(\vphi) =  \int_{\mcX_0} \theta_\eta(\vphi) g_\vphi \frac{(\ddc\vphi)^n}{n!} .
\end{split}
\end{align}
The first equality is by Lemma \ref{poly_E}, the third equality is by dominated convergence theorem.
\end{proof}

\begin{definition}
For any $\phi\in \mcH^\NA$, we define
\begin{equation}
\Lam^\NA_g(\phi)=\Lam^\NA(\phi)=\frac{1}{L^{\cdot n}}\bar{\mcL}\cdot L_{\bP^1}^{\cdot n}.
\end{equation}
We define the Non-Archimedean $\bfJ_g$ functional as
\begin{equation}
\bfJ_g^{\NA} = \Lam^\NA_g- \bfE_g^{\NA}
\end{equation}
and define the Non-Archimedean generalized Ding-functional as
\begin{equation}\label{eq-DNA}
\bfD^{\NA} = -\bfE^{\NA}_g + \bfL_\Theta^{\NA} .
\end{equation}
\end{definition}

\subsection{Slope formulas}
\begin{definition}
Let $\bfF$ be a functional on $\mcE^1_T(X,\omega)$. Let $\Phi$ be a psh ray of linear growth.
The slope at infinity of $\bfF$ along $\Phi$ is defined as
\begin{equation}
\bfF'^\infty(\Phi) =  \lim_{t\to +\infty}\frac{\bfF(\Phi(t))}{t}
\end{equation}
if the limit exists.
\end{definition}
\begin{proposition}
\label{F_slope}
Let $\phi = \phi_{(\mcX,\mcL)}\in \mcH^{\NA}(L)$ be a normal ample test configuration.
Let $\Phi(t)$ be a smooth psh ray associated to $(\mcX,\mcL)$.
Then
\begin{align}
\bfF'^\infty(\Phi)=\bfF^\NA(\phi).
\end{align}
where $t = -\log|\tau|$, $\tau\in \bB$, and $\bfF$ can be $\bfE_g,\bfJ_g,\bfD_g$.
\end{proposition}
\begin{proof}
We will first show the proof for $\bfE_g$. 
We first assume that $g =\prod_{\alpha=1}^r y_\alpha^{k_\alpha}$ is a monomial. Then 
\begin{equation}
\bfE_g(\vphi)=\bfE^{[\vec{k}]}(\vphi^{[\vec{k}]}).
\end{equation}
We can apply the same argument in \cite[Proof of Theorem 4.2]{BHJ19} to $(\mcX^{[\vec{k}]}, \mcL^{[\vec{k}]})$ to get
\begin{eqnarray*}
\label{E_slope}
\lim_{t\rightarrow+\infty}\frac{\bfE_g(\vphi(t))}{t}&=&\lim_{t\to +\infty}\frac{\bfE^{[\vec{k}]}(\vphi(t))}{t}=\frac{1}{\bV_g}\frac{(\bar{\mcL}^{[\vec{k}]})^{n+k+1}}{(n+k+1)!}=(\bfE^{[\vec{k}]})^\NA(\phi).
\end{eqnarray*}


For general continuous $g$, by Stone-Weierstrass theorem, we can find a sequence of polynomials $g_i$ which converges to $g $ uniformly over the moment polytope $P$. For simplicity of notations, we assume $\bV_{g_i}=\bV_g=1$. 
It's easy to see for each fixed $\vphi$, 
\begin{equation}
\lim_{i\rightarrow+\infty} \bfE_{g_i}(\vphi)=\bfE_g(\vphi).
\end{equation}
For any $\epsilon>0$, there exists an $i_0>0$, such that for any indices $i,i'>i_0$, $|g_i - g_{i'}|<\epsilon$.
Then
\begin{align}
\label{bound_1}
\begin{split}
&|\frac{(\bfE_{g_i}-\bfE_{g_{i'}})(\vphi(t))}{t}| \\
&\leq \frac{1}{t}|\int_0^t\int_X(\dot{\vphi}-C+C) (g_i-g_{i'}) \frac{(\ddc\vphi)^n}{n!}ds|+ \frac{1}{t}|(\bfE_{g_i}-\bfE_{g_{i'}})(\vphi(0))|\\
&\leq \frac{\epsilon}{t}\big( |\int_0^t\int_X(\dot{\vphi}-C)  \frac{(\ddc\vphi)^n}{n!}ds|+Ct\big) + \frac{C'\epsilon}{t}\\
&\leq  \epsilon \big( (|\bfE^\NA-C|+C)+\frac{C'}{t}\big)  ,
\end{split}
\end{align}
where we have used $\dot{\vphi}-C\leq 0$ for some constant $C$ in the third line.


Then we have:
\begin{eqnarray*}
\lim_{t\rightarrow+\infty} \frac{\bfE_{g}(\vphi(t))}{t}&=&\lim_{t\rightarrow+\infty}\lim_{i\rightarrow+\infty} \frac{\bfE_{g_i}(\vphi(t))}{t}\\
&=&\lim_{i\rightarrow+\infty}\lim_{t\rightarrow+\infty} \frac{\bfE_{g_i}(\vphi(t))}{t}=\lim_{i\rightarrow+\infty} \bfE^\NA_{g_i}(\phi)\\
&=&\frac{1}{\bV_g}\lim_{i\rightarrow+\infty}\int_{\mcX_0}\theta e^{f_i(\tilde{\vphi})} (\ddc\tilde{\vphi})^n\\
&=&\frac{1}{\bV_g}\int_{\mcX_0}\theta_\xi(\tilde{\vphi}) e^{f_{\tilde{\vphi}}} (\ddc\tilde{\vphi})^n = \bfE^\NA_g.
\end{eqnarray*}
We can exchange the limits in the second line because of the estimate \eqref{bound_1}.



To show the slope formula for $\bfJ_g$, we only need to show that
\begin{equation}
\lim_{t\to\infty} \frac{1}{t}\int_X (\vphi(t)-\vphi_0)e^{f_{\vphi_0}}\frac{(\ddc\vphi_0)^n}{n!} = \Lam^\NA_g(\phi)=\Lam^\NA(\phi).
\end{equation}
Since $\Phi(t)$ has linear growth, we have $\lim_{t\to\infty}(\sup_X\frac{\vphi(t)-\vphi_0}{t}) = \lambda_{\rm max}(\Phi)<+\infty$.
By Hartogs' compactness lemma, up to a subsequence, $\vphi(t)-\vphi_0-\sup_X(\vphi(t)-\vphi_0)$ converges weakly to a function in $\mcE^1_T(X,\omega)$.
Then
\begin{eqnarray*}
\label{lambda_max}
&&\lim_{t\to\infty} \frac{1}{t}\int_X (\vphi(t)-\vphi_0)e^{f_{\vphi_0}}\frac{(\ddc\vphi_0)^n}{n!} \\
&=& \lambda_{\rm max} + 
\lim_{t\rightarrow+\infty}\frac{1}{t}\int_X (\vphi(t)-\vphi_0-\sup_X(\vphi(t)-\vphi_0))e^{f_{\vphi_0}}\frac{(\ddc\vphi_0)^n}{n!}\\
&=& \lambda_{\rm max}.
\end{eqnarray*}
On the other hand, it is known that $\lambda_{\max}(\Phi)=\Lam^\NA(\phi)$.

Since $\bfD_g = -\bfE_g + \bfL_\Theta$, the slope formula for $\bfD_g$ follows from the proof for $\bfE_g$ above and 
the proof for $\bfL_\Theta$ as in \cite{BBJ15}.
\end{proof}
We have the following non-Archimedean version of Ding's inequality (\cite[Lemma 6.17]{BoJ18a}).
\begin{corollary}
For any $t\in [0,1]$ and $\phi\in \mcH^{\NA}$,
we have the inequality:
\begin{equation}
\bfJ^\NA_g((1-t)\phi_\triv+t\phi)\le t^{1+\frac{1}{C}}\bfJ^\NA_g(\phi).
\end{equation}
\end{corollary}
\begin{proof}
This follows from the Archimedean inequality \eqref{J_t} and the slope formula in the above proposition.
\end{proof}
Based on the previous discussion, it is convenient to introduce the following:
\begin{definition}\label{def-intg}
Let $(\mcX, \mcL)$ be a test configuration of $(X, L)$, we define:
If $g =\prod_{\alpha=1}^r y_\alpha^{k_\alpha}$, then
\begin{equation}
\frac{1}{p!}\left(\bar{\mcL}^{\cdot p}\cdot L_{\bP^1}^{\cdot n+1-p}\right)_g:=\frac{1}{(k+p)!}(\bar{\mcL}^{[\vec{k}]})^{\cdot p} \cdot (L_{\bP^1}^{[\vec{k}]})^{\cdot n+1-p} .
\end{equation}
Moreover, if $\Delta$ is a $T_\bC\times\bC^*$-invariant divisor on $\mcX$ and if $\Delta^{[k]}=(\Delta\times \bS^{[\vec{k}]})/(S^1)^r$ is the associated invariant divisor on $\mcX^{[\vec{k}]}$, then we set:
\begin{equation}
\frac{1}{p!} (\Delta \cdot \bar{\mcL}^p\cdot L_{\bP^1}^{\cdot n-p})_g:=\frac{1}{(k+p)!} \Delta^{[\vec{k}]}\cdot (\bar{\mcL}^{[\vec{k}]})^{\cdot p} \cdot (L_{\bP^1}^{[\vec{k}]})^{\cdot n+1-p} .
\end{equation}

For a general continuous $g$ defined on $P$, we choose a sequence of polynomials $g_i$ that converges uniformly to $g $ and define:
\begin{equation}
\left(\bar{\mcL}^{\cdot p}\cdot L_{\bP^1}^{\cdot n+1-p}\right)_g=\lim_{i\rightarrow+\infty}\left(\bar{\mcL}^{\cdot p}\cdot L_{\bP^1}^{\cdot n+1-p}\right)_{g_i}.
\end{equation}
\begin{equation}
(\Delta \cdot \bar{\mcL}^p\cdot L_{\bP^1}^{\cdot n-p})_g=\lim_{i\rightarrow+\infty} (\Delta \cdot \bar{\mcL}^p\cdot L_{\bP^1}^{\cdot n-p})_{g_i}.
\end{equation}
\end{definition}
\begin{remark}
During our completion of this paper, we noticed a preprint of Inoue \cite{Ino20} in which a more general framework of equvariant intersections is used to define an equivariant version of K-stability adapted to the usual K\"{a}hler-Ricci solitons.
\end{remark}

Recall $D = \mcL-\rho^* L_\bC$.
We can use the same argument for $\bfE'^\infty=\bfE^\NA$ to get:
\begin{equation}
\frac{1}{n!}(D\cdot \mcL^{\cdot n})_g:=D^\bvk\cdot \frac{(\mcL^\bvk)^{\cdot n+k}}{(n+k)!}= \lim_{t\to\infty}\frac{1}{t} \int_{X^\bvk} (\vphi^\bvk-\vphi_0^\bvk) \frac{((\ddc\vphi)^\bvk)^{n+k}}{(n+k)!}
\end{equation}
\begin{proposition}
\label{bfII}
For any $\phi=\phi_{(\mcX, \mcL)}$ we have
\begin{equation}
\label{bfI_f^NA}
\bfI_g^\NA(\phi) = \Lam^\NA_g - (D\cdot \mcL^{\cdot n})_g.
\end{equation}
Furthermore, we have the slope formula for $\bfI_g^\NA$
\begin{equation}
\bfI_g^\NA(\mcL) = 
\bfI_g'^\infty(\Phi) .
\end{equation}
\end{proposition}
\begin{proof}
As shown in formula \eqref{lambda_max},
$\Lam = \lim_{t\to +\infty} \frac{1}{t}\int_X (\vphi-\vphi_0)e^{f_{\vphi_0}}\frac{(\ddc\vphi_0)^n}{n!}$.

Again, by the fact $\Phi$ is a smooth psh ray, there exists a $C>0$, such that $\vphi(t)-\vphi_0< C(t+1)$.
Then we have
For $\epsilon>0$, there exists $i_0\in \NN$, such that for any $i,i'>i_0$, $|g_i-g_{i'}|<\epsilon$.
\begin{align}
\label{bound_2}
\begin{split}
&\frac{1}{t}|\int_X(\vphi(t)-\vphi_0)(g_{i,\vphi_0}-g_{i',\vphi_0})| \frac{(\ddc\vphi_0)^n}{n!} \leq \\
& \frac{\epsilon}{t}\Big( |\int_X(\vphi(t)-\vphi_0-C(t+1))\frac{(\ddc\vphi)^n}{n!}|
+ |\int_X C (t+1)\frac{(\ddc\vphi)^n}{n!} | \Big)\\
&\leq \epsilon \big( |(D\cdot \mcL^{\cdot n})_g- C'| + C' \big) .
\end{split}
\end{align}
Take $t\to\infty$, we can see $|(D\cdot \mcL^{\cdot n})_{g_i} - (D\cdot \mcL^{\cdot n})_{g_{i'}}| < C'' \epsilon$,
which implies $\bfI^\NA_g(\phi) = \lim_{i\to\infty}\bfI^\NA_{g_i}(\phi)$ and is well-defined.
Furthermore,
\begin{align}
\begin{split}
\bfI_g'^\infty(\Phi) &= \lim_{t\to\infty}\frac{1}{t} \bfI_g(\vphi(t)) \\
&= \lim_{t\to\infty}\lim_{i\to\infty} \frac{1}{t} \bfI_{g_i} (\vphi(t))  \\
&= \lim_{i\to\infty} \lim_{t\to\infty}\frac{1}{t} \bfI_{g_i} (\vphi(t)) = \lim_{i\to\infty} \bfI^\NA_{g_i} (\phi) = \bfI^\NA_g(\phi) ,
\end{split}
\end{align}
where the third equality is because of \eqref{bound_2} and dominated convergence theorem.

\end{proof}

Set
\begin{equation}
\mcH^\NA_{T\times K}=\{\phi=\phi_{(\mcX, \mcL)}; (\mcX, \mcL) \text{ is a $T_\bC\times \bG$ equivariant test configuration of } (X, L)\}.
\end{equation}

In the following, we will abbreviate $\phi_{(\mcX,\mcL)}$ as $\phi$ when there is no ambiguity.
From now on, let $\bG$ be a reductive complex Lie group of $\Aut_T(X, D, \Theta)$ (see Definition \ref{def-AutT}). Assume $\bG=K_\bC$ and let $\bT$ be the center of $\bG$. Then $\bT\cong (\bC^*)^\fr$. We will denote $N_\bZ={\rm Hom}(\bC^*, \bT), N_\bR=N_\bZ\otimes_\bZ\bR\cong \bR^\fr$.

\begin{definition}[see \cite{His16b}]
\label{xi_twist}
For any $\xi\in N_\bR$ and a $(T_\bC\times \bG)$-equivariant test configuration $(\mcX,\mcL)$, the $\xi$-twisted test configuration
$(\mcX,\mcL)_\xi$ is defined as the following.
Let $\zeta$ be the holomorphic vector field that generates the $\CC^*$-action of $(\mcX,\mcL)$.
Then $(\mcX,\mcL)_\xi$ is the $\bT$-equivariant test configuration with the holomorphic vector field
$\zeta+\xi$. If $\phi=\phi_{(\mcX, \mcL)}$, we also denote $\phi_\xi=\phi_{(\mcX_\xi, \mcL_\xi)}$.

For any $\phi\in \mcH^\NA_{T\times K}$, set 
\begin{equation}
\bfJ^\NA_{g,\bT}(\phi)=\inf_{\xi\in N_\bR}\bfJ^\NA_{g,\bT}(\phi_\xi).
\end{equation}

\end{definition}

\begin{proposition}
\label{trivial}
For any $\xi\in N_\bR$, it induces a test configuration $\phi_\xi$ with $\CC^*$-action genearted by $\xi$.
Then $\bfJ_g^{\NA}(\phi_\xi) = 0$ if and only if $\xi = 0$.
\end{proposition}
\begin{proof}
The ``if" direction is straightforward. We need to show the ``only if" direction.
By \eqref{J}, 
we have
\begin{equation}
\frac{1}{C} \leq \lim_{t\to\infty}\frac{\bfJ(\sigma_\xi(t)^*\vphi)}{t} \leq \lim_{t\to\infty}\frac{\bfJ_g(\sigma_\xi(t)^*\vphi)}{t} 
\leq C \lim_{t\to\infty}\frac{\bfJ(\sigma_\xi(t)^*\vphi)}{t}  .
\end{equation}
Then $\bfJ_g'^\infty(\Phi_\xi) = 0$ implies $\bfJ'^\infty(\Phi_\xi)=\bfJ^\NA(\phi_\xi)= 0$.
By \cite{BHJ17}, this implies $\phi_\xi=0$ and $\xi=0$.
\end{proof}

We also have the following equivariant version of slope formula.
\begin{proposition}
\label{J_trivial}
Let $\phi=\phi_{(\mcX, \mcL)}$ be an $T_\CC\times \bG$-equivariant ample normal test configuration.
Let $\Phi$ be a $T\times K$-invariant smooth psh ray associated to $(\mcX,\mcL)$.
Then
\begin{align}
\bfJ_{g,\bT}^{\NA}(\phi) = \bfJ'^\infty_{g,\bT}(\Phi) .
\end{align}
\end{proposition}
\begin{proof}
LHS $\geq$ RHS can be seen from the definition.
We need to show LHS $\leq$ RHS.
We have $\inf_{\xi\in N_\bR}\bfJ_g(\sigma_{t\xi}^*\Phi(t)) = \bfJ_g(\sigma_{t\xi_t}^*\Phi(t)))$.
First, we will show that $|\xi_t| \leq C $ for some $C>0$. This is deduced from the quasi-triangle inequality
\begin{align*}
\bfJ_g(\sigma_{t\xi_t}^*\vphi_0) &\leq C (\bfJ_g(\sigma_{t\xi_t}^*\Phi(t)) + \sigma_{t\xi_t}^*\bfJ_{g,\Phi(t)}(\vphi_0))\\
&\leq C(\bfJ_{g}^{\NA}(\phi)\cdot t + o(t)) .
\end{align*}
Since $|\xi_t|<C$, we can choose a sequence such that $\xi_{t_j}$ converges to $\xi_\infty$.
We need to show that $\frac{1}{t_j}|\bfJ_g(\sigma_{\xi_{t_j}}(t_j)^*\Phi(t_j)) - \bfJ_g(\sigma_{\xi_\infty}(t_j)^*\Phi(t_j))|\to 0$.
It suffices to show
\begin{equation}
\frac{1}{t_j}|\sigma_{t_j\xi_{t_j}}^*\Phi(t_j) - \sigma_{t_j\xi_\infty}^*\Phi(t_j)|\to 0 ,
\end{equation}
since 
\begin{align*}
|\bfJ_g(\sigma_{t_j\xi_{t_j}}^*\Phi(t_j)) - \bfJ_g(\sigma_{t_j\xi_\infty}^*\Phi(t_j))| &\leq C \bfI(\sigma_{t_j\xi_{t_j}}^*\Phi(t_j),
\sigma_{t_j\xi_\infty}^*\Phi(t_j)) \\
&\leq C' |\sigma_{t_j\xi_{t_j}}^*\Phi(t_j)- \sigma_{t_j\xi_\infty}^*\Phi(t_j)| .
\end{align*}
Choose $q\in \NN$ such that sections of $q\mcL$ embed $\mcX$ into a projection space.
It turns out that it is equivalent and more convenient to work with the Fubini-Study metric $\vphi_\FS$ instead of $\vphi$.
Let $s_0,\cdots,s_{N_q}$ be a basis of $H^0(X,qL)$. Denote $\zeta$ as the vector field that induces the $\CC^*$-action of
the test configuration. Denote the weight of $\zeta$ on $s_i$ as $\lambda_i$; denote the weight of $\xi$ on $s_i$ as $\la\xi, a_i\ra$.
Then 
\begin{align*}
\frac{1}{t_j}|\sigma_{\xi_{t_j}}(t_j)^*\vphi_\FS(t_j) - \sigma_{\xi_{\infty}}(t_j)^*\vphi_\FS(t_j)|&=
\frac{1}{t_j}|\log(\frac{\sum_i \tau_j^{-2(\lambda_i + \la\xi_{t_j},a_i\ra)}|s_i|^2_\vphi}{\sum_i \tau_j^{-2(\lambda_i + \la\xi_\infty,a_i\ra)}|s_i|^2_\vphi})| \\
& \leq |\la\xi_{t_{i_0}}-\xi_{t_\infty}, a_{i_0}\ra| + o(1) \\
&\leq C |\xi_{t_j} - \xi_\infty| + o(1) ,
\end{align*}
where $t_j = -\log|\tau_j|^2$, $i_0$ is the index such that $\lambda_{i_0}+\la\xi_\infty, a_{i_0}\ra = \sup_i \lambda_i + \la\xi_\infty,a_i\ra$.
Then the proof is concluded.
\end{proof}

\begin{proposition}
\label{J_product}
Let $(\mcX,\mcL)$ be a $T_\CC\times \bG$-equivariant ample normal test configuration.
Then $\bfJ^{\NA}_{g,\bT}(\mcL) = 0$ if and only if $(\mcX,\mcL)$ is induced by $\sigma_{\xi}$, where $\xi\in N_\bR$.
\end{proposition}
\begin{proof}
The ``if" direction is straightforward. We need to show the ``only if" direction.
Let $\Phi$ be a smooth psh ray associated to $(\mcX,\mcL)$.
As shown in the proof of Proposition \ref{J_trivial}, there exists a $\xi\in N_\bR$, such that
$\bfJ^\NA_{g,\bT}(\mcL) = \bfJ^\NA_g(\mcL_\xi) = \lim_{t\to +\infty} \frac{\bfJ_g(\sigma_{t\xi}^*\Phi(t))}{t}$.
We may choose $\Phi$ be the pull-back of the Fubini-Study metric by the embedding $q\mcL$ for some $q\in \NN$.
Let $\{\lambda_0,\cdots,\lambda_{N_q}\}$, $\{\mu_0,\cdots,\mu_{N_q}\}$ be eigenvalues of the $\CC^*$-action and $\xi$-action on $H^0(X,qL)$
accordingly.
Then 
$\bfJ^\NA(\mcL_\xi) = \max_{0\leq i\leq N_q}(\lambda_i+\mu_i) - \frac{1}{N_q+1} \sum_{i=0}^{N_q} (\lambda_i + \mu_i) = 0$.
Consequently, $\lambda_i + \mu_i = \lambda_j +\mu_j$, and $(\mcX_\xi,\mcL_\xi)$ is a trivial test configuration.
Thus $(\mcX,\mcL)$ is test configuration induced by $\sigma_{-\xi}$.
\end{proof}







\subsection{$\bG$-uniform stability and valuative criterion}
We introduce:
\begin{definition}
Let $\bG$ be a connected reductive subgroup of $\Aut_T(X, D, \Theta)$ and $\bT$ be its center.
$\bX=(X, D+\Theta, T)$ is $\bG$-uniformly $g$-Ding-stable if there exists $\gamma>0$ such that for any $T_\CC\times\bG$-equivariant test configuration $\phi\in \mcH^\NA(L)$, we have:
\begin{eqnarray*}
\bfD^\NA(\phi)\ge \gamma\cdot  \bfJ^\NA_{g,\bT}(\phi).
\end{eqnarray*}
\end{definition}
In this section, we will explain a valuative criterion for the $\bG$-uniform $g$-Ding-stability for $\bX$ which generalizes the results in \cite{Fuj19a,Li19}. Since the proof is similar to the usual case as considered in \cite{Li19} which generalizes the argument in \cite{BBJ18}. We will only emphasize the key points and leave the details to the reader.

For simplicity of notations, we assume that $L$ is Cartier. See \cite[2.4.2]{Li19} for modifications of notations when $L$ is only $\bQ$-Cartier.
Let $W_\bullet=\{W_m\}$ be a $T$-invariant graded linear series (\cite[2.4.A]{Laz04}). We first define a $g$-weighted volume denoted by $\vol_g(W_\bullet)$. For this, we decompose
\begin{equation}
W_m=\bigoplus_{\alpha\in M_\bZ}W_{m,\alpha}\subset H^0(X, mL)=:R_m,
\end{equation}
where $W_{m,\alpha}=\{s\in W_m; t\circ s=t^\alpha\cdot s\}$.
Then we define:
\begin{equation}
\vol_g(W_\bullet)=\lim_{m\rightarrow+\infty} \frac{n!}{m^n} \sum_\alpha g(\frac{\alpha}{m})\dim W_{m,\alpha}.
\end{equation}
The existence of this limit can be proved using the theory of Newton-Okounkov bodies.

Let $\mcF=\{\mcF^{\lambda} R_m\}$ be a $T$-invariant filtration as considered in \cite{BC11} (see \cite[2.4.2]{Li19}). 
For each $\alpha\in M_\bZ$, let $\lambda^{(m, \alpha)}_1\ge \lambda^{(m, \alpha)}_2\ge \cdots \ge \lambda^{(m, \alpha)}_{N_{m,\alpha}}$ be the successive minima of the filtration $\mcF^\lambda R_{m,\alpha}$. 

Set
\begin{eqnarray}
f_m(\lambda)&=&\frac{n!}{m^n}\sum_{\lambda^{(m,\alpha)}_k\ge \lambda } g(\frac{\alpha}{m})\dim (\cF^{m\lambda}R_{m,\alpha}), \\
\nu_m&:=&-d f_m(\lambda)=\frac{n!}{m^n}\sum_{\alpha, k}g(\frac{\alpha}{m})\delta_{\frac{\lambda^{(m,\alpha)}}{m}}.
\end{eqnarray}
On the other hand, we set $\cF^{(\lambda)}=\{\cF^{m\lambda}R_m\}$. Then as $m\rightarrow+\infty$, $f_m(\lambda)\rightarrow \vol_g(\cF^{(\lambda)})$. So arguing as in \cite{BC11}, we know that $\nu_m$ converges to the measure:
\begin{equation}
\DHM_g(\cF):=\frac{1}{\bV_g}\left(-d\vol_g(\cF^{(t)})\right).
\end{equation}

Let $\phi=\phi_{\mcF}$ be the (continuous) non-Archimedean metric associated to $\mcF$. Then we define:
\begin{eqnarray}
\bfE^\NA_g(\phi_\mcF)&=&\frac{1}{\bV_g}\int_\bR \lambda (-d \vol_g(\cF R^{(\lambda)}_\bullet)).
\end{eqnarray}
For any $v\in (X^{\rm div}_\bQ)^T$, it defines a filtration $\mcF_v$ by setting:
\begin{equation}
\mcF_v^{\lambda} R_m=\{s\in R_m; v(s)\ge \lambda\}.
\end{equation}
In this case, by integration by parts, we known that:
\begin{equation}
\bfE^\NA(\phi_{\mcF_v})=\frac{1}{\bV_g}\int_0^{+\infty} \vol_g(\mcF R^{(\lambda)}_\bullet)d\lambda= S_g(v)
\end{equation}
where we introduced:
\begin{equation}\label{eq-Sgv}
S_g(v):=S_{L,g}(v):=\frac{1}{\bV_g}\int_0^{+\infty} \vol_g(\mcF R^{(\lambda)}_\bullet)d\lambda.
\end{equation}
We have the following valuative criterion:
\begin{theorem}\label{thm-valDing}
$(X, D+\Theta)$ is $\bG$-uniformly $g$-Ding-stable if and only if there exists $\gamma>1$ such that for any $v\in (X^{\rm div}_\bQ)^T$, there exists $\xi\in N_\bR$ satisfying:
\begin{equation}
A_{D+\Theta}(v_\xi)-\gamma \cdot S_g(v_\xi)\ge 0.
\end{equation}
\end{theorem}

It is clear that we can adopt the proof in \cite[Corollary 3.2]{Li19} and \cite[4.2]{Li19} to prove the valuative criterion for $\bG$-uniform $g$-Ding-stability without using the MMP program. In each step of the argument, we just need to replace $A_{(X,D)}$ by $A_{(X, D+\Theta)}$ and replace $S(v)$ by $S_g(v)$. So we omit the details of the argument, except for recording the following useful lemma in its proof, which generalizes a corresponding inequality in \cite{BBJ18,Li19}. Note that the proof in \cite[4.2]{Li19} does not need the solution of non-Archimedean Monge-Amp\`{e}and Legendre duality as used in \cite{BBJ18}.

\begin{lemma}\label{lem-NAlegend}
Given any  $\phi=\phi_{(\mcX,\mcL)}$ for a $T_\CC\times \bG$-equivariant semi-ample test configuration
\begin{equation}\label{eq-SL2bfE}
\inf_{v\in (X^{\rm div}_\bQ)^{\bG}} (S_g(v)+(\phi-\phi_\triv)(v))\ge \inf_{v\in X^{\rm div}_\bQ}(S_g(v)+(\phi-\phi_\triv)(v))\ge  \bfE^\NA_g(\phi).
\end{equation}
\end{lemma}

\section{Yau-Tian-Donaldson conjecture for twisted K\"{a}hler-Ricci $g$-solitons}
\label{stable_and_proper}


\begin{lemma}
\label{E_twist}
Let $(\mcX,\mcL)$ be a $T_\CC\times\bG$-equivariant normal ample test configuration. Let $\xi$ be a holomorphic vector field in $N_\bR$,
$(X_\CC,L_{\bC,\xi})$ be a product test configuration induced by $\xi$ using the canonical lifting of $\xi$ (see \eqref{eq-canlift}). Then
\begin{equation} 
\bfE^\NA_g(\mcL_\xi) = \bfE^\NA_g(\mcL) + \bfE^\NA_g(L_{\bC,\xi})
\end{equation} 
\begin{equation} 
\bfL_\Theta^\NA(\mcL_\xi) = \bfL_\Theta^\NA(\mcL)
\end{equation} 
and 
\begin{equation}\label{eq-Fut2}
\Fut_g(\xi)=-\bfE^\NA_g(L_{\bC,\xi}). 
\end{equation}
\end{lemma}
\begin{proof}
We will use the pull-back of Fubini-Study metric as a $T\times K$-invariant smooth psh ray $\Phi$ associated to $(\mcX,\mcL)$.(details can be found in the proof of Proposition \ref{J_trivial}.)
The advantage of using this metric is, we can express  $\vphi(t) = \sigma_{t\lambda}^*\vphi$, where $\sigma_{t\lambda}$ is the $C^*$-action of the test configuration.
Let $\Phi_\xi$ be the $\xi$-twisted psh ray. Then $\phi_\xi(t) = \sigma_{t\xi}^*\sigma_{t\lambda}^*\vphi = \sigma_{t\xi+t\lambda}^*\vphi$.
Then
\begin{align}
\begin{split}
\bfE^\NA_g(\mcL_\xi) &= \bfE'^\infty_g(\Phi_\xi)\\
&= \lim_{t\to\infty}\frac{\bfE_g(\vphi_\xi(t))}{t}\\
&= \lim_{t\to\infty}\frac{\bfE_g(\vphi_\xi(t)-\bfE_g(\vphi(t))}{t} + \lim_{t\to\infty}\frac{\bfE_g(\vphi(t))}{t} \\
&= \lim_{t\to\infty}\frac{\bfE_{g,\sigma_{t\lambda}^*\vphi}(\sigma_{t\lambda}^*\sigma_{t\xi}^*\vphi)}{t} + \bfE^\NA_g(\mcL)\\
&= \lim_{t\to\infty}\frac{\bfE_g(\sigma_{t\xi}^*\vphi)}{t} + \bfE^\NA_g(\mcL)\\
&= \bfE^\NA_g(L_{\bC,\xi}) + \bfE^\NA_g(\mcL)
\end{split}
\end{align}
The identity for $\bfL_\Theta^\NA(\mcL)$ can be proved using the same argument as \cite[Proposition 3.3]{Li19}.

Moreover, the identity \eqref{eq-Fut2} follows from the identity $\bfE^\NA_g(L_{\bC,\xi})=\int_X \theta_{\xi,\vphi}\MA_g(\vphi)$ and the formula \eqref{eq-Futform} for the generalized Futaki-invariant.

\end{proof}

\begin{lemma}
\label{destablizer}
Let $\bG$ be the reductive Lie group defined before. Assume $\bfM$ is not $\bG$-coercive. 
Then for any $\vphi\in \mcE^1_{T\times K}(X,L)$,
there exists a $T\times K$-invariant geodesic ray $\Phi$ emanating from $\vphi$ satisfying:
\begin{enumerate}
\item We have the normalization:
\begin{equation}\label{eq-Phinorm}
\sup(\vphi(t)-\vphi_0)=0, \quad \bfE(\vphi(t))=-t.
\end{equation}
\item
\begin{equation}\label{eq-Phidec}
\bfM'^\infty(\Phi)\le 0.
\end{equation}
\item
\begin{equation}\label{eq-Jxilb}
\inf_{\xi\in N_\bR} \bfJ'^\infty(\Phi_\xi)=1.
\end{equation}
\end{enumerate}

\end{lemma}
\begin{proof}
Assume $\bfM$ is not $\bG$-coercive. Let $\vphi$ be an arbitary metric in $\mcE^1_{T\times K}(X,L)$.
Then by the assumption and Lemma \ref{J_convex}(ii), 
There exists a sequence of $\vphi_j\in \mcE^1_{T\times K}$, where $\bfJ_g(\vphi_j) = \inf_{\sigma\in \bT}\bfJ_g(\sigma^*\vphi_j)$, such that
$\bfM(\vphi_j)\leq \delta_j \bfJ_g(\vphi_j) - C_j$, where $\delta_j \to 0$, and $C_j\to \infty$.
Since the entropy $\bfH_{g,\Theta}(\vphi_j)$ has a lower bound, we have $\bfJ_g(\vphi_j)\to \infty$.
Let $\Phi_j(t)$ be the geodesic ray, that emanates from $\vphi$ and passes through $\vphi_j$. Denote the distance between $\vphi$ and $\vphi_j$
by $T_j$. Then $\Phi_j(T_j) = \vphi_j$.
By the convexity of $\bfM$ (Section \ref{convex_singular}), for $t\in [0,T_j]$, we have
$\bfM(\Phi_j(t)) \leq (1-\frac{t}{T_j}) \bfM(\vphi)+ (\delta_j\frac{\bfJ_g(\vphi_j)}{T_j}-\frac{C_j}{T_j})t$.

We can assume $\sup(\vphi_j-\vphi_0)=0$. As $j\to\infty$, up to choosing a subsequence, $\Phi_j$ converges weakly to a geodesic ray $\Phi$. 
From our construction, we can see that $\Phi$ is not an orbit of $\TT$.
Let $T>0$ be any positive constant. When $t\in [0,T]$, $\bfH_{g,\Theta}(\Phi_j(t))$ is uniformly bounded above.
By Lemma \ref{H_compact}, $\bfI_g(\Phi_j(t)), \bfJ_g(\Phi_j(t))$ converges uniformly. 
In addition, $\bfH_{g,\Theta}$ is lower semi-continuous in strong topology.
Then $\bfM(\Phi(t))\leq \lim_{j\to\infty}(\bfM(\Phi_j(t))) \leq \bfM(\vphi)$.
Since $\bfM$ is also convex , $\bfM$ is decreasing along $\Phi$.

With the Hartogs' lemma \ref{lem-Hartogs} and arguing in the same way as \cite[Proof of Proposition 6.2]{Li20}, we get the conclusion.
\end{proof}

Now we can prove the Yau-Tian-Donaldson conjecture by using the method developed in \cite{BBJ15, BBJ18, LTW19,His19,Li19}.
\begin{theorem}
The generalized Ding functional $\bfD$ is $\bG$-coercive if and only if
$\bG$ is reductive and $\bX$ is $\bG$-uniform Ding-stabile.
\end{theorem}
\begin{proof}[proof of the ``only if" direction]
Let $(\mcX,\mcL)$ be an ample normal test configuration. Let $\Phi$ be a smooth psh ray associated to $(\mcX,\mcL)$.
Since $\bfD$ is $\bG$-coercive, there exist $\delta>0, C>0$ such that
$\bfD(\Phi(t)) \geq \delta \bfJ_{g,\bT}(\Phi(t)) - C$.
Then by slope formulas,
\begin{equation}
\bfD^\NA(\mcL) = \lim_{t\to\infty}\frac{\bfD(\Phi(t))}{t} \geq \lim_{t\to\infty} \delta \frac{\bfJ_{g,\bT}(\Phi(t))}{t} = \delta \bfJ^\NA_{g,\bT}(\mcL)
\end{equation}
Thus $X$ is $\bG$-uniform Ding-stable.
\end{proof}

\begin{proof}[proof of the ``if" direction]
We will decompose the proof of the ``only if" part into four steps.

\vskip 10pt

{\bf Step 1:}
Construct a destabling geodesic ray.

Assume $\bfD$ is not $\bG$-coercive. By Theorem \ref{equivalent}, this implies $\bfM$ is not $\bG$-coercive.
By Lemma \ref{destablizer},
there exists a destablizing geodesic ray $\Phi$ emanating from $0$, which satisfies the conditions \eqref{eq-Phinorm}-\eqref{eq-Jxilb}. 
\vskip 10pt

{\bf Step 2:}
Approximate the destablizing geodesic ray by test configurations. We use the construction in \cite{BBJ18, LTW19}.

We need two approximations: approximating $L$ by ample line bundles over the resolution of $X$ and approximating psh rays by test configurations.

Let $\rho: \tilde{X}\rightarrow X$ be a log resolution of singularities obtained by a composition of blowing-up along smooth centers. Assume $\{E_i\}$ are the exceptional divisors. Then there exist $b_i\in \bQ_{>0}$ such that $\rho^*L-\sum_i b_i E_i$ is ample on $\tilde{X}$, where $L=-K_X-D-B$ as before. Moreover, we have the identity
\begin{eqnarray*}
-K_{\tilde{X}}-D'-\rho^*B&=&\frac{1}{1+\epsilon}\left(\rho^*(-K_X-D-B)+\epsilon (\rho^*(-K_X-D-B)-\sum_i b_i E_i)\right)\\
&&\quad +\sum_i (-a_i+\frac{\epsilon}{1+\epsilon}b_i) E_i.
\end{eqnarray*}
Set:
\begin{equation}
\tilde{\Theta}=\rho^*\Theta, \quad D_\epsilon=D'+\sum_i (-a_i+\frac{\epsilon}{1+\epsilon}b_i)E_i, \quad P=\rho^*L-\sum_i b_i E_i, \quad L_\epsilon=\frac{1}{1+\epsilon}(\rho^*L+\epsilon P).
\end{equation}
Then we have:
\begin{equation}
-K_\tX-D_\epsilon-\rho^*B= L_\epsilon.
\end{equation}
Fix a smooth positively $e^{-\vphi_P}$ curved metric on $P$ and set $\Phi_\epsilon=\frac{1}{1+\epsilon}(\rho^*\Phi+\epsilon \vphi_P)$. Then $\Phi_\epsilon$ is a subgeodesic ray for the line bundle $L_\epsilon$ over $\tX$. 

Let $\fa_m = \mcJ(m\Phi_\epsilon)$ be the multiplier ideal sheaf with respect to $m\Phi_\epsilon$.
We would like to use the normal blow-up along $\fa_m$, $\Bl_{\fa_m}(\tX_\CC)$ as test configurations that approximate $\Phi_\epsilon$.
Let $\hat{L}_\epsilon = (1+\epsilon)L_\epsilon$, $\hat{\Phi}_\epsilon = (1+\epsilon)\Phi_\epsilon$.
We need to first show that, there exists an integer $m_0>0$, such that for any $m>0$,
$\oo_{\tX\times \CC}((m_0+m)\hat{L}_\epsilon \otimes \fa_m)$ is globally generated. If so, $\Bl_{\fa_m}(\tX_\CC)$
can be embedded into a projective space, which provides a $T_\bC\times \bG$-equivariant test configuration.
By Castelnuovo-Mumford regularity theorem, in order to show finite generation, it suffices to show
$H^i(\tX_\CC, ((m_0+m)\hat{L}_\epsilon- i P)\otimes \fa_m) = 0$, for any $i>0$, 
which can be reduced to show that $R^i(\oo_{\tX_\CC}((m_0+m)\hat{L}_\epsilon-(n+1) P)\otimes \fa_m) = 0$.
There exists an $m_0>0$, such that $-K_\tX + (m_0+m)\hat{L}_\epsilon-(n+1)P$ is $p_2$-ample for any $m\in \NN$,
where $p_2: \tX_\CC \rightarrow \CC$.
By applying Nadel vanishing to $(-K_\tX + (m_0+m)\hat{L}-(n+1)P)\otimes \fa_m$,
the vanishing of higher order direct image sheaves is proved.

By Demailly's Bergman kernel approximation, we have
$\hat{\Phi}_{\epsilon,m} = \frac{1}{2(m_0+m)}\log(\sum_j |s^{(m)}_j|^2)$, where $\{s^{(m)}_j\}$
is an orthonormal basis of $H^0(\tX_\CC, \oo_{\tX_\CC}(m_0+m)\hat{L}_\epsilon \otimes \fa_m)$, with respect to the $L^2$-norm
$\int_{\tX_\CC}|s|^2 e^{m\Phi_\epsilon}$.
Corresponding, we have $\hat{\phi}_{\epsilon,m}-\hat{\phi}_\triv$ as a function on $\tX^\div_\bQ$. By \cite[Lemma 5.7]{BBJ15},
$\hat{\phi}_\epsilon \leq \hat{\phi}_{\epsilon,m} \leq \frac{m}{m_0+m}\hat{\phi}_\epsilon + \frac{1}{m}(A_\tX+1)$,
and $\lim_{m\to\infty}\bfL_\Theta^\NA(\hat{\phi}_{\epsilon,m}) = \bfL_\Theta^\NA(\hat{\phi}_{\epsilon})$.
By monotonicity of $\bfE_g, \Lam_g$ functionals, we have
\begin{align}
\lim_{m\to\infty} \bfE^{\NA}_g(\hat{\phi}_{\epsilon,m}) &\geq \bfE'^\infty_g(\hat{\Phi}_{\epsilon})\\
\lim_{m\to\infty} \Lam^\NA_g(\hat{\Phi}_{\epsilon,m}) &\geq \Lam'^\infty_g(\hat{\Phi}_{\epsilon}).
\end{align}
Furthermore, 
\begin{lemma}
For any $\xi\in N_\bR$,
\begin{align}
\label{E_epsilon}
\lim_{\epsilon\to 0} \bfE'^\infty_{g,\vphi_{0,\epsilon}}(\hat{\Phi}_{\epsilon,\xi}) &= \bfE'^\infty_g(\Phi_\xi),\\
\label{Lambda_epsilon}
\lim_{\epsilon\to 0} \Lam'^\infty_{\vphi_{0,\epsilon}}(\hat{\Phi}_{\epsilon,\xi}) &= \Lam'^\infty(\Phi_\xi), \\
\label{L_epsilon}
\lim_{\epsilon\to 0} \bfL_\Theta^\NA(\hat{\phi}_{\epsilon,\xi}) &= \bfL_\Theta^\NA(\phi).
\end{align}
\end{lemma}
\begin{proof}
Equalities \eqref{Lambda_epsilon}, \eqref{L_epsilon} have been shown in \cite{Li19}.
By \cite{Li19}, we also have 
\begin{align*}
\lim_{\epsilon\to 0}(\bfE^\bvk_{\vphi_{0,\epsilon}})'^\infty(\hat{\Phi}^\bvk_{\epsilon,\xi}) = (\bfE^\bvk)'^\infty(\Phi^\bvk_\xi).
\end{align*}
Since for any $\delta>0$, there exists a polynomial $p_i = \sum_\vk a_\vk \prod_{\alpha=1}^r y_i^{k_\alpha}$, such that $|p_i-g|<\delta$.
As shown in the proof of Proposition \ref{F_slope}, there exists a $C>0$, uniform over $\epsilon$, such that
\begin{align*}
|\bfE_{g,\vphi_{0,\epsilon}}'^\infty(\hat{\Phi}_{\epsilon,\xi})-\sum_{\vk}a_\vk (\bfE_{\vphi_{0,\epsilon}}^\bvk)'^\infty(\hat{\Phi}^\bvk_{\epsilon,\xi})| < C \delta.
\end{align*}
Then the uniform convergence implies \eqref{E_epsilon}.
\end{proof}

\vskip 10pt

{\bf Step 3:} Completion of the proof
By Lemma \ref{destablizer}, $\bfM$ is decreasing along $\Phi$. Then $\bfD'^\infty(\Phi)\leq \bfM'^\infty\leq 0$.
By the construction of $\Phi$, $\Lam'^\infty_g(\Phi) = 0$, $-\bfE_g'^\infty(\Phi) = \bfJ_{g,\bT}'^\infty(\Phi) = 1$.
Then
$\bfL_\Theta'^\infty(\Phi) \leq -1$.

Choose a subsequence of $\bG$-invariant divisorial valuations $v_k\in X^{\rm div}$ such that
\begin{equation}
\bfL_\Theta'^\infty(\Phi) \leq A_{(X,D+\Theta)}(v_k) - G(v_k)(\Phi) < \bfL_\Theta'^\infty(\Phi) + \frac{1}{k}.
\end{equation}
By valuation criterion, there exists $\delta = \delta_{\bG}(X,D+\Theta) >1 $, and $\xi_k\in N_\bR$, such that
\begin{equation}
A_{(X,D+\Theta)}(v_{k,\xi_k}) \geq \delta S_g(v_{k,\xi_k}).
\end{equation}

From the argument of the valuation criterion in Theorem \ref{thm-valDing} and the same calculation as in \cite[4.4]{LTW19}, we have
\begin{equation}
A_{(\tilde{X}, D_\epsilon+\tilde{\Theta})}(v_{k,\xi_k}) \geq \delta' S_{L_\epsilon, g}(v_{k,\xi_k})
\end{equation}
where $\delta'>1$. Indeed, this follows from the estimate: there exists $C>0$ such that for any $v\in X^{\rm div}_\bQ$,
\begin{eqnarray*}
\frac{A_{(\tilde{X}, D_\epsilon+\tilde{\Theta})}(v)}{S_{L_\epsilon,g}(v)}\ge (1-C\epsilon)\frac{A_{(X,D+\Theta)}(v)}{S_g(v)}
\end{eqnarray*}
for some $C>0$, which follows from two inequalities:
\begin{enumerate}
\item Note that the identity $D_\epsilon=D_0+\frac{\epsilon}{1+\epsilon}\sum_i b_i E_i$.  
\begin{eqnarray*}
\frac{A_{(\tX, D_\epsilon+\tilde{\Theta})}(v)}{A_{(X, D+\Theta)}(v)}&=&\frac{A_\tX(v)-v(D_\epsilon)-v(\Theta)}{A_\tX(v)-v(D_0)-v(\Theta)}=1-\frac{\epsilon}{1+\epsilon}\frac{\sum_i b_i v(E_i)}{A_\tX(v)-v(D_0)-v(\Theta)}\\
&\ge& 1-\frac{\epsilon}{1+\epsilon}{\rm lct}(\tX, D_0+\Theta+\sum_i b_i E_i)^{-1}.
\end{eqnarray*}

\item Recall that $L_\epsilon=\rho^*L-\frac{\epsilon}{1+\epsilon}\sum_i b_i E_i$. It is easy to see that the integrands in $S_g$ and $S_{L_\epsilon,g}$ has a comparison:
\begin{equation}
\vol_g(L-xv)\ge \vol_g(L_\epsilon-xv).
\end{equation}
\end{enumerate}

Moreover by the same argument as in \cite[3.1]{Li19} (see also \cite[2.1.3]{Li20}), we have the identity for any $\xi\in N_\bR$:
\begin{eqnarray*}
(\phi_{\epsilon,m,\xi}-\phi_\triv)(v)&=&(\phi_{\epsilon,m}-\phi_\triv)(v_\xi)+A_{(\tilde{D}, D_\epsilon+\tilde{\Theta})}(v_{\xi})-A_{(\tilde{D}, D_\epsilon+\tilde{\Theta})}(v). 
\end{eqnarray*}
Again note that we are using the canonical lifting of $\xi$ (see \eqref{eq-canlift}) to twist the non-Archimedean metric.

Now we can estimate in the same way as \cite{Li19}: 
\begin{align}
\begin{split}
& \quad \bfL^\NA_{(\tX,D_\epsilon+\tilde{\Theta})}(\phi_{\epsilon,m}) + O(\epsilon,m^{-1},k^{-1}) \\
&= A_{(\tX,D_\epsilon+\tilde{\Theta})}(v_k) + (\phi_{\epsilon,m}-\phi_\triv)(v_k)\\
&= A_{(\tX,D_\epsilon+\tilde{\Theta})}(v_{k,\xi_k}) + (\phi_{\epsilon,m,-\xi_k}-\phi_\triv)(v_{k,\xi_k}) \\
&\geq \delta' S_{L_\epsilon}(v_{k,\xi_k}) + (\phi_{\epsilon,m,-\xi_k}-\phi_\triv)(v_{k,\xi_k}) \\
&\geq \delta' \bfE^\NA_g((\delta')^{-1} (\phi_{\epsilon,m,-\xi_k}-\phi_\triv)) \\
&= ( -\delta'\bfJ^\NA_g({\delta'}^{-1}(\phi_{\epsilon,m,-\xi_k}-\phi_\triv)) +\bfJ^\NA_g((\phi_{\epsilon,m,-\xi_k}-\phi_\triv))) + \bfE^\NA_g((\phi_{\epsilon,m,-\xi_k}-\phi_\triv)) \\
&\geq (1-(\delta')^{\frac{-1}{C}}) \bfJ^\NA_g((\phi_{\epsilon,m,-\xi_k}-\phi_\triv)) + \bfE^\NA_g((\phi_{\epsilon,m,-\xi_k}-\phi_\triv))\\
&= (1-(\delta')^{\frac{-1}{C}}) \Lam'^\infty_g(\Phi_{\epsilon,m,-\xi_k}) + (\delta')^{\frac{-1}{C}}\bfE'^\infty_g(\Phi_{\epsilon,m,-\xi_k})\\
&\geq (1-(\delta')^{\frac{-1}{C}})\Lam'^\infty_g(\Phi_{-\xi_k}) + (\delta')^{\frac{-1}{C}}\bfE'^\infty_g(\Phi_{-\xi_k})\\
&= (1-(\delta')^{\frac{-1}{C}})\bfJ'^\infty_g(\Phi_{-\xi_k}) + \bfE'^\infty_g(\Phi_{-\xi_k})\\
&= (1-(\delta')^{\frac{-1}{C}})\bfJ'^\infty_g(\Phi_{-\xi_k}) + \bfE'^\infty_g(\Phi)+\Fut_g(\xi_k)\\
&\ge (1-(\delta')^{\frac{-1}{C}})-1=-(\delta')^{\frac{-1}{C}}.
\end{split}
\end{align}
The second inequality used Lemma \ref{lem-NAlegend}. The last inequality uses \eqref{eq-Jxilb}. Moreover we have used the identity \eqref{eq-Fut2} and the fact the $\Fut_g(\xi_k)\equiv 0$.
Letting $\epsilon\rightarrow 0, m\rightarrow+\infty, k\rightarrow+\infty$, we get a contradiction to $\bfL^\NA(\phi)\le -1$. 

\begin{remark}
By the same argument as in \cite[5.4]{Li19}, we actually know that $|\xi_k|$ in the above proof is uniformly bounded. Moreover, if $X$ is smooth, then the above argument can be simplified (see \cite{BBJ18, His19, Li20})
\end{remark}

\end{proof}

\section{Stability via special test configurations}\label{sec-MMP}

\begin{definition}
A test configuration $(\mcX, \mcL)$ of $(X, L)$ is a special test configuration if $\mcX_0$ is a normal projective variety and $\mcL$ is relatively ample.
\end{definition}
\begin{remark}
If $(X, D)$ is log Fano and $\Theta=0$, then the usual definition of special test configuration also requires that $\mcL\sim_{\pi,\bQ} -(K_\mcX+\mcD)$.
Since we are considering the general twisting, the special test configuration is in a more general sense compared with the log Fano case. See also Remark \ref{rem-special}.
\end{remark}

\begin{theorem}
Let $\bG\subseteq \Aut_T(X, D, \Theta)$ be a reductive subgroup. 
Then $(X, D+\Theta)$ is $\bG$-uniformly $g$-Ding-stable if and only if if it is $\bG$-uniformly $g$-Ding-stable for all $T_\CC\times \bG$-equivariant special test configurations.
\end{theorem}

We will use the notations for $g$-intersection of equivariant line bundles
as defined in Definition \ref{def-intg}. The following observation is the key to our later calculations.
\begin{theorem}
Let $(\mcX, \mcL)$ be a test configuration of $(X, L)$. Assume $\mcX_0=\sum_i b_i E_i$. 
Assume that $\mcL_\lambda=\mcL+\sum_i c_i(\lambda) E_i$ for $\lambda\in [0, \epsilon)$ and $c_i(\lambda)$ are differentiable functions of $\lambda\in [0, \epsilon)$.
We have the following formula:
\begin{equation}\label{eq-dintg}
\frac{d}{d\lambda}\frac{1}{(n+1)!}(\bar{\mcL}_\lambda^{\cdot n+1})_g=\frac{1}{n!}(\bar{\mcL}_\lambda^{\cdot n}\cdot \frac{d\bar{\mcL}_\lambda}{d\lambda})_g=\sum_i  \frac{d c_i}{d\lambda} \int_{E_i} g_\vphi  \frac{(\ddc\vphi)^n}{n!}, 
\end{equation}
where $\vphi$ is a smooth Hermitian metric on $\mcL_0\rightarrow \mcX_0$.
In particular, if $\frac{d c_i(\lambda)}{d\lambda}\ge 0$ and $g\ge 0$ over $P$, then we have the non-negativity:
\begin{equation}
\frac{d}{d\lambda} (\bar{\mcL}_\lambda^{\cdot n+1})_g\ge 0.
\end{equation}
\end{theorem}
\begin{proof}
If $g=y^{\vec{k}}$ is a monomial, then we have $\frac{(\bar{\mcL}^{\cdot n+1})_g}{(n+1)!}=\frac{(\bar{\mcL}^{[\vec{k}]})^{\cdot n+k+1}}{(n+k+1)!}$. So 
\begin{equation}
\frac{d}{d\lambda} \frac{(\bar{\mcL}^{\cdot n+1})_g}{(n+1)!}=\frac{1}{(n+k)!} (\mcL^{\bvk})^{\cdot n+k}\cdot \dot{\mcL}^\bvk=\frac{1}{n!}(\bar{\mcL}^{\cdot n}\cdot \dot{\mcL})_g.
\end{equation}
By using the construction in our proof of Proposition \ref{MA_eta}, we see that \eqref{eq-dintg} holds for monomial and hence any polynomial function $g$. 
For a general continuous $g$, by Stone-Weierstrass theorem, we can find a sequence of polynomials functions $g_i$ that converges to $g$ uniformly over $P$. Then

\begin{eqnarray*}
\frac{d}{d\lambda} \frac{(\bar{\mcL}_\lambda^{\cdot n+1})_g}{(n+1)!}&=& \frac{d}{d \lambda} \lim_{i\rightarrow+\infty}\frac{(\bar{\mcL}_\lambda^{\cdot n+1})_{g_i}}{(n+1)!}\\
&=&\lim_{i\rightarrow+\infty}\frac{d}{d\lambda}\frac{(\bar{\mcL}_\lambda^{\cdot n+1})_{g_i}}{(n+1)!}=\lim_{i\rightarrow+\infty} \frac{1}{n!}\left(\mcL^{\cdot n}\cdot \frac{d}{d\lambda}\mcL_\lambda\right)_{g_i}.
\end{eqnarray*}
The conclusion follows easily. Note that 
the switch of limit and derivative follows from the following standard fact from real analysis.

\end{proof}

\begin{theorem}[\cite{Rud76}]
Suppose $\{f_m\}$ is a sequence of functions, differentiable on $[a,b]$ and such that $f_m(x_0)$ converges for some point $x_0\in [a, b]$. If $\{f'_m\}$ converges uniformly on $[a, b]$, then $\{f_m\}$ converges uniformly on $[a, b]$, to a function $f$, and
\begin{equation}
f'(x)=\lim_{m\rightarrow+\infty} f_m'(x).
\end{equation}
\end{theorem}


Now we assume that $\mathfrak{S}$ is $T$-invariant sub-linear system of $|mB|$ for some $m\in \bN$. Choose a $T$-equivariant basis $\mathfrak{s}:=\{s_1,\dots, s_N\}$. Then $e^{-\psi}=\frac{1}{(\sum_i |s_i|^2)^{2/m}}$ is a possibly singular Hermitian metric on $B$ and its curvature current is
\begin{equation}\label{eq-LSTheta}
\Theta:=\Theta_\mathfrak{s}:=\frac{1}{m}\ddc \log \sum_i |s_i|^2.
\end{equation}
In this case, we say that $\Theta$ is associated to the sub-linear system $\mathfrak{S}$.

\begin{lemma}\label{lem-character}
Let $\mathfrak{S}$ be a $T$-invariant sub-linear system of $|mB|$. Then for a generic choice of basis $\mathfrak{s}=\{s_1, \dots, s_N\}$ the following statement holds true: there is a character $\chi: \Aut_T(X, D, \Theta_{\mathfrak{s}}) \rightarrow \bC^*$ such that for any $s\in \mathfrak{S}$, we have $\sigma \cdot s=\chi(s) s$. In particular, $\mathfrak{S}$ is $\Aut_T(X,D,\Theta_{\mathfrak{s}})$-invariant.
\end{lemma}
\begin{proof}
Let $\xi \in \mathfrak{aut}_T(X, D, \Theta)$ also denote the corresponding holomorphic vector field. Then
\begin{equation}
0=\iota_\xi \Theta=\bar{\partial} \left(\frac{\sum_i \la \mathfrak{L}_\xi s_i, s_i\ra}{\sum_i |s_i|^2}\right).
\end{equation}
Since $\mathfrak{s}=\{s_i\}$ is generic, we have $\mathfrak{L}_\xi s_i =\chi(\xi) s_i$ for some $\chi(\xi)\in \bC^*$. 
\end{proof}

For any $\Delta\in \frac{1}{m}\mathfrak{S}$, we set:
\begin{equation}
\bfL^\NA_{\Delta}(\phi)=\frac{1}{n!}\inf_{v} \left(A_{(X,D)}(v)-v(\Delta)+\phi(v)\right).
\end{equation}

\begin{lemma}
With the above assumptions, 
\begin{equation}
\bfL_\Theta^\NA(\phi)=\sup\left\{ \bfL^\NA_{\Delta}(\phi); \Delta\in \mathfrak{S} \right\}.
\end{equation}
Moreover for a general divisor $\Delta\in\mathfrak{S}$, we have $\bfL^\NA_\mathfrak{S}(\phi)=\bfL_\Delta(\phi)$.
\end{lemma}
This follows easily from the identity: for any $v\in X^{\rm div}_\bQ$, 
\begin{equation}
v(\Theta)=\inf\left\{ v(\Delta); \Delta\in \frac{1}{m}\mathfrak{S} \right\}=\frac{1}{m}v(\mathfrak{b}(\mathfrak{S})),
\end{equation}
where $\mathfrak{b}(\mathfrak{S})$ is the base ideal of $\mathfrak{S}$.
We will follow the MMP process in \cite{LX14} that was adapted to the Ding stability in \cite{BBJ15, Fuj19a}, and to the twisted K-stability in \cite{BLZ19}. We will essentially show that the MMP process respects the twisted $g$-Ding-stability as well.
Moreover, by the argument in \cite[4.1]{Li19}, we don't need to worry about the twisting by $\xi\in N_\bR$ in the calculation.
\begin{theorem}\label{thm-MMP}
Let $\bX=(X, D, \Theta_{\mathfrak{s}})$ and $\mathfrak{s}\subset \mathfrak{S}$ be as above. Let $(\mcX, \mcL)/\PP^1$ be a normal, ample test configuration for $(X, L)$. 
There is an integer $d$ and a special test configuration $(\mcX^s, \mcL^s)$ such that for any $\epsilon \in [0,1]$ and any $\xi\in N_\bR$, we have
\begin{equation}
\bfD^\NA(\phi^s)-\epsilon \bfJ^\NA_{g,\bT}(\phi^s)\ge d \cdot \left(\bfD^\NA(\phi)-\epsilon \bfJ^\NA_{g,\bT}(\phi)\right).
\end{equation}
\end{theorem}

\begin{proof}

{\it Step 1:} Choose a general divisor $\Delta\in \frac{1}{m}\mathfrak{S}$. Set $Q=D+\Delta$ and $L=-K_X-D-\Delta$.

Then there exist $d\in \ZZ_{>0}$, a projective birational $\CC^*$-equivariant morphism $\pi: \mcX^{\rm lc}\rightarrow \mcX^{(d)}$ and a normal, ample test configuration $(\mcX^\lc, \mcQ^\lc , \mcL^\lc)/\PP^1$ for $(X, Q, L)$ such that for any $\epsilon \in [0, 1]$ and any $\xi\in N_\bR$,
 \begin{equation}\label{eq-dec1}
d \left(\bfD^\NA(\phi_\xi)-\epsilon \bfJ_g^\NA(\phi_\xi)\right)\ge \bfD^\NA(\phi^\lc_\xi)-\epsilon \bfJ_g^\NA(\phi^\lc_\xi).
\end{equation}


Choose a semistable reduction of $\mcX\rightarrow \bC$. By this, we mean that there is an integer $d$ and a $\bG$-equivariant log resolution of singularities $\tilde{\mcX}\rightarrow \mcX_d:=\mcX\times_{\bC, t\rightarrow t^d}\bC$ such that $(\tilde{\mcX}, \mcX_0)$ is simple normal crossing. Since the linear system $\mathfrak{S}$ is $(T_\bC\times \bG)$-invariant, we can run a $(T_\bC\times \bG\times \bC^*)$-equivariant MMP (see \cite{And01, Pas17, Pro12}) to get a log canonical modification: $\pi^\lc: \mcX^\lc\rightarrow\mcX_d$ such that
if $\mcQ^\lc$ is the pushforward of $\mcQ$ then
 $(\mcX^\lc, \mcX^\lc_0+\mcQ^\lc)$ is log canonical and $K_{\mcX^\lc}+\mcQ^\lc$ is relatively ample over $\mcX_d$. In the following calculation, we can assume $d=1$ by multiplicativity of both sides of \eqref{eq-dec1}. 
Set $E=K_{\mcX^\lc}+\mcQ^{\lc}+(\pi^\lc)^*\mcL=\sum_{i=1}^k e_i \mcX_{0,i}$ with $e_1\le e_2\le \dots \le e_k$ and $\mcL_\lambda=\rho^* \mcL+\lambda E$.
Then
\begin{eqnarray*}
\bfL^\NA_\Theta(\phi_{\lambda,\xi})=\bfL^\NA_\Theta(\phi_\lambda)=\bfL^\NA_\Delta(\phi_\lambda)=(1+\lambda)e_1.
\end{eqnarray*}
As in the argument in \cite{Li19}, we can reduce the calculation to the case when $\xi=0$ in which case:
\begin{eqnarray*}
n!(\bfD^\NA(\phi_\lambda) - \epsilon \bfJ_g^\NA(\phi_\lambda))&=&-(1-\epsilon)\frac{(\bar{\mcL}_{\lambda}^{\cdot n+1})_g}{n+1}-\epsilon (\bar{\mcL}_\lambda \cdot L_{\bP^1}^{\cdot n})_g+(1+\lambda)e_1\\
 &=& -(1-\epsilon)\frac{(\bar{\mcL}_\lambda^{\cdot n+1})_g}{n+1}-\epsilon (\bar{\mcL}_{s} \cdot L_{\bP^1}^{\cdot n})_g+(1+\lambda)e_1 (\mcX^\lc_0\cdot \bar{\mcL}^{\cdot n})_g
\end{eqnarray*}
To see that the difference is decreasing, we calculate:
\begin{eqnarray*}
&&n!\frac{d}{d\lambda}(\bfD^\NA_g-\epsilon\bfJ^\NA_g)(\phi_\lambda)=-(1-\epsilon)(\bar{\mcL}_\lambda^{\cdot n} \cdot E)_g-\epsilon (E\cdot (L_{\bP^1})^{\cdot n})_g+e_1 \\
&=& -(1-\epsilon)(\bar{\mcL}_\lambda^{\cdot n} \cdot E)_g + (1-\epsilon) e_1 (\mcX^\lc_0 \cdot \bar{\mcL}_\lambda^{\cdot n})_g  -\epsilon (E\cdot (L_{\bP^1})^{\cdot n})_g+\epsilon e_1 (\mcX^\lc_0 \cdot \bar{\mcL}^{\cdot n})_g \\
&=&-\sum_i (e_i-e_1 )\left[(1-\epsilon)(E_i\cdot \bar{\mcL}_\lambda^{\cdot n})_g+\epsilon (E_i\cdot  (L_{\bP^1})^{\cdot n})_g\right]\le 0.
\end{eqnarray*}
By the argument in \cite{BLZ19}, we know that the outcome $(\mcX^\lc, \mcL^\lc)$ does not depend on the choice of the general member of $\Delta\in \frac{1}{m}\mathfrak{S}$. Moreover $\mcQ^\lc$ is just the closure of $Q\times \bC^*$ in $\mcX^\lc$ ($Q$ is called compatible with $\mcX^\lc$ in \cite{BLZ19}).

{\it Step 2:}
With the $(\mcX^\lc, \mcQ^\lc, \mcL^\lc)$ obtained from the first step, we run a relative MMP with scaling to get a normal, ample test configuration $(\mcX^{\rm ac}, \mcL^{\rm ac})/\PP^1$ for $(X, -K_X)$ with 
$(\mcX^{\rm ac}, \mcQ^{\rm ac}+\mcX^{\rm ac}_0)$ log canonical such that $-(K_{\mcX^{\rm ac}}+\mcQ^{\rm ac})\sim_\QQ\mcL^{\rm ac}$. Because $m \Delta\in \mathfrak{S}$ is a general member of the $\bG$-invariant linear system, the MMP is automatically $\bG$-equivariant and the outcome does not depend on the choice of $\Delta$. Moreover, for any $\epsilon\in [0,1]$ we have:
\begin{equation}\label{eq-dec2}
\bfD^\NA_g(\phi^\lc)-\epsilon \bfJ^\NA_{g,\bT}(\phi^\lc)\ge \bfD^\NA_g(\phi^{\rm ac})-\epsilon\bfJ^\NA_{g,\bT}(\phi^{\rm ac}).
\end{equation}
More concretely, we take $\ell\gg 1$ such that $\mcH^\lc=\mcL^\lc-(\ell+1)^{-1} (\mcL^\lc+K_{\mcX^\lc}+\mcQ^\lc)$ is relatively ample. Set 
$\mcX^0=\mcX^\lc$, $\mcQ^0=\mcQ^\lc$, $\mcL^0=\mcL^\lc$, $\mcH^0=\mcH^\lc$ and $\lambda_0=\ell+1$. 
Then $K_{\mcX^0}+\mcQ^0+\lambda_0 \mcH^0=\ell \mcL^0$. We run a sequence of $(K_{\mcX^0}+\mcQ^0)$-MMP over $\bC$ with scaling of $\mcH^0$.  Then we obtain a sequence of models 
$$
\mcX^0\dasharrow \mcX^1\dasharrow\cdots \dasharrow \mcX^k
$$
and a sequence of critical values 
$$
\lambda_{i+1}=\min\{\lambda; K_{\mcX^i}+\mcQ^i+\lambda \mcH^i \text{ is nef over } \bC\}
$$
with $\ell+1=\lambda_0\ge \lambda_1\ge \cdots\ge \lambda_k>\lambda_{k+1}=1$. For any $\lambda_i\ge \lambda\ge \lambda_{i+1}$,
we let $\mcH^i$ (resp. $\mcQ^i$) be the pushforward of $\mcH$ (resp. $\mcQ$) to $\mcX^i$ and set
\begin{equation}
\mcL^i_\lambda=\frac{1}{\lambda-1}\left(K_{\mcX^i}+\mcQ^i+\lambda\mcH^i\right)
=\frac{1}{\lambda-1}(K_{\mcX^i}+\mcQ^i+\mcH^i)+\mcH^i=:\frac{1}{\lambda-1}E+\mcH^i.
\end{equation}
Write $E=\sum_{j=1}^k e_j \mcX^i_{0,j}$ with $e_1\le e_2\le \cdots \le e_k$.
By the argument in \cite{Li19}, we reduce to the case when $\xi=0$ in which case the statement follow from the following decreasing property (see \cite{Fuj19a} and \cite{LX14}):
\begin{eqnarray*}
n!\frac{d}{d\lambda} \bfD^\NA(\phi_\lambda^i)&=&\frac{d}{d\lambda}\left(-(1-\epsilon)\frac{(\bar{\mcL}_{\lambda}^{i\cdot n+1})_g}{n+1}-\epsilon (\bar{\mcL}\cdot (L_{\bP^1})^{\cdot n})_g+\frac{\lambda}{\lambda-1}e_1\right)\\
&=&(1-\epsilon) \frac{1}{(\lambda-1)^2} (E\cdot \bar{\mcL}^{i\cdot n})_g+\epsilon \frac{1}{(\lambda-1)^2} (E\cdot (L_{\bP^1})^{\cdot n})_g-\frac{1}{(\lambda-1)^2}e_1 \\
&=&\frac{1}{(\lambda-1)^2}\sum_i (e_i-e_1) \left[(1-\epsilon) (E_i\cdot \bar{\mcL}^{\cdot n})_g+\epsilon (E_i\cdot (L_{\bP^1})^{\cdot n}_g\right]\ge 0.
\end{eqnarray*}

{\it Step 3:}
With the test configuration $(\mcX^\ac, \mcQ^\ac, \mcL^\ac)$ obtained from step 2, there exists 
$d\in \ZZ_{>0}$ and a projective birational $T_\bC\times \bC^*$-equivariant birational map $\mcX^{(d)}\dasharrow \mcX^s$ over $\PP^1$ such that:
\begin{equation}\label{eq-dec3}
d (\bfD^\NA(\phi^\ac)-\epsilon \bfJ^\NA_{g,\bT}(\phi^\ac))\ge \bfD^\NA(\phi^s)-\epsilon \bfJ^\NA_{g,\bT}(\phi^s).
\end{equation}
As in \cite{LX14}, this is achieved by doing a base change and run an MMP. By the argument in \cite{BLZ19}, the outcome does not depend on the $\Delta\in \frac{1}{m}\mathfrak{S}$ and is automatically $\bG$-equivariant.

Let $E=-K_{\mcX^s/\PP^1}-(-K_{\mcX'/\PP^1})$. Then $E\ge 0$ by the negativity lemma. $\mcL_\lambda=-K_{\mcX'/\PP^1}+\lambda E$. As in \cite{Li19}, we verify the decreasing \eqref{eq-dec3} by reducing to the case when $\xi=0$, and calculate:
\begin{eqnarray*}
n!\frac{d}{d\lambda}(\bfD^\NA(\phi_\lambda)-\epsilon \bfJ^\NA_g(\phi_\lambda))&=&-\frac{d}{d\lambda}\left((1-\epsilon)\frac{(\bar{\mcL}^{\cdot n+1})_g}{n+1}-\epsilon (\bar{\mcL}\cdot (L_{\bP^1})^{\cdot n})_g+t e_1\right)\\
&=& -\sum_i (e_i-e_1)\left[(1-\epsilon)(E_i\cdot \bar{\mcL}^{\cdot n})_g+\epsilon (E_i\cdot (L_{\bP^1})^{\cdot n})_g\right]\le 0.
\end{eqnarray*}

So we have obtained special test configuration $(\mcX^s, \mcQ^s, \mcL^s)$ of $(X, Q)$ which is $T_\bC\times \bG$-equivariant. 

\end{proof}

\begin{remark}\label{rem-special}
If $B\sim_\bQ t(-K_X-D)$ for some $t\in [0,1)$, then $-(K_X+D)=\frac{1}{1-t}(-K_X-D-B)=\frac{1}{1-t}L$. The end product $\pi^s: (\mcX^s, \mcD^s)\rightarrow \bC$ is 
is a special test configuration of $(X, D)$ satisfying $-(K_{\mcX^s}+\mcD^s)\sim_{\pi^s, \bQ}\frac{1}{1-t}\mcL$.
\end{remark}

As in \cite[section 4]{Li20} which generalizes \cite{Fuj19a, Fuj19b}, Theorem \ref{thm-MMP} yields another proof of the valuative criterion when $\Theta$ is associated to a linear system:
\begin{corollary}
Let $(X, D, \Theta)$ be as above. Then the valuative criterion holds true.  In other words, $(X, D+\Theta)$ is $\bG$-uniformly $g$-Ding-stable if and only if there exists $\gamma>1$ such that for any $v\in X^{\rm div}$, there exists $\xi\in N_\bR$ satisfying:
\begin{equation}
A_{D+\Theta}(v_\xi)-\gamma \cdot S_g(v_\xi)\ge 0.
\end{equation}

\end{corollary}

\section{Examples}
\label{sec-exmp}
It should be clear that our results generalize most of variational study of K\"{a}hler-Ricci/Mabuchi soliton metrics on log Fano varieties. Here we just point out some simple consequence. We leave the other applications to the interested reader.

Recall that $\bT=C(\Aut_T(X, D, \Theta))$ is a complex torus of rank $\fr$. Any one-parameter subgroup of $\bT$ generated by $\xi\in N_\bR$ gives a 
$T_\bC\times\Aut_T(X, D, \Theta)$-equivariant product ($\bR$-)test configuration $(X_{\bC, \xi}, L_{\bC, \xi})$ of $(X, L)$. 
In this case, $\bfD^\NA_g(X_{\bC, \xi}, L_{\bC, \xi})=\Fut_g([\omega], \xi)$. Since we can replace $\xi$ by $-\xi$, we see that 
the $\bG$-uniform $g$-Ding semistability implies
that $\Fut_g([\omega], \xi)=0$ for any $\xi\in N_\bR$.
\vskip 2mm

When $\log g$ is affine, the equation \eqref{MA_equation} reduces to the K\"{a}hler-Ricci soliton equation, and when $g$ is affine, it reduces to the case of Mabuchi solitons. Note that in these two cases, the vanishing of Futaki invariant uniquely determines the function $g$, which means the uniqueness of $V$ in \eqref{eq-KRsm} and \eqref{eq-Msoliton} for which the corresponding equation could possibly have a solution. 


Generalizing the application of results in \cite{DaS16}, it is clear that the theorem \ref{thm-YTD} allows us to get new examples of K\"{a}hler-Ricci $g$-solitons on possibly singular varieties with large symmetries. 

For example, let $(X, D)$ be a log Fano toric variety determined by a monotone labelled polytope $
P=\{l_i=\la \nu_i, x\ra \le 1\}$ (see \cite{Leg16}). 
For simplicity assume $\Theta=0$. By Theorem \ref{thm-YTD}, we know that the existence of K\"{a}hler-Ricci $g$-soliton is equivalent to the vanishing of generalized Futaki invariant which is equivalent to the vanishing of the weighted barycenter:
\begin{equation}\label{eq-toricFut}
\int_P x_i g(x) dx=0, \quad i=1,\dots, n.
\end{equation}
This generalizes the works on toric K\"{a}hler-Ricci solitons in \cite{WZ04, SZ11, BB13, Leg16} when $\log g$ is affine, and existence results about toric Mabuchi solitons in \cite{Yao17, Nak19} when $g$ is affine. Note that in the Mabuchi soliton case, the constraint \eqref{eq-toricFut} uniquely determines $g=1-\theta_P$ where $\theta_P$ is the extremal function as defined in \cite{FM95}. The condition $g>0$, which we assumed, becomes an obstruction to the existence of solutions to the corresponding equation.


More generally, similar applications can be applied to $T$-varieties of complexity one (see \cite{CaS18, IS17}) and spherical varieties as in \cite{Del16} or Fano $G$-varieties (for a reductive complex Lie group $G$) as in \cite{LTZ20}. In other words, one can effectively check their stability hence get the (non-)existence of $g$-soliton on such varieties. \footnote{T. Delcroix pointed to us that a related generation for the existence of coupled complex Monge-Amp\`{e}re equations on Fano horosymmetric manifolds has appeared in \cite{DeHu18}. }


\section{Appendix: Mabuchi functionals as Kempf-Ness functionals}
\label{moment}
In this section, we will build up the moment map for the equation \eqref{MA_equation}.
We will not use the moment map explicitly in our paper. However, the moment map provides us with a Kempf-Ness picture,
which illustrates the naturality of using the variational approach to study \eqref{MA_equation}.
For simplicity, we will only formulate the moment map for any polarized projective manifold $(X, L)$ without twistings by $D$ and $\Theta$. 
After the completion of the paper, we were informed by Inoue that the moment map picture has been built up
for K\"ahler Ricci solitons in \cite{Inoue19} and constant weighted scalar-curvature K\"ahler metrics in \cite{Lahdili19}.
Nonetheless, for the completeness of the paper, we will still include a proof under our settings in this Appendix.

Recall our notation: the holomorphic vector fields $V_1,\dots, V_r$ are generators of the complexified toric action $T_\CC$.
We will fix the Symplectic form $\omega$ in the following.
We will define the following Lie algebra of the Hamiltonian action
\begin{align}
\label{mcG}
Lie(\mcG) = \{v \in C^\infty_0(X): v \text{ is $T$-invariant. } \}
\end{align}
where $(\bar\partial v)^{\#} =V_{v}^{1,0} = \vphi^{i\bar{j}}\partial_{\bar{j}}v$.
The Lie algebra is associated with a metric $\la,\ra_{g_\vphi}$ defined by
\begin{align}
\label{norm_f}
\la v_1,v_2\ra_{g_\vphi}  = \int_X v_1 v_2 g_\vphi  (\ddc\vphi)^n
\end{align}
for any $v_1, v_2 \in Lie(\mcG)$, where $g_\vphi = e^{f_\vphi}$.

Let $\mcJ$ be the set of almost complex structures on $X$ that are compatible with the symplectic form $\ddc\vphi$.
The moduli space under our consideration is
\begin{align}
\label{mcJ}
\mcJ_T := \{J\in \mcJ: J \text{ is integral, invariant under the } T \text{-action and is compatible with } \ddc\vphi. \}
\end{align}
The space $\mcJ$ is associated with a Symplectic form $(,)_{g_\vphi}$ induced by $\ddc\vphi$:
\begin{align}
(\mu_1,\mu_2)_{g_\vphi} = \int_X \la\mu_1,\mu_2\ra_{\ddc\vphi} g_\vphi (\ddc\vphi)^n
\end{align}
for any $\mu_1,\mu_2 \in T_J \mcJ$. $(,)_{g_\vphi}$ is closed since $g_\vphi \omega^n$ is a $2n$-form independent of the choice
of the complex structure.

Denote $T^{1,0} = (TX_{\CC})^{1,0}$. Then
$T_J \mcJ \simeq \Omega^{0,1}(T^{1,0})$. Since $J$ is unitary, the Symplectic form $\omega$ induces a duality
$\Omega^{0,1}(T^{1,0}) \simeq S^2(T^{1,0})$.

Define the action of $Lie(\mcG)$ on $\mcJ_T$ by:
$\mathfrak{L}_{V_{v}^{1,0}}J$
where 
$V_v^{1,0} = J V_v + \sqrt{-1} V_v$, $V_v$ is the Hamiltonian vector field induced by $v$. Let $\odot$ denotes for the symmetric product.
By \cite[Lemma 10]{Don97} $\mathfrak{L}_{V_v^{1,0}}J = \bar\partial V_v^{1,0} - \overline{V_v^{1,0}} \lrcorner N$,
where $N=0$ when $J$ is integrable. 
We can see that the image of the action is kept to be $T$-invariant.
Let $V_{\theta_\alpha}$ be the Hamiltonian vector field induced by $\theta_\alpha$, i.e, $V_{\theta_\alpha}\lrcorner \ddc\vphi = d \theta_\alpha$.
Since $J$ is $T$-invariant, $\mathfrak{L}_{V_{\theta_\alpha}}J = 0$. Furthermore, since $J$ is integrable, $\mathfrak{L}_{JV_{\theta_\alpha}}J=0$,
which implies $\xi_\alpha = JV_{\theta_\alpha}+\sqrt{-1}V_{\theta_\alpha}$ is a holomorphic vector field with respect to $J$.
We also have $\dot{\theta}_\alpha = 0$, $\dot{f}_\vphi = \sum_\alpha  f_\alpha  \dot{\theta}_\alpha = 0$.

\begin{proposition}
For the action defined above, the corresponding moment map $\bfm$ is
\begin{align}
\label{moment_formula}
R_\vphi - \underline{R} - \Delta  f_\vphi   - \frac{\underline{R}}{n}\sum_\alpha  f_\alpha  \mathfrak{L}_{\xi_\alpha}(\vphi) - V_{f,\vphi}( f_\vphi ) - \sum_\alpha  f_\alpha  \mathfrak{L}_{\xi_\alpha}(\log((\ddc\vphi)^n)),
\end{align}
where $R_\vphi$ is the scalar curvature. 
\end{proposition}
\begin{remark}
We can compare the definition above with the moment map defined in \cite{Don97} for the case of K\"ahler Einstein (CscK).
In that case, the moment map is uniquely determined up to a constant.
For the definition \eqref{moment_formula}, $\bfm$ is uniquely determined up to a constant and the choice of the lifting of $\xi_\alpha$ ($\alpha = 1,\cdots,r$)
to the line bundle $L$.

When $L=-K_X$, let $h_\vphi = -\log(\frac{(\ddc\vphi)^n}{e^{-\vphi}})$, we can reformulate \eqref{moment_formula} into
\begin{equation}
\Delta(h_\vphi- f_\vphi ) + V_{f,\vphi}(h_\vphi- f_\vphi ) 
\end{equation}
\end{remark}
\begin{proof}
A direct calculation shows that:
$$(\sum_\alpha  f_\alpha  \mathfrak{L}_{\xi_\alpha}(\vphi))' = \sum_\alpha( f_\alpha  \theta_\alpha)' = 0, $$ 
$(\sum_\alpha  f_\alpha  \mathfrak{L}_{\xi_\alpha}\log((\ddc\vphi)^n))' = \sum_\alpha( f_\alpha  \Delta\theta_\alpha)'$.
We only need to show that the dual of the infinitesimal action is the infinitesimal moment map, which is
\begin{equation}
\bfm' = R_{\vphi}' - (V_{f,\vphi}( f_\vphi ))' - (\Delta f_\vphi )' - \sum_\alpha ( f_\alpha  \Delta \theta_\alpha)'.
\end{equation}
Since $V_{f,\vphi}( f_\vphi ) = \sum_{\alpha,\beta} f_\alpha  f_\beta \xi_\alpha(\theta_\beta) = \sum_{\alpha,\beta} f_\alpha  f_\beta (\ddc\vphi)^{i\bar{j}}\partial_i\theta_\alpha\partial_{\bar{j}}\theta_\beta = \sum_{\alpha,\beta} f_\alpha  f_\beta \omega(J V_{\theta_\alpha}, V_{\theta_\beta})$,
we have
\begin{equation}
(V_{f,\vphi}( f_\vphi ))' = \sum_{\alpha,\beta} f_\alpha  f_\beta \omega(\mu V_{\theta_\alpha}, V_{\theta_\beta}) = \sum_{\alpha,\beta} f_\alpha  f_\beta (\xi_\alpha\odot \xi_\beta, \mu) = (V_{f,\vphi}\odot V_{f,\vphi}, \mu).
\end{equation}

The infinitesimal action $P: Lie(\mcG) \rightarrow S^2(T^{1,0})$ can be furthermore decomposed into:
$P = P_2 \circ P_1$, 
\begin{align}
P_1: Lie(\mcG) \rightarrow \Gamma(T^{1,0}), \; v \rightarrow V_v^{1,0},
\end{align}
\begin{align}
P_2: T^{1,0} \rightarrow S^2(T^{1,0}),\;  v \rightarrow \mathfrak{L}_v J.
\end{align}

Now consider the dual of $P$. 
Let $\mu \in S^2(T^{1,0})\simeq \Omega^{0,1}(T^{1,0})$.

Then by \cite{Don97}[proof of Proposition 9] and an integration by parts calculation, 
\begin{align*}
(P(v),\mu)_{g_\vphi} &= (\mathfrak{L}_{V_v}J,\mu)   = -(\bar\partial V_v^{1,0}, \mu)_{g_\vphi} \\
&= \int_X  -(V^{1,0}_v, \bar\partial^*\mu) g_\vphi \omega^n + \int_X (V_v^{1,0}\odot V_{f,\vphi}, \mu) g_\vphi \omega^n \\
&= \int_X v R_\vphi' g_\vphi \omega^n +   + \int_X v \sum_{\alpha,\beta} f_{\alpha\beta}  (\xi_\alpha\odot \xi_\beta , \mu) g_\vphi \omega^n  + 2\int_X (\mu, V^{1,0}_v\odot V_{f,\vphi})g_\vphi \omega^n\\
&+\int_X v (V_{f,\vphi}\odot V_{f,\vphi}, \mu) g_\vphi \omega^n \\
\end{align*}
And
\begin{align*}
(-\int_X v \Delta f_\vphi  g_\vphi \omega^n)' &= \bigg( \int_X \Big( \omega(J dv, d f_\vphi ) + v \omega(J d f_\vphi ,d f_\vphi )\Big) g_\vphi \omega^n \bigg)'\\
&= \int_X (\mu, V_{f,\vphi}\odot V_v^{1,0}) g_\vphi \omega^n + \int_X v(\mu, V_{f,\vphi}\odot V_{f,\vphi}) g_\vphi \omega^n,
\end{align*}
\begin{align*}
(-\sum_\alpha \int_X v  f_\alpha  \Delta\theta_\alpha g_\vphi \omega^n)' &= \Big(\int_X \sum_{\alpha\beta} v  f_{\alpha\beta}  \omega(J d f_\beta, d f_\alpha ) g_\vphi \omega^n \\
&+ \sum_{\alpha,\beta} v f_\alpha  f_\beta \omega(J d f_\beta, d  f_\alpha ) g_\vphi \omega^n
+ \sum_\alpha  f_\vphi  \omega(J d f_\alpha , dv) g_\vphi \omega^n \Big)' \\
&= \sum_{\alpha,\beta} \int_X  f_{\alpha\beta}  (\mu, \xi_\alpha\odot \xi_\beta)g_\vphi\omega^n
+\int_X v (\mu, V_{f,\vphi}\odot V_{f,\vphi})g_\vphi \omega^n \\
&+ \int_X (\mu, V_{f,\vphi}\odot V^{1,0}_v) g_\vphi \omega^n
\end{align*}
By summing up the results above, we have
\begin{align*}
&(v, R_\vphi'-(V_{f,\vphi}( f_\vphi ))'-(\Delta f_\vphi )'-\sum_\alpha ( f_\alpha \Delta\theta_\alpha)')_{g_\vphi}\\ 
&= \int_X v \big( R_\vphi'-(V_{f,\vphi}( f_\vphi ))'-(\Delta f_\vphi )'\big) g_\vphi \omega^n = (\mu, \mathfrak{L}_{v}J)_{g_\vphi}.
\end{align*}
The proof is concluded.
Then the dual of $P$ is
$R_{f,\vphi}'-(V_{f,\vphi}( f_\vphi ))'$. This concludes the proof.

\end{proof}

We have the following formal picture. By Kempf-Ness, the stable points in GIT quotient $\mcG_\CC\sslash \mcG$ corresponds to the symplectic quotient
$\bfm^{-1}(0)/\mcG$. We should expect the solution to \eqref{MA_equation} is equivalent to a stability condition.
Meanwhile, the generalized Mabuchi functional can by considered as a Kempf-Ness functional (compatible with \eqref{Mabuchi-double-integral})
\begin{eqnarray*}
&&\bfM(\vphi) = \int_0^1dt \\
&&\hskip 5mm \int_X \dot{\vphi}\Big(-R_\vphi +\underline{R} +\Delta f_\vphi  +\frac{\underline{R}}{n}\sum_\alpha  f_\alpha  \mathfrak{L}_{\xi_\alpha}\vphi +V_{f,\vphi}( f_\vphi ) + \sum_\alpha  f_\alpha  \mathfrak{L}_{\xi_\alpha}(\log((\ddc\vphi)^n))\Big)g_\vphi \frac{(\ddc\vphi)^n}{n!} 
\end{eqnarray*}
Since the Kempf-Ness functional is ``convex''(the Hamiltonian of $\bfm$ can be considered as the K\"ahler potential of the
moduli space), we should expect the generalized Mabuchi-functional is convex along the weak geodesic.

When $L = -K_X$, let $h_\vphi = -\log((\ddc\vphi)^n)-\vphi$ be the Ricci potential.
Then ${\bf\mathfrak{m}} = \Delta (h_\vphi- f_\vphi ) + V_{f,\vphi}(h_\vphi- f_\vphi ) = 0$ implies $h_\vphi- f_\vphi  = 0$, which furthermore
implies the $g$-soliton equation $Ric_\vphi = \ddc\vphi +\ddc f_\vphi $.

\begin{remark}\label{rem-Mabuchi}
It seems to us that the formula for Mabuchi functionals used in \cite{His19} from
\cite[Definition 2.21, (2.47) ]{His19} is not correct,
since its variation does not give the correct soliton equation. 

Moreover, it appears that in \cite{BW14} to make some argument work, the Mabuchi functional from \cite[Page 23]{BW14} should be defined as $F_V(\MA_g(\vphi))$ instead of $F_V(\MA(\vphi))$.
\end{remark}

\section{Appendix: Non-Archimedean Entropy and generalized Mabuchi functional}
In this appendix, we will define the non-Archimedean entropy $\bfH^\NA_{g,\Theta}$ and generalized Mabuchi functional $\bfM^\NA$.
We will also show the slope formulas for them. For simplicity, we will only consider the case when $X$ is a smooth Fano manifold,
and $\Theta=0, D=0$.
In this case, $d\mu_0 = e^{-\vphi_0}$, and $\bfH_{g,\Theta}$ is reduced to $\bfH_g$. See \cite{WZZ16, Ino20} for related discussions.

\subsection{Polynomial $g$}
First, assume $g  = \sum_\vk a_\vk \prod_{\alpha=1}^r\theta_\alpha^{k_\alpha}$ is a polynomial, where $\vk$ is in a finite set.
\begin{align}
\begin{split}
\bfH_{g}(\vphi) &= \int_X \log(\frac{g_\vphi (\ddc\vphi)^n/n!}{d\mu_0}) g_\vphi  \frac{(\ddc\vphi)^n}{n!} \\
&= \sum_{\vk} a_{\vk} \bfH^{\bvk}(\vphi)
\end{split}
\end{align}
where
\begin{align}
\begin{split}
\bfH^{\bvk}(\vphi) &= \int_X \log(\frac{g_\vphi (\ddc\vphi)^n/n!}{d\mu_0}) \prod_\alpha \theta_\alpha^{k_\alpha} \frac{(\ddc\vphi)^n}{n!} \\
&= \int_{X^\bvk} \log(\frac{(\ddc\vphi)^n/n!}{d\mu_0}) \frac{((\ddc\vphi)^{\bvk})^{n+k}}{(n+k)!} + \int_{X^\bvk} \log(g_\vphi) \frac{((\ddc\vphi)^\bvk)^{n+k}}{(n+k)!}
\end{split}
\end{align}

Recall that, for a test configuration $(\mcX,\mcL)$ that dominates $X\times \CC^1$, for the entropy \\ $\int_X \log(\frac{(\ddc\vphi)^n/n!}{e^{-\vphi_0}})\frac{(\ddc\vphi)^n}{n!}$,
the corresponding non-Archimedean entropy is defined as \cite{BHJ17}
\begin{equation}
\int_{X^\div_\bQ} A_{X}(v) \MA^\NA(\phi) = \mcK^{\llg}_{\omcX/\PP^1}\cdot \omcL^{\cdot n} - \rho^*(K^\llg_{X_{\PP^1}/\PP^1})\cdot \omcL^{\cdot n}
\end{equation}
where $\rho: \mcX\dashrightarrow X\times\CC$, $(\overline{\mcX},\overline{\mcL})$ is the compactified test configuration.

Following the same idea, we give the non-Archimedean definition for the entropy functional in our setting.
\begin{definition}
\label{H_f^NA}
Let $(\mcX,\mcL)$ be a test configuration that dominates $X\times \CC^1$.
Define the non-Archimedean version entropies for $\bfH^{\bvk},\bfH_{g}$  as
\begin{align}
\label{bfH^NA}
\begin{split}
(\bfH^{\bvk})^\NA(\mcL^\bvk) 
&= (\mcK^{\llg}_{\omcX/\PP^1})^\bvk\cdot (\omcL^\bvk)^{\cdot (n+k)} - (\rho^\bvk)^*(K^{\llg}_{X_{\PP^1}/\PP^1})^\bvk \cdot (\omcL^{\bvk})^{\cdot (n+k)} \\
\bfH^\NA_g(\mcL) &= \sum_{\vk} a_{\vk} (\bfH^\bvk)^\NA(\mcL^\bvk)
\end{split}
\end{align}
\end{definition}
\begin{remark}
$K^\bvk_X$ and $K_{X^\bvk}$ are two different Cartier divisors in general.
$-K^\bvk_X$ is ample under our construction. However, $X^\bvk$ is in general not a Fano variety.
\end{remark}

Recall the generalized Mabuchi functional
\begin{align}
\begin{split}
\bfM(\phi) &= \bfH_{g}(\vphi) + \int_X (\vphi-\vphi_0)g_\vphi \frac{(\ddc\vphi)^n}{n!} - \bfE_g(\vphi) 
\end{split}
\end{align}
and
\begin{align*}
\int_X(\vphi - \vphi_0) g_\vphi \frac{(\ddc\vphi)^n}{n!} &= \sum_{\vk}a_{\vk}\int_X (\vphi-\vphi_0)\prod_\alpha \theta_\alpha^{k_\alpha} \frac{(\ddc\vphi)^n}{n!}\\
&= \sum_{\vk}a_{\vk} \int_{X^\bvk} (\vphi^\bvk-\vphi_0^\bvk) \frac{((\ddc\vphi)^\bvk)^{n+k}}{(n+k)!}
\end{align*}
and
\begin{align*}
&\int_{X^\bvk} \log(\frac{(\ddc\vphi)^n}{e^{-\vphi_0}}) \frac{((\ddc\vphi)^\bvk)^{n+k}}{(n+k)!} + \int_{X^\bvk} (\vphi^\bvk-\vphi_0^\bvk) \frac{(\ddc\vphi)^{n+k}}{(n+k)!}
- \bfE^\bvk(\phi^\bvk) =\\
& \la(\log((\ddc\vphi)^n))^\bvk,\vphi^\bvk,\cdots,\vphi^\bvk\ra_{X^\bvk} - \la\vphi_0^\bvk,\cdots,\vphi_0^\bvk\ra_{X^\bvk} \\
& +(\la\vphi^\bvk,\cdots,\vphi^\bvk\ra_{X^\bvk}-\la\vphi_0^\bvk,\cdots,\vphi_0^\bvk\ra_{X^\bvk}) - \bfE^\bvk(\vphi^\bvk)
\end{align*}
where $\la\cdots\ra_{X^\bvk}$ denotes the metric on the Deligne pairing.
Since we have chosen $\vphi_0$ as the reference metric, $\la \vphi_0^\bvk,\cdots,\vphi_0^\bvk \ra_{X^\bvk} = 0$.
The calculation above inspires the following definition.
\begin{definition}
Let $(\mcX,\mcL)$ be a test configuration.
We define the non-Archimedean version of the generalized Mabuchi functional as
\begin{equation}
(\bfM^{\bvk})^\NA(\mcL^\bvk) = \bfH^\bvk(\mcL^\bvk) - \bfI^\bvk(\mcL^\bvk)+\bfJ^\bvk(\mcL^\bvk)
\end{equation}
which is equivalent to 
\begin{align}
\label{M_k}
\begin{split}
(\bfM^{\bvk})^\NA(\mcL^\bvk) &= \bfH^\bvk(\mcL^\bvk) +  (D^\bvk \cdot (\mcL^\bvk)^{\cdot n+k} 
- \bfE^{\bvk}(\mcL^\bvk)\\
&= \bfH^\bvk(\mcL^\bvk) + (\omcL^\bvk)^{\cdot n+k+1} + (\rho^\bvk)^*K_X^\bvk\cdot (\omcL^\bvk)^{\cdot n+k}
- \bfE^{\bvk}(\mcL^\bvk)\\
&= (\mcK^\llg_{\omcX/\PP^1})^\bvk\cdot(\omcL^\bvk)^{\cdot n+k} + (n+k)\bfE^\bvk(\omcL^\bvk) 
\end{split}
\end{align}
\begin{align}
\label{bfM_f^NA}
\bfM^\NA_g (\mcL) &= \sum_{\vk} a_\vk (\bfM^\bvk)^\NA(\mcL^\bvk)
\end{align}
\end{definition}

\subsection{continuous $g$}
In this subsection, we will first show several estimates that will be used later, then define $\bfH_{g}^\NA,\bfM^\NA$ for continuous $g$.
At last, we will prove the corresponding slope formulas.
Let $g_i$ be a sequence of polynomials that converges to $g$ in $C^0(P)$-norm.
Let $(\mcX,\mcL)$ be an ample normal test configuration
that dominates $X\times \CC$. 
Let $e^{\Psi_{\reff}}$ be a smooth metric on $\mcK^\llg_{\mcX/\CC}$. This auxilliary metric is introduced in \cite{BHJ17} in the proof of the slope formulas.
We will use it in the proof of the estimates below.

\begin{lemma}
\label{psi_reff}
There exists a $C>0$, such that 
\begin{equation}
|\int_X \log(\frac{(\ddc\vphi)^n/n!}{e^{\psi_{\reff}}}) g_{i\vphi} \frac{(\ddc\vphi)^n}{n!}| < C (\log|t|+1)
\end{equation}
\end{lemma}
\begin{proof}
The constant $C$ may change from line by line in the proof below.
By the proof of \cite[Lemma 3.10]{BHJ17}, there exists a $C>0$, such that
$\log(\frac{(\ddc\vphi)^n/n!}{e^{\psi_\reff}})<C$ uniformly for all $t$. Then
\begin{align*}
&|\int_X \log(\frac{(\ddc\vphi)^n/n!}{e^{\psi_{\reff}}}) g_{i\vphi} \frac{(\ddc\vphi)^n}{n!}| \\
&\leq |\int_X (\log(\frac{(\ddc\vphi)^n/n!}{e^{\psi_{\reff}}}) -C) g_{i\vphi} \frac{(\ddc\vphi)^n}{n!}| + C \int_X g_{i\vphi}\frac{(\ddc\vphi)^n}{n!} \\
&\leq \sup_P(g_i) |\int_X (\log(\frac{(\ddc\vphi)^n/n!}{e^{\psi_{\reff}}}) -C)\frac{(\ddc\vphi)^n}{n!}| + C \\
&\leq C( |\int_X \log(\frac{(\ddc\vphi)^n/n!}{e^{\psi_{\reff}}}) \frac{(\ddc\vphi)^n}{n!}| + 1) \\
&< C(\log|t|+1)
\end{align*}
The last inequality is by \cite[Lemma 3.10]{BHJ17}.
\end{proof}

\begin{lemma}
\label{g_i}
For any $\epsilon>0$, there exists an $i_0>0$, such that for any $i,j>i_0$,
\begin{equation}
|\int_X \log(\frac{e^{\psi_\reff}}{d\mu_0})(g_{i\vphi}-g_{j\vphi}) \frac{(\ddc\vphi)^n}{n!}| < \epsilon (t+1)
\end{equation}
\end{lemma}
\begin{proof}
Since $g_i$ converges to $g$ in $C^0(P)$, for any $\epsilon'>0$, there exists an $i_0$, such that for any $i,j>i_0$, $|g_i-g_j|<\epsilon'$.
Since we restrict to the case that $X$ is a Fano manifold, $d\mu_0 = e^{-h_0}\frac{\ddc\vphi_0^n}{n!}$ is smooth and non-degenerate.
Since $e^{\psi_\reff}$ is a smooth metric on $\mcK^\llg_{\omcX/\PP^1}$, there exists a $C>0$ such that $\log(\frac{\psi_\reff}{d\mu_0})<C$ uniformly.
Then
\begin{align*}
&|\int_X \log(\frac{e^{\psi_\reff}}{d\mu_0})(g_{i\vphi}-g_{j\vphi}) \frac{(\ddc\vphi)^n}{n!}|\\ 
&\leq |\int_X (\log(\frac{e^{\psi_\reff}}{d\mu_0})-C)(g_{i\vphi}-g_{j\vphi}) \frac{(\ddc\vphi)^n}{n!}| + C\epsilon'\\ 
&\leq C \epsilon' (|\int_X \log(\frac{e^{\psi_\reff}}{d\mu_0}) \frac{(\ddc\vphi)^n}{n!}|+1)\\
&\leq C \epsilon' (C't + 1)
\end{align*}
where the last line is by \cite[Lemma 3.9]{BHJ17}.
\end{proof}

\begin{lemma}
\label{H_NA_infty}
$(\bfH^\bvk)^{'\infty}(\Phi) = (\bfH^\bvk)^\NA(\phi)$,
$(\bfM^\bvk)^{'\infty}(\Phi) = (\bfM^\bvk)^\NA(\phi)$.
\end{lemma}
\begin{proof}
Since $e^{\psi_\reff},e^{-\vphi_0}$ are smooth metrics on $\mcK^\llg_{\omcX/\PP^1}, (\rho)^*K^\llg_{X_{\PP^1}/\PP^1}$, by \cite[Lemma 3.9]{BHJ17}, we have
\begin{equation}
\la \psi^\bvk_\reff,\vphi^\bvk,\cdots,\vphi^\bvk \ra - \la \psi^\bvk_{\reff,0},\vphi^\bvk_0,\cdots,\vphi^\bvk_0 \ra = t (\mcK^\llg_{\omcX/\PP^1})^\bvk \cdot (\mcL^\bvk)^{\cdot (n+k)} + O(1)
\end{equation}
\begin{equation}
\la \vphi^\bvk_0,\vphi^\bvk,\cdots,\vphi^\bvk \ra - \la \vphi^\bvk_0,\vphi^\bvk_0,\cdots,\vphi^\bvk_0 \ra = t (\rho^\bvk)^*(K^\llg_{X_{\PP^1}/\PP^1})^\bvk \cdot (\mcL^\bvk)^{\cdot (n+k)} + O(1)
\end{equation}
Then
\begin{align*}
& (\bfH^\bvk)^{'\infty}(\Phi) = \lim_{t\to\infty} \frac{1}{t}\int_{X^\bvk} \log(\frac{(\ddc\vphi)^n/n!}{e^{-\vphi_0}})\frac{(\ddc\vphi^\bvk)^{n+k}}{(n+k)!}\\
&= \lim_{t\to\infty}\frac{1}{t}\big( \int_{X^\bvk} \log(\frac{(\ddc\vphi)^n/n!}{e^{\psi_\reff}})\frac{(\ddc\vphi^\bvk)^{n+k}}{(n+k)!} 
+ \int_{X^\bvk} \log(\frac{e^{\psi_\reff}}{e^{\vphi_0}})\frac{(\ddc\vphi^\bvk)^{n+k}}{(n+k)!} \big)\\
&=
\lim_{t\to\infty}\frac{1}{t}\big( \la \psi^\bvk_\reff,\vphi^\bvk,\cdots,\vphi^\bvk \ra - \la \psi^\bvk_{\reff,0},\vphi^\bvk_0,\cdots,\vphi^\bvk_0 \ra -
(\la \vphi^\bvk_0,\vphi^\bvk,\cdots,\vphi^\bvk \ra - \la \vphi^\bvk_0,\vphi^\bvk_0,\cdots,\vphi^\bvk_0 \ra) \big) \\
&=(\mcK^\llg_{\omcX/\PP^1})^\bvk \cdot (\mcL^\bvk)^{\cdot (n+k)} - (\rho^\bvk)^*(K^\llg_{X_{\PP^1}/\PP^1})^\bvk \cdot (\mcL^\bvk)^{\cdot (n+k)}
\end{align*}
where the equality of the second and third line is because of Lemma \ref{psi_reff}. 
The proof for $(\bfM^\bvk)^\NA$ follows similarly.
\end{proof}

\begin{definition}
For a continuous function $g$, we define the non-Archimedean entropy and generalized Mabuchi functional as
\begin{equation}
\label{Hg_NA}
\bfH^\NA_g(\mcL) = \lim_{i\to\infty} \bfH^\NA_{g_i}(\mcL)
\end{equation}
\begin{equation}
\label{Mg_NA}
\bfM^\NA(\mcL) = \lim_{i\to\infty} \bfM^\NA_{g_i}(\mcL)
\end{equation}
where $\bfM^\NA_{g_i}$ is the generalized Mabuchi functional with respect to $g_i$.
\end{definition}

\begin{proposition}
The limits in \eqref{Hg_NA},\eqref{Mg_NA} converge. And we have the slope formulas
\begin{equation}
\bfH_{g}^{'\infty}(\Phi) = \bfH_{g}^\NA(\phi)
\end{equation}
\begin{equation}
\bfM^{'\infty}(\Phi) = \bfM^\NA(\phi)
\end{equation}
\end{proposition}
\begin{proof}
By Lemma \ref{H_NA_infty}, for each polynomial $g_i$, $\bfH_{g_i}^{'\infty}(\Phi) = \bfH_{g_i}^\NA(\phi)$.
Let $i_0$ be sufficiently large, and $i,j>i_0$.
By Lemma \ref{psi_reff}, Lemma \ref{g_i}, 
\begin{align*}
& |\bfH_{g_i}(\vphi)-\bfH_{g_j}(\vphi)| \leq \\
&|\int_X \log(\frac{(\ddc\vphi)^n/n!}{e^{\psi_\reff}})g_{i\vphi}\frac{(\ddc\vphi)^n}{n!}|+
|\int_X \log(\frac{(\ddc\vphi)^n/n!}{e^{\psi_\reff}})g_{j\vphi}\frac{(\ddc\vphi)^n}{n!}|\\
&+|\int_X \log(\frac{e^{\psi_\reff}}{d\mu_0})(g_{i\vphi}-g_{j\vphi})\frac{(\ddc\vphi)^n}{n!}|\leq C(\epsilon t + \log(t) + 1)
\end{align*}
Then
$\frac{1}{t}\bfH_{g_i}(\vphi)$ converges to $\bfH_{g_i}^{'\infty}(\Phi) = \bfH_{g_i}^\NA(\phi)$ uniformly.
Specifically, $|\bfH_{g_i}^\NA(\phi)-\bfH_{g_j}^\NA(\phi)|<C\epsilon$.
Then $\bfH_{g_i}^\NA(\phi)$ is a Cauchy sequence, which converges to a limit $\bfH_{g}^\NA(\phi)$.
And by the dominated convergence theorem, $\bfH_{g}^{'\infty}(\Phi) = \bfH_{g}^\NA(\phi)$.

The statement for $\bfM^\NA$ can be proved similarly.
\end{proof}


\end{document}